\documentclass{amsart}
\DeclareRobustCommand{\SkipTocEntry}[5]{} 
\usepackage{pdfsync} 
\usepackage{cancel}
\usepackage{ulem}
\usepackage{leftindex}
\usepackage{amsfonts,amsmath,amsthm,mathtools}
\newtheorem{theorem}{Theorem}
\numberwithin{theorem}{section}
\newtheorem{lemma}[theorem]{Lemma}
\newtheorem{proposition}[theorem]{Proposition}

\newtheorem{corollary}[theorem]{Corollary}
\theoremstyle{definition}

\newtheorem{remark}[theorem]{Remark}
\numberwithin{equation}{section}
\usepackage{amssymb}
\usepackage[colorinlistoftodos,prependcaption,textsize=tiny]{todonotes}
\usepackage[colorlinks,linkcolor={blue},citecolor={blue},urlcolor={purple},]{hyperref}\usepackage{graphicx}
\usepackage{csquotes}
\usepackage{wrapfig}
\usepackage{subfig}
\usepackage{epsfig}
\usepackage{float,mathrsfs}
\usepackage[shortlabels]{enumitem}
\usepackage{bbm}
\usepackage{todonotes}
\usepackage[markup=underlined]{changes}
\definechangesauthor[color=red]{Foivos}

\usepackage{algpseudocode}
\usepackage[font=scriptsize]{caption}
\counterwithout{figure}{section}
\usepackage{color}
\definecolor{refkey}{rgb}{0.9451,0.2706,0.4941}\definecolor{labelkey}{rgb}{0.9451,0.2706,0.4941}

\definecolor{darkred}{RGB}{139,0,0}
\definecolor{darkgreen}{RGB}{0,100,0}
\definecolor{darkmagenta}{RGB}{139,0,139}
\definecolor{gray}{RGB}{180,180,180}

\renewcommand{\epsilon}{\varepsilon}

\newcommand{\rd}{{\mathrm{d}}}

\newcommand{\bbN}{{\mathbb{N}}}
\newcommand{\bbR}{{\mathbb{R}}}

\newcommand{\Heff}{H_{\mathrm{eff}}}

\newcommand{\bbE}{\mathbb{E}}

\newcommand{\mask}[1]{{}}

\newtheorem{assumption}{Assumption}

\usepackage{pgfplots}
\pgfplotsset{compat=1.18}

\title{The stochastic
Landau--Lifshitz--Baryakhtar equation in critical spaces}

\author{F. Butori}
\address{Scuola Normale Superiore, Pisa, Italy}
\email{federico.butori@sns.it}

\author{F. Evangelopoulos-Ntemiris}
\address{Delft University of Technology, Delft, Netherlands}
\email{f.a.evangelopoulos-ntemiris@tudelft.nl}

\author{\\L. Marino}
\address{Scuola Normale Superiore, Pisa, Italy}
\email{lorenzo.marino@sns.it}

\author{M. Orteu-Capdevila}
\address{Freie Universit\"at Berlin, Berlin, Germany}
\email{m.orteu.capdevila@fu-berlin.de}

\allowdisplaybreaks

\begin{document}

\thanks{The second author has received funding from the VICI subsidy VI.C.212.027 of the Dutch Research Council (NWO). The fourth author acknowledges funding by the
Deutsche Forschungsgemeinschaft (DFG, German Research Foundation) – CRC/TRR 388
``Rough Analysis, Stochastic Dynamics and Related Fields" – Project ID 516748464}
\begin{abstract}
We study the stochastic
Landau--Lifshitz--Baryakhtar equation with multiplicative noise
and homogeneous Neumann boundary conditions on smooth bounded
domains in dimensions $d\le3$. Relying on stochastic maximal regularity results, we establish global well-posedness in a variational strong, intermediate, and weak setting, as well as for initial
data in the scaling-critical Besov spaces $B^{d/q-1}_{q,p}$
for every $q\in[2,\infty)$ and $p\in(2,\infty)$. 
\end{abstract}

\keywords{Stochastic Landau--Lifshitz--Baryakhtar equation, nonlinear SPDEs, global well-posedness, critical spaces, stochastic maximal regularity, energy estimates, rough initial data, multiplicative noise}
\subjclass[2020]{Primary: 60H15; Secondary: 35B65, 35K91, 35K35}

\maketitle
\setcounter{tocdepth}{2}
	\tableofcontents

\section{Introduction}
The Landau--Lifshitz--Baryakhtar equation (LLBar) \cite{Bar84, Bar89} is a continuum model used to describe the dynamics of magnetization in ferromagnetic materials. It generalizes both the Landau--Lifshitz--Gilbert equation (LLG) \cite{Gi55} and the Landau--Lifshitz--Bloch equation (LLB) \cite{Ga97}. The LLG equation assumes that the magnitude of the magnetization vector remains constant, a reasonable assumption only if the material is at very low temperatures. In contrast, the LLB equation allows the magnetization magnitude to vary, extending its applicability to higher-temperature regimes. The LLBar model further incorporates a nonlocal (exchange) damping contribution, commonly known as the Baryakhtar term. This additional term significantly improves the description of short-wavelength spin-wave propagation and lifetimes. In fact, it is the only model that aligns with certain experimental observations in micromagnetics concerning ultrafast magnetization dynamics at elevated temperatures. The existence, uniqueness, and regularity of solutions to the LLBar equation have been studied in \cite{SoeTran23}. 

From a physical perspective, in many practical applications the effective magnetic fields are inevitably influenced by environmental randomness, including thermal fluctuations, magnetic field variations, and external noise sources. These factors introduce stochasticity into the effective field and consequently affect its evolution. Motivated by this, it appears necessary to extend the LLBar model in order to incorporate random fluctuations of the effective field into the magnetization dynamics, yielding a stochastic version of the LLBar model. Such an approach has previously been considered in \cite{GoSoTr25, XuLiuZhang26, XuLiuZhang26_1}.

We can now present our model in more detail. Let $\mathscr D \subset \mathbb{R}^d$, $d\in \{1,2,3\}$, be a bounded, sufficiently smooth domain. The stochastic Landau--Lifshitz--Baryakhtar (sLLBar) equation describes the magnetization $u$ of the body $\mathscr D$ as:
\begin{equation}  \label{eq: LLB}      \tag{sLLBar}
\begin{aligned}
\begin{cases}
{\rm d}u = (\lambda_r \Heff(u) - \lambda_e\Delta  \Heff(u) - \gamma u\times  \Heff(u)) \, {\rm d}t + \sum_{k = 1}^\infty g_k(t,\omega, \cdot, u,\nabla u) \, {\rm d}W^k(t) \, \text{ in } \mathscr D, \\ 
        \partial_n u = 0, \partial_n \Heff(u)  = 0 \, \text{ on } \partial \mathscr D, \\
        u(0, \cdot) = u_0, \, \text{ in } \mathscr D.
    \end{cases}    
\end{aligned}
\end{equation}

Here, $ \Heff (u) \coloneq  \alpha\Delta u + \kappa_1u - \kappa_2|u|^2 u$ is the effective magnetic field, which is  derived from the variational derivative of the system's free energy; $- \gamma u\times  \Heff(u)$ describes the standard Larmor precession with electron gyromagnetic ratio $\gamma>0$; $\lambda_r \Heff(u)$ is the local (relativistic) longitudinal damping term; and $\lambda_e\Delta  \Heff(u)$ is usually called the Baryakhtar exchange damping term. Both parameters $\lambda_r$ and $\lambda_e$ are positive. Above, $\alpha>0$ represents the exchange stiffness, while $\kappa_1\in \mathbb{R}$ and $\kappa_2>0$ are phenomenological coefficients tied to the longitudinal magnetic susceptibility and the equilibrium magnetization magnitude. The phase transition of the material is dictated by the sign of the parameter $\kappa_1$, which changes depending on the temperature $T$ of the material, relative to the Curie temperature $T_C$. In the ferromagnetic state (when $T<T_C$), the material possesses a spontaneous macroscopic magnetization and $\kappa_1$ is positive. On the other hand, in the paramagnetic state (when $T>T_C$), thermal fluctuations completely overpower the exchange interactions, destroying the macroscopic magnetic ordering with the consequence that $\kappa_1$ becomes negative.

Let $(\Omega,\mathcal F,(\mathcal F_t)_{t\ge0},
\mathbb P)$ be a complete filtered probability space satisfying the usual
conditions. We assume that the family $\{W^k\}_{k\ge1}$ in \eqref{eq: LLB} is a sequence of independent standard real-valued $(\mathcal F_t)$-Wiener processes. The noise coefficient is
described by a progressively measurable map
\[
g\colon
\mathbb R_+\times\Omega\times\mathscr D
\times\mathbb R^3\times\mathbb R^{3\times d}
\longrightarrow
\ell^2(\mathbb N;\mathbb R^3),\]
subject to the setting-dependent regularity, Lipschitz, and integrability
conditions specified below in Section \ref{section:main results}.

Taking into account the definition of $ \Heff$, \eqref{eq: LLB} can be written as an evolution equation of the form 
\begin{equation}
     \label{eq: LLB1}
     \begin{cases}
    {\rm d}u + Au \, {\rm d} t 
        = F(u) \, {\rm d}t  + \sum_{k = 1}^\infty g_k(u) \, {\rm d}W^k(t), \quad t>0
     \\
     u(0)=u_0
     \end{cases}  
\end{equation}
where the linear operator $A$ and the nonlinearity $F$ are given by
\begin{equation}\label{eq:def of A and F}
    Au \coloneq \alpha_1\Delta^2 u +\alpha_2\Delta u +\alpha_3 u, \qquad F(u) \coloneq \gamma_1\Delta(|u|^2 u)+ \gamma_2\Delta u \times u-\gamma_3 |u|^2 u,
\end{equation}
with  $\alpha_1, \gamma_1,\gamma_2,\gamma_3>0$ positive constants and  $\alpha_2, \alpha_3 \in \mathbb{R}$.  

\addtocontents{toc}{\SkipTocEntry}
\subsection{Overview of previous results in the literature}

There is a handful of existing results in the literature addressing the local and global well-posedness of \eqref{eq: LLB}, most of them based on a Faedo--Galerkin approximation combined with tightness arguments. In this section, we provide a short overview of what has been shown so far.

The well-posedness and the existence of an invariant measure of \eqref{eq: LLB1} were investigated in \cite{GoSoTr25}, where the authors proved existence of global analytically strong solutions with paths in $C([0, T]; H^2)\cap L^2(0, T; H^4)$ in dimension $d=1,2,3$. 

In \cite{XuLiuZhang26_1}, the authors showed existence and uniqueness of a local maximal pathwise weak solution with paths in $L^{\infty}(0,T; H^1) \cap L^2(0,T; H^3)$ in $d=1,2,3$, as well as existence and uniqueness of global pathwise analytically weak solutions in $C([0, T]; H^1)\cap L^2(0, T; H^3) $ in $d=1$. Additionally, they prove existence and uniqueness of global pathwise very weak solutions, meaning analytically weak solutions with paths in $C([0, T]; L^2)\cap L^2(0, T; H^2)$, but only in $d=1,2$. For $d=3$, uniqueness remains open and the authors showed only existence of global martingale very weak solutions. 

With this work, we unify all the preceding cases in a single framework and extend some of the results to more general noise and/or higher dimension. We note that the setting in \cite{GoSoTr25} corresponds to our strong setting, while the path regularity of the weak and very weak solutions in \cite{XuLiuZhang26_1} corresponds to that of our intermediate and weak ones, respectively -- and which in our case are analytically strong. On top of that, thanks to the stochastic maximal regularity and $L^p(L^q)$-theory we are able to derive global existence and uniqueness also in the case of rough initial data coming from the scaling-critical trace space $B^{d/q-1}_{q,p}$.

We note that there has also been some work on the well-posedness of the stochastic LLBar equation with L\'evy noise, which we do not consider here. Indeed, in \cite{XuLiuZhang26} the authors prove existence and uniqueness of a probabilistically strong solution with paths in $\mathbb D([0,T]; H^1)\cap L^2(0,T;H^3)$ for dimensions $d = 1,2,3$, and establish a Freidlin--Wentzell type large deviations principle.

\addtocontents{toc}{\SkipTocEntry}
\subsection{Outline of the paper}

Section~\ref{Sec: Preliminaries} introduces the function spaces
associated with the Neumann realization of the fourth-order operator,
establishes its stochastic maximal regularity, and explains the scaling that leads to the critical Besov spaces 
$B^{d/q-1}_{q,p}$. 

Section~\ref{section:main results} states our global well-posedness
results under the corresponding assumptions on the noise. We first
consider three variational settings: the weak setting with initial
data in $L^2$ for $d\le2$, and the intermediate and strong settings
with initial data in $H^1$ and ${}_\nu H^2$, respectively, for
$d\le3$. We then treat initial data in
$B^{d/q-1}_{q,p}$ in $d\le 3$, for every $q\in[2,\infty)$ and $p\in(2,\infty)$,
as well as the endpoint $p=q=2$. These critical spaces include
distributions of negative regularity when $q>d$. The resulting
solutions regularize instantaneously at positive times.

Section \ref{section:proofs} contains the proofs. Local well-posedness follows from stochastic maximal regularity techniques and suitable local Lipschitz estimates for the drift and noise coefficients. 
In the weak setting, a coercivity estimate yields global
well-posedness directly. In the intermediate and strong settings, global well-posedness requires additional energy estimates. Finally, for initial data in the critical space $B^{d/q-1}_{q,p}$, we construct local
solutions by using stochastic maximal $L^p$-regularity techniques. 
Instantaneous regularization and uniqueness allow us to identify
the resulting local solutions, at positive times, with the global
solutions in the intermediate setting. The bounds for these
$H^1$ solutions then rule out finite-time blow-up and yield
global well-posedness.

\section{Preliminaries}
\label{Sec: Preliminaries}
\addtocontents{toc}{\SkipTocEntry}
\subsection{Notation}
 Throughout this paper, $C$ denotes a generic positive constant that may change from line to line. Given a set $D$ and functions $a,b\!: D \to \mathbb R$, we will use the standard notation $a\lesssim b$ in case there is a generic constant $C>0$ such that $a(x)\leq Cb(x)$ for all $x \in D$. Moreover, we write $a\eqsim b$ if $a\lesssim b$ and $b \lesssim a$. To stress the dependence of the general constant $C$ in some other variable, e.g. $T$, we will write $\lesssim_T$. For matrices $M,N\in\mathbb R^{3\times d}$, we write $
    M:N \coloneq \sum_{j=1}^d \sum_{i=1}^3 M_{i j}\cdot N_{i j}
    $
for the Frobenius product.

Let $E$ be a Banach space.
In the interest of brevity, the domain $\mathscr D \subset \bbR^d$ is omitted from the notation of function spaces when no confusion is likely to arise; e.g., $L^q(E)$ refers to $L^q(\mathscr D;E)$. Moreover, we let $\ell^2(E)\coloneq \ell^2(\bbN;E)$. In case $E=\bbR^3$, we simply write $L^q$ instead of $L^q(\bbR^3)$ and $\ell^2$ instead of $\ell^2(\bbR^3)$.

For a Banach couple $(X_0, X_1)$, we use 
$$ X_{\alpha,p} \coloneq (X_0, X_1)_{\alpha,p} \ \ \text{and} \ \ X_{\alpha} \coloneq [X_{0}, X_1]_{\alpha} , \ \ \ \alpha\in (0,1), \ p\in [1,\infty],$$
for the real and complex interpolation spaces, respectively. Moreover, for a Gelfand triple $\mathcal V \hookrightarrow \mathcal H \hookrightarrow \mathcal V^*$, we write 
$$ \mathcal V_{\alpha,p} \coloneq (\mathcal V^*,\mathcal V)_{\alpha,p}, \quad \mathcal V_{\alpha} \coloneq [\mathcal V^*,\mathcal V]_{\alpha} .$$

\addtocontents{toc}{\SkipTocEntry}
\subsection{Function spaces}
Let $A=\alpha_1\Delta^2+\alpha_2\Delta+\alpha_3$ be as in \eqref{eq:def of A and F}, where $\alpha_1>0$ and $\alpha_2,\alpha_3 \in \bbR$. Let $ q\in (1,\infty)$ and let 
$$A_q \colon D(A_q) \subseteq L^q \to L^q, \quad A_q(u)=Au$$
denote the $L^q$-realization of $A$ with domain
$$D(A_q) \coloneq \{u \in H^{4,q} \colon \partial_n u = \partial_n \Delta u =0\}.$$
By \cite[Theorem~2.3]{new_thoughts}, there exists $\lambda\geq0$ sufficiently large so that $\lambda+A_q$ has a bounded $H^\infty$-calculus of angle less than $\pi/2$. Hence, the interpolation-extrapolation scale generated by $A_q$ is well defined (cf.\,\cite[Appendix~A.1]{AgVe25} or \cite[Chapter V]{Amann}), which we denote by 
$$ ( {}_\nu  H^{s,q},A_{s,q})_{s \ge -4} .$$
The realization of $A$ on this scale can be defined by 
$$A_{s,q} \coloneq 
\begin{cases}
     {}_\nu  H^{s,q}-\text{realization of }A_q, &s\ge 0
    \\
    \text{closure of $A_q$ in $ {}_\nu H^{s,q}$}, & s \in [-4,0).
\end{cases}$$
In particular, $A_{0,q}=A_q$. Moreover, for $s<0$, one has  
$
     {}_\nu H^{s,q}
=
    (  {}_\nu  H^{-s,q'})^*$, where $\frac1q+\frac1{q'}=1 $; 
see \cite[Theorem~A.3]{AgVe25} for instance. One can show that $A_q$ and $A_{s,q}$ are similar, and thus, after a sufficiently large shift, the operator 
$$ A_{s,q} \colon D(A_{s,q}) \subseteq  {}_\nu  H^{s,q} \to  {}_\nu H^{s,q} , \quad D(A_{s,q})= {}_\nu  H^{s+4,q} $$
    has a bounded $H^\infty$-calculus of angle less than $\pi/2$.

For the notion of \textit{stochastic maximal $L^p_\kappa$-regularity}, $\mathcal {SMR}^\bullet_{p,\kappa}$ for short, the reader is referred  to \cite[Section 3]{AgVe25}.

\begin{lemma}[Stochastic maximal regularity]
\label{lemma:SMR}
Let $q\in[2,\infty)$, $s\ge -4$, and let $A_{s,q}$ denote the realization of $A$ on $ {}_\nu H^{s,q}$. Then,
$$
A_{s,q}\in\mathcal{SMR}^{\bullet}_{p,\kappa}, \quad \text{for every} \quad p \in (2,\infty), \, \kappa \in [0,p/2-1).
$$
Moreover, if $q=2$ then $A_{s,2}\in\mathcal{SMR}^{\bullet}_{2,0}$.
\end{lemma}
\begin{proof}
    By the above, there is $\lambda_0\ge0$ large enough so that $\lambda_0+A_{s,q}$ has a bounded $H^\infty$-calculus of angle less than $\pi/2$. Therefore,  \cite[Theorem~3.14]{AgVe25} implies the claim.
\end{proof}

 We note that the family of operators $(A_{s,q})$ is consistent, in the sense that for all $1<q_0,q_1<\infty$ and $s_0,s_1\ge -4$, 
 $$ A_{s_0,q_0} u = A_{s_1,q_1} u, \quad u \in  {}_\nu  H^{s_0+4,q_0} \cap  {}_\nu  H^{s_1+4,q_1}.$$
 We now present some standard results on interpolation with boundary conditions (see e.g.\,\cite{Seeley} or \cite{Guidetti}). For $s_0,s_1 \ge -4 , \, q \in (1,\infty)$ and $\theta \in (0,1)$,
 $$ [ {}_\nu H^{s_0,q},  {}_\nu H^{s_1,q}]_{\theta} =  {}_\nu H^{(1-\theta) s_0 +\theta s_1,q} .$$
Moreover, the following characterization holds
$$
      {}_\nu  H^{s,q}
    =
    \begin{cases}
        H^{s,q},
        & 0 \le s<1+\tfrac1q,
        \\[2mm]
        \left\{
            u\in H^{s,q}:
            \partial_n u=0
            \text{ on }\partial\mathscr D
        \right\},
        & 1+\tfrac1q<s<3+\tfrac 1q,
        \\[2mm]
        \{ u \in H^{s,q} \colon \partial_n u = \partial_n \Delta u = 0 \text{ on }\partial\mathscr D\}, & 3+\tfrac 1q<s\le 4.
    \end{cases}
$$
We define the scale of Besov spaces 
\[
     {}_\nu  B_{q,p}^{s}
     \coloneq 
    \bigl(
          {}_\nu  H^{s_0,q},
          {}_\nu  H^{s_1,q}
    \bigr)_{\theta,p},
\]
where $-4\le s_0<s_1\le 4$, $\theta \in (0,1)$ and $s=(1-\theta)s_0+\theta s_1$. Then, one has 
\[
     {}_\nu  B_{q,p}^{s}
    =
    \begin{cases}
        B_{q,p}^{s}(\mathscr D),
        & 0 \le s<1+\tfrac1q,\\[2mm]
        \{
            u\in B_{q,p}^{s}(\mathscr D):
            \partial_n u=0
            \text{ on }\partial\mathscr D
        \},
        & 1+\tfrac1q < s< 3+\tfrac 1q
        \\[2mm]
        \{
            u\in B_{q,p}^{s}(\mathscr D):
            \partial_n u=\partial_n \Delta u = 0
            \text{ on }\partial\mathscr D
        \}, & 3+\tfrac 1q<s<4.
    \end{cases}
    \]
For notational convenience, we omit the subscript $\nu$ in the norms
of these spaces, writing $\|\cdot\|_{H^{s,q}}$ and
$\|\cdot\|_{B^s_{q,p}}$ for the norms of ${}_\nu H^{s,q}$ and
${}_\nu B^s_{q,p}$, respectively.

\addtocontents{toc}{\SkipTocEntry}
\subsection{Scaling and criticality} Before stating our main results, we present a scaling argument for \eqref{eq: LLB1} on $\bbR^d$. 
Following the literature (see, e.g., \cite{AgVe25}), we say that a Banach space is \textit{critical} for \eqref{eq: LLB} if functions in this Banach space are invariant under the following map for the initial data
$$u_0 \mapsto u_{0,\lambda} \coloneq  \lambda^{1/4} u_0(\lambda^{1/4} \cdot ).$$
We have 
$$\|u_\lambda (t) \|_{L^q(\bbR^d;\bbR^3)} = \lambda^{\tfrac 14-\tfrac d{4q}} \|u(\lambda t)\|_{L^q(\bbR^d;\bbR^3)}, \qquad \|u_\lambda(t)\|_{\dot B^s_{q,p}(\bbR^d;\bbR^3)} \eqsim \lambda^{\tfrac 14+\tfrac 14 (s-\tfrac dq)} \|u(\lambda t)\|_{\dot B^s_{q,p}(\bbR^d;\bbR^3)}.$$
Thus the scaling exponent vanishes when $q=d$ in the Lebesgue scale
and when $s=d/q-1$ in the (homogeneous) Besov scale. Therefore, the scaling-critical spaces are $ \dot B^{d/q-1}_{q,p}(\bbR^d;\bbR^3)$ and $ L^d (\bbR^d;\bbR^3)$. 

\section{Main results}\label{section:main results}
In this section, we present the main results of our paper. We begin by studying the global well-posedness of \eqref{eq: LLB1} by means of a variational approach, with a Gelfand triple $\mathcal V \hookrightarrow \mathcal H \hookrightarrow \mathcal V^*$. First, we work in dimension $d\le 2$ with $\mathcal V= {}_\nu H^2$ and $\mathcal H=L^2$, which we call the weak setting; later, in dimension $d\le 3$ we work both with $\mathcal V =  {}_\nu H^3$ and $\mathcal H=H^1$ (intermediate setting), and $\mathcal V= {}_\nu H^4$ and $\mathcal H=  {}_\nu  H^2$ (strong setting). Lastly, we state a global well-posedness result for rough initial data from the scaling-critical trace space $B^{d/q-1}_{q,p}$.

\subsection{The variational setting}

\subsubsection{Weak setting}
 We first consider the weak setting with $\mathcal V =  {}_\nu H^2 = \{ u \in H^2 \colon \partial_n u = 0 \}$ and  $\mathcal H =L^2$ in dimension $d\le 2$. In this case, we will require the following assumptions on the noise:
\begin{assumption}[Noise assumptions in the weak setting]\label{ass:noise in weak setting} The function  $g=(g_k)_{k\ge1}\colon \bbR_+ \times \Omega \times \mathscr D\times \bbR^3\times  \bbR^{3d} \to \ell^2$ satisfies
    \begin{enumerate}[(1)]
        \item (Measurability)
        For every $(y,z)\in\mathbb R^3\times\mathbb R^{3\times d}$, the map
$
    (t,\omega,x)\longmapsto g(t,\omega,x,y,z)
$
is $\mathcal P\otimes\mathcal B(\mathscr D)$-measurable, where $\mathcal P$ is the progressive $\sigma$-algebra on $\bbR_+\times\Omega$;

\item (Integrability at the origin) For every $T>0$, the function $g^{(0)}(t,\omega)\coloneq g(t,\omega,\cdot,0,0)$ belongs to $L^2((0,T)\times\Omega; L^2(\ell^2))$, namely, 
\begin{equation}
\label{eq:g0-weak-integrability}
    \bbE \int_0^T
    \|g(t,\omega,\cdot,0,0)\|_{L^2(\ell^2)}^2
    \,\rd t
    \le C_T;
\end{equation}

    \item (Lipschitz in $u,\nabla u$) There is $L>0$ such that, uniformly in $(t,\omega,x)$, 
$y,y'\in\bbR^3$ and $z,z'\in \bbR^{3d}$,
    \begin{equation}\label{ineq:g is Lip}
        \|g(t,\omega,x,y,z ) - g(t,\omega,x,y',z') \|_{\ell^2} \le L |y-y'| + L |z-z'|.
    \end{equation}

\end{enumerate}
\end{assumption}

\begin{theorem}[Global well-posedness in the weak setting]\label{theorem:GWP_d=2_L2_time}
    Let $d\in\{1,2\}$ and suppose that Assumption \ref{ass:noise in weak setting} holds. Then, for any $u_0 \in L_{\mathcal F_0}^0(\Omega; L^2)$, there exists a unique global solution $u$ to \eqref{eq: LLB1} such that 
    \[
        u \in L_{\rm loc}^2([0,\infty);   {}_\nu  H^2)\cap  C([0,\infty); L^2) \, \text{ a.s.} 
    \]    
    Moreover, for every $T>0$ and $\gamma \in (0,1)$, 
    \begin{equation}\label{ineq:energy estimate weak 1}
        \begin{aligned}
            &\sup_{t \in [0,T]} \mathbb E  \| u(t) \|_{L^2}^2 +  \mathbb E  \int_0^T \| u(t) \|_{H^2}^2 \, {\rm d}t 
       \\
       & \qquad \leq C_T \Bigl( 1 + \mathbb{E} \| u_0 \|_{L^2}^2 + \mathbb{E}  \int_0^T \| g(t,\omega,\cdot,0,0) \|_{ L^2(\ell^2)}^{2} \, {\rm d}t  \Bigr),
        \end{aligned}
    \end{equation}
    and 
    \begin{equation} \label{ineq:energy estimate weak 2}
    \begin{aligned}  
        &\mathbb E  \sup_{t \in [0,T]}  \| u(t) \|_{L^2}^{2\gamma}  + \mathbb E  \Bigl ( \int_0^T \| u(t) \|_{H^2}^2 \, {\rm d}t \Bigr)^{\gamma}  
       \\
       & \qquad \leq C_{\gamma,T} \Bigl( 1 + \mathbb{E}  \| u_0 \|_{L^2}^{2\gamma}  +   \bbE \Bigl( \int_0^T
    \|g(t,\omega,\cdot,0,0)\|_{L^2(\ell^2)}^2 \,\rd t \Bigr)^\gamma
     \Bigr).
    \end{aligned}
    \end{equation}
\end{theorem}

Furthermore, the following continuous dependence on the initial data holds.

\begin{corollary}\label{Coroll:GWP_d=2_L2_time}
Suppose that the assumptions of Theorem \ref{theorem:GWP_d=2_L2_time} hold.  If $u_0^n \in L^0_{\mathcal F_0} (\Omega;L^2)$ are such that $\|u_0-u^n_0\|_{L^2} \to 0$ in probability, then for every $T>0$,
    \begin{equation}
        \|u-u^n\|_{L^2(0,T;H^2)}+\|u-u^n\|_{ C([0,T];L^2)} \to 0 \quad \text{in probability},
    \end{equation}
    where $u^n$ is the unique global solution to \eqref{eq: LLB1} with initial data $u_0^n$.
\end{corollary}

\begin{remark}[Parabolic regularization]\label{remark:parabolic reg weak setting}
    By replacing \eqref{eq:g0-weak-integrability} with the stronger assumption 
    $$g^{(0)} \in L^\infty((0,T)\times \Omega ; L^{q_0}(\ell^2)) ,\quad \text{ for every $T>0$}, \quad \text{ for some $q_0 \in [2,\infty)$},$$
one can show that the following instantaneous regularization holds almost surely:
    \begin{equation*}
        u \in H^{\theta,r}_{\mathrm{loc}}((0,\infty); {}_\nu H^{2-4\theta,q}), \quad 0 \le \theta<1/2, \quad r \in [2,\infty), \quad q \in [2,q_0].
        \end{equation*}
We leave the details to the interested reader, as the proof is analogous to that of Lemma \ref{lemma:parab reg variational} and we will not pursue this extension further here.
\end{remark}

The proofs of Theorem \ref{theorem:GWP_d=2_L2_time} and Corollary \ref{Coroll:GWP_d=2_L2_time} are presented in Section \ref{subsec: proofs weak setting}. As they will reveal, the weak setting is critical in dimension $d = 2$. Consequently, to treat the higher dimension $d=3$, it is necessary to work in stronger function spaces. This motivates the intermediate and strong settings introduced in the following two sections.

\subsubsection{Intermediate setting}
We consider the intermediate setting with 
$\mathcal V =  {}_\nu H^3 = \{u\in H^3: \partial_n u=0 \text{ on }\partial\mathscr D\}$ 
and $\mathcal H =H^1$. In this case, we can prove global well-posedness in  dimension $d\le 3$. In order to do so, we need higher-order regularity on $g$; specifically, we assume that $g$ and its first order derivatives are Lipschitz continuous. Moreover, contrary to the weak setting, we do not allow dependence on the gradient variable.

\begin{assumption}[Noise assumptions in the intermediate setting]\label{ass:noise_intermediate} Suppose that the function  $g=(g_k)_{k\ge1}\colon \bbR_+\times\Omega\times\mathscr D\times\bbR^3 \to \ell^2$ is independent of the gradient variable. Moreover, 
    \begin{enumerate}
    \item[(1)] (Measurability) For every $y\in\bbR^3$, the map $(t,\omega,x)\mapsto g(t,\omega,x,y)$ is
$\mathcal P\otimes \mathcal B(\mathscr D)$-measurable, where $\mathcal P$ is the progressive $\sigma$-algebra on $\bbR_+\times\Omega$;

\item [(2)]  (Integrability at the origin) For every $T>0$, the function $g^{(0)}(t,\omega)\coloneq g(t,\omega,\cdot,0)$ belongs to 
 $L^\infty((0,T)\times\Omega; W^{1,\infty}(\ell^2))$,
namely, there is $K_T>0$ such that, 
\begin{equation}\label{eq:g0-assumption_intermediate}
\|g(t,\omega,\cdot,0)\|_{W^{1,\infty}(\ell^2)}\le K_T \quad \text{for a.e.
$(t,\omega)$};
\end{equation}
    \item[(3)] (Lipschitz in $u$) The map $(x,y)\mapsto g(t,\omega,x,y)$ is continuously differentiable for a.e.
$(t,\omega)$, and there is $L>0$ such that, uniformly in $(t,\omega,x)$ and
$y,y'\in\bbR^3$,
\begin{align}
\|g(t,\omega,x,y)-g(t,\omega,x,y')\|_{\ell^2}
&\le L|y-y'|,\label{eq:g-Lip}\\
\|\nabla_x g(t,\omega,x,y)-\nabla_x g(t,\omega,x,y')\|_{\ell^2}
&\le L|y-y'|,\label{eq:gx-Lip}\\
\|\partial_y g(t,\omega,x,y)\|_{\ell^2}
&\le L,\label{eq:gy-bdd}\\
\|\partial_y g(t,\omega,x,y)-\partial_y g(t,\omega,x,y')\|_{\ell^2}
&\le L|y-y'|.\label{eq:gy-Lip}
\end{align}
\end{enumerate}
\end{assumption}

\begin{theorem}[Global well-posedness in the intermediate setting]\label{theorem:GWP_intermediate}
    Let $d\in\{1,2,3\}$ and suppose that Assumption \ref{ass:noise_intermediate} holds. Then, for any $u_0 \in L_{\mathcal F_0}^{0}(\Omega; H^1)$, there exists a unique global solution $u$ to equation \eqref{eq: LLB1} such that 
    \[
        u \in L_{\rm loc}^2([0,\infty);   {}_\nu  H^3)\cap  C([0,\infty); H^1) \, \text{ a.s.}
    \]
    Moreover, the following instantaneous regularization holds almost surely:
     \begin{equation}\label{eq:parabolic reg of intermediate sol0}
        u \in H^{\theta,r}_{\mathrm{loc}}((0,\infty); {}_\nu H^{3-4\theta,q}) , \quad 0\le \theta<1/2, \quad r,q \in [2,\infty).
    \end{equation}
   In addition, if  $u_0 \in L^{8}_{\mathcal F_0}(\Omega;H^1) $, then for every $T>0$, 
   \begin{align}\label{eq:intermediate-main-energy-estimate 2}
\mathbb E 
    \sup_{t\in[0,T]}
    \|u(t)\|_{H^1}^2
+
\mathbb E\int_0^T
    \|u(t)\|_{H^3}^2
\,\rd t
\le
C_{T}\left(
    1+\mathbb E\|u_0\|_{H^1}^{8}
\right).
\end{align}
\end{theorem}

Furthermore, the following continuous dependence on the initial data holds.

\begin{corollary}\label{Coroll:GWP_intermediate}
Suppose that the assumptions of Theorem \ref{theorem:GWP_intermediate} hold.  If $u_0^n \in L^0_{\mathcal F_0} (\Omega;H^1)$ are such that $\|u_0-u^n_0\|_{H^1} \to 0$ in probability, then for every $T>0$,
    \begin{equation}
        \|u-u^n\|_{L^2(0,T;H^3)}+\|u-u^n\|_{ C([0,T];H^1)} \to 0 \quad \text{in probability as } n \to \infty,
    \end{equation}
    where $u^n$ is the unique global solution to \eqref{eq: LLB1} with initial data $u_0^n$. 
\end{corollary}

\begin{remark}
The global well-posedness conclusion of
Theorem~\ref{theorem:GWP_intermediate} remains valid
if \eqref{eq:g0-assumption_intermediate} is replaced by
the weaker assumption
\begin{equation}\label{eq:intermediate-weaker-origin intro}
g^{(0)}\in
L^\infty((0,T)\times\Omega;W^{1,q_0}(\ell^2)),
\qquad T>0,
\quad\text{for some }q_0 \in [4,\infty).
\end{equation}
Under this assumption, the proof establishes the
instantaneous regularization
\eqref{eq:parabolic reg of intermediate sol0}
for $q\in[2,q_0]$.
\end{remark}

The proofs of Theorem \ref{theorem:GWP_intermediate} and Corollary \ref{Coroll:GWP_intermediate} are presented in Section \ref{Sec:intermediate_setting}.

\subsubsection{Strong setting}
We now consider the strong setting with $\mathcal V=  {}_\nu H^4=\{ u \in H^4 \colon \partial_n u= \partial_n \Delta u =0\} $ and $\mathcal H =  {}_\nu H^2 = \{ u \in H^2 \colon  \partial_n u = 0 \}$. Similar to the intermediate setting, we require higher-order regularity on $g$ and do not allow for dependence on the gradient variable. Moreover, we require $g$ to satisfy a compatibility condition as given below.

\begin{assumption}[Noise assumptions in the strong setting]\label{assumptions on g strong setting}
Suppose that the function $    g=(g_k)_{k\ge1}
    \colon
    \mathbb R_+\times\Omega\times{\mathscr D}\times\mathbb R^3
    \longrightarrow \ell^2
$ is independent of the gradient variable. Moreover, 
\begin{enumerate}
\item[(1)] (Measurability) For every $y\in\mathbb R^3$, the map
$
    (t,\omega,x)\longmapsto g(t,\omega,x,y)
$
is $\mathcal P\otimes\mathcal B({\mathscr D})$-measurable.
The same measurability property holds for all the derivatives appearing
below;

\item[(2)]  (Integrability at the origin) For every $T>0$, the function $g^{(0)}(t,\omega)\coloneq g(t,\omega,\cdot,0)$ belongs to $L^\infty((0,T)\times\Omega;  {}_\nu H^2(\ell^2))$, namely, there is $K_T>0$ such that, 
\begin{equation}
\label{eq:g0-H2-strong}
    \|g(t,\omega,\cdot,0)\|_{H^2(\ell^2)}
    \le K_T \quad\text{ for a.e. $(t,\omega)$;}
\end{equation}

\item[(3)] (Lipschitz in $u$) The map $(x,y)\mapsto g(t,\omega,x,y)$ is twice continuously differentiable for a.e.
$(t,\omega)$, and there is $L>0$ such that, uniformly in $(t,\omega,x)$ and
$y,y'\in\bbR^3$,
\begin{align}
\|g(t,\omega,x,y)-g(t,\omega,x,y')\|_{\ell^2}
&\le L|y-y'|,
\label{eq:g-strong-Lip}
\\
\|\nabla_xg(t,\omega,x,y)-\nabla_xg(t,\omega,x,y')\|_{\ell^2}
&\le L|y-y'|,
\label{eq:gx-strong-Lip}
\\
\|\nabla_x^2g(t,\omega,x,y)-\nabla_x^2g(t,\omega,x,y')\|_{\ell^2}
&\le L|y-y'|
\label{eq:gxx-strong-Lip}
\\
\|\partial_yg(t,\omega,x,y)\|_{\ell^2}
&\le L,
\label{eq:gy-strong-bdd}
\\
\|\partial_{xy}^2g(t,\omega,x,y)\|_{\ell^2}
&\le L,
\label{eq:gxy-strong-bdd}
\\
\|\partial_{yy}^2g(t,\omega,x,y)\|_{\ell^2}
&\le L
\label{eq:gyy-strong-bdd}
\\
\|\partial_yg(t,\omega,x,y)-\partial_yg(t,\omega,x,y')\|_{\ell^2}
&\le L|y-y'|,
\label{eq:gy-strong-Lip}
\\
\|\partial_{xy}^2g(t,\omega,x,y)-\partial_{xy}^2g(t,\omega,x,y')\|_{\ell^2}
&\le L|y-y'|,
\label{eq:gxy-strong-Lip}
\\
\|\partial_{yy}^2g(t,\omega,x,y)-\partial_{yy}^2g(t,\omega,x,y')\|_{\ell^2}
&\le L|y-y'|.
\label{eq:gyy-strong-Lip}
\end{align}

\item[(4)](Compatibility condition) The noise is compatible with the homogeneous Neumann boundary
condition: for a.e.\,$(t,\omega)$ and every $y\in\mathbb R^3$,
\begin{equation}
\label{eq:g-Neumann-compatibility}
    \partial_n^x g(t,\omega,\cdot,y)=0
    \quad\text{on }\partial\mathscr D.
\end{equation}
\end{enumerate}
\end{assumption}

\begin{theorem}[Global well-posedness in the strong setting]\label{theorem:GWP_d=3_L2_time}
    Let $d\in\{1,2,3\}$ and suppose that Assumption \ref{assumptions on g strong setting} holds. Then for any $u_0 \in L_{\mathcal F_0}^{0}(\Omega;  {}_\nu H^2)$, there exists a unique global solution $u$ to equation \eqref{eq: LLB1} such that 
    \[
        u \in L_{\rm loc}^2([0,\infty);   {}_\nu  H^4)\cap  C([0,\infty);  {}_\nu H^2)\, \text{ a.s.}
    \]
    Moreover, if $u_0 \in L_{\mathcal F_0}^{28}(\Omega; H^1)\cap L_{\mathcal F_0}^{2}(\Omega;  {}_\nu H^2)$, then for every $T>0$,
\begin{align}
\label{eq:strong-main-energy-estimate}
\mathbb E
    \sup_{t\in[0,T]}
    \|u(t)\|_{H^2}^2
+
\mathbb E\int_0^T
    \|u(t)\|_{H^4}^2
\,\rd t
\le
C_T\left(
1+\mathbb E\|u_0\|_{H^2}^2
+\bbE \|u_0\|_{H^1}^{28}
\right).
\end{align}
\end{theorem}

Furthermore, the following continuous dependence on the initial data holds.
\begin{corollary}\label{Coroll:GWP_strong}
Suppose that the assumptions of Theorem \ref{theorem:GWP_d=3_L2_time} hold.  If $u_0^n \in L_{\mathcal F_0}^{0}(\Omega;  {}_\nu H^2)$ are such that $\|u_0-u^n_0\|_{H^2} \to 0$ in probability, then for every $T>0$,
    \begin{equation}
        \|u-u^n\|_{L^2(0,T;H^4)}+\|u-u^n\|_{ C([0,T];H^2)} \to 0 \quad \text{in probability as } n \to \infty,
    \end{equation}
    where $u^n$ is the unique global solution to \eqref{eq: LLB1} with initial data $u_0^n$. 
\end{corollary}

\begin{remark}[Parabolic regularization]
Analogously to Remark \ref{remark:parabolic reg weak setting}, by replacing \eqref{eq:g0-H2-strong} with the more general assumption 
    $$g^{(0)} \in L^\infty((0,T)\times \Omega ; {}_\nu H^{2,q_0}(\ell^2)) ,\quad \text{ for every $T>0$}, \quad \text{ for some $q_0 \in [2,\infty)$},$$
one can show that the following instantaneous regularization holds almost surely:
    \begin{equation*}
        u \in H^{\theta,r}_{\mathrm{loc}}((0,\infty); {}_\nu H^{4-4\theta,q}), \quad 0 \le \theta<1/2, \quad r \in [2,\infty), \quad q \in [2,q_0].
        \end{equation*}
\end{remark}

The proofs of Theorem \ref{theorem:GWP_d=3_L2_time} and Corollary \ref{Coroll:GWP_strong} are presented in Section \ref{Sec:strong_setting}.

\subsection{Rough initial data} \label{sec:rough initial data}
We now examine the global well posedness of equation \eqref{eq: LLB1} with rough initial data from the scaling-critical space
\[
    B^{d/q-1}_{q,p}, \quad q \in [2,\infty), \quad p \in [2,\infty).
\]
Note that when
\(q>d\), the exponent \(d/q-1\) is negative, so the initial
data may live in the space of distributions and need not belong to any of the
variational energy spaces considered above.  Moreover, as $q\uparrow\infty$, the smoothness index $d/q-1$ approaches $-1$ from above.

\begin{theorem}[Global well-posedness in the critical space]
\label{thm:rough-final-global}
Suppose that Assumption~\ref{ass:noise_intermediate} holds. Let $d\in \{1,2,3\}$ and assume either 
\begin{enumerate}[(I)]
    \item   $q\in[2,\infty)$, $p \in (2,\infty)$ and $s,\kappa$ satisfy 
    $$
 3-\frac dq<s<\min\left\{3,5-\frac dq-\frac4p\right\}, \qquad 
\kappa = \frac{p(5 - s - d/q) - 4}{4};
$$

\item  $p=q=2$ and 
$$s=3-\frac d2 , \qquad \kappa=0.$$
\end{enumerate}
Then, for every
$u_0\in L^0_{\mathcal F_0}(\Omega;B^{d/q-1}_{q,p})$, there exists a unique global 
solution $u$ to equation \eqref{eq: LLB1} such that, almost surely,
\begin{equation}\label{eq:rough-global-critical-regularity}
 u\in L^p_{\mathrm{loc}}([0,\infty),t^\kappa; {}_\nu H^{4-s,q})
       \cap C([0,\infty);B^{d/q-1}_{q,p}).
\end{equation}
Moreover, the following instantaneous regularization holds almost surely:
\begin{equation}\label{eq:rough-global-smoothing}
 u\in H^{\theta,r}_{\mathrm{loc}}
       ((0,\infty); {}_\nu H^{3-4\theta,\zeta}),
 \qquad 0\le\theta<\frac12,\quad r,\zeta\in[2,\infty).
\end{equation}
\end{theorem}

\begin{remark}
Theorem \ref{thm:rough-final-global} remains valid
if \eqref{eq:g0-assumption_intermediate} is replaced by
the weaker assumption 
\eqref{eq:intermediate-weaker-origin intro}. 
Under this assumption, the proof establishes the
instantaneous regularization
\eqref{eq:rough-global-smoothing}
for $\zeta \in[2,q_0]$.
\end{remark}

The proof of Theorem \ref{thm:rough-final-global} is presented in
Section \ref{subsec:critical-LpLq-full-range}.

\section{Proofs}\label{section:proofs}

\subsection{Overview and auxiliary results}

In this section, we prove the main theorems presented in section \ref{section:main results}. We start by considering the problem in an abstract variational setting and show how it can be solved from the general framework of \cite{AgVe22, AgVe25} for the specific choices of Hilbert spaces $\mathcal V$, $\mathcal H$ leading to our weak, intermediate, and strong setting.

Let $(\mathcal V, \mathcal H, \mathcal V^*)$ be a Gelfand triple. Following the notation in \cite{AgVe25}, we rewrite \eqref{eq: LLB1} as
\[
\begin{cases}
    {\rm d}u + Au \, {\rm d}t = F(u) \, {\rm d}t + G(u) \, {\rm d}W_{\ell^2}, &t>0
    \\
    u(0)=u_0,
\end{cases}
\]
where $A \in \mathscr L( \mathcal V , \mathcal V^*)$ and $G\colon \bbR_+ \times \Omega \times \mathscr D \times \mathcal V \to \mathscr L_2(\ell^2, \mathcal H)  $ are given by 
\begin{equation}\label{eq:definition of A and G}
    Au = \alpha_1\Delta^2 u +\alpha_2\Delta u +\alpha_3 u, \quad  (G(u))(t,\omega,x) = \big(g_k(t,\omega,x,u(x),\nabla u(x))\big)_{k\ge1} ,
\end{equation}
 and the nonlinearity $F  \colon \mathcal V \to \mathcal V^* $ is split into three terms $F=\gamma_1F_1+\gamma_2F_2+\gamma_3F_3$ with 
\begin{equation}\label{eq:definition of F}
    F_1(u) =  \Delta(|u|^2 u), \quad F_2(u) =  \Delta u \times u, \quad F_3(u) = -|u|^2 u.
\end{equation}
Recall that $\alpha_1, \gamma_1,\gamma_2,\gamma_3$ are positive constants and  $\alpha_2, \alpha_3\in \mathbb{R}$. 
Note that the measurability assumptions of $g$ imply that $G$ is well defined and that $(t,\omega)\mapsto G(t,\omega)$ is progressively measurable. 

Our main tool for proving local well-posedness in the variational setting is \cite[Theorem 3.3]{AgVe24a} (see also \cite[Theorem 6.2 and Remark 4.4]{AgVe25}), which we recall for later reference.

\begin{theorem}[Local well-posedness and blow-up criteria]\label{Thm:LWP_variational}
Assume that 
\begin{enumerate}[(1)]
\item There exist $\theta>0$ and $M\ge 0$ such that for all $u\in \mathcal V$, 
\begin{equation}\label{ineq:coerc_A}
\langle u, Au\rangle_{\mathcal V, \mathcal V^*} \ge \theta \|u\|_{\mathcal V}^2 -M \|u\|_{\mathcal H}^2;
\end{equation}

\item The map $
        G\colon
        \mathbb R_+\times\Omega\times \mathscr D \times \mathcal V
        \longrightarrow
        \mathscr L_2(\ell^2,\mathcal H)
    $
    is $\mathcal P\otimes\mathcal B(\mathcal V)$-measurable.
    Moreover, setting 
    $        G_0(t,\omega,x) \coloneq G(t,\omega,x, 0),
    $
    one has
    \begin{equation}\label{eq:G0-L2loc}
        G_0\in
        L^2_{\mathrm{loc}}(
            [0,\infty);
            \mathscr L_2(\ell^2,\mathcal H))\, \text{ a.s.}
    \end{equation}
    
\item  There is $L>0$ such that for all $u,v \in \mathcal V$,
\begin{equation}\label{ineq:Local_Lip_F}
\|F_j(u) - F_j(v)\|_{\mathcal V^*}  \le  L \big(1 + \|u\|_{\mathcal V_{\beta_j}}^{\rho_j} + \|v\|_{\mathcal V_{\beta_j}}^{\rho_j}\big) \|u - v\|_{\mathcal V_{\beta_j}}, \quad j\in \{1,2,3\},
\end{equation}
and, for all $t \in \bbR_+$ and $\omega \in \Omega$, 
\begin{equation}\label{ineq:Local_Lip_G}
		\|G(t,\omega,u) - G(t,\omega,v)\|_{\mathscr{L}_2(\ell^2,\mathcal H)}  \le L  \big(1 + \|u\|_{\mathcal V_{\beta_4}}^{\rho_4} + \|v\|_{\mathcal V_{\beta_4}}^{\rho_4}\big) \|u - v\|_{\mathcal V_{\beta_4}} ,
	\end{equation}
	where $\beta_j \in (1/2, 1)$ and $\rho_j \geqslant 0$ satisfy the sub-criticality condition
\begin{equation}\label{ineq:subcriticality condition}
		(2\beta_j - 1)(\rho_j + 1) \le 1, \quad j \in \{1, \ldots, 4\}.
	\end{equation}
\end{enumerate}
Then, for any $u_0 \in L_{\mathcal F_0}^0(\Omega; \mathcal{H})$, there exists a unique maximal solution $(u,\sigma)$ to equation \eqref{eq: LLB1} with $\sigma>0$ a.s., such that $u \in L^2_{\mathrm{loc}}([0,\sigma);\mathcal V) \cap C([0,\sigma);\mathcal H)$.
Moreover, the following blow-up criteria hold:
\begin{align}
    &\mathbb{P}\Bigl(\sigma<\infty,\sup_{t\in [0,\sigma)}\|u(t)\|_{\mathcal{H}}+\|u\|_{L^2(0,\sigma;\mathcal{V})}<\infty\Bigr)=0,
   \intertext{and, if \eqref{ineq:subcriticality condition} holds with strict inequality,}
   \label{eq: blow-up-criterion-subcritical}
      &\mathbb{P}\Bigl(\sigma<\infty,\sup_{t\in [0,\sigma)}\|u(t)\|_{\mathcal{H}}<\infty\Bigr)=0.
\end{align}
\end{theorem}

To relate the above assumptions to the growth condition in
\cite[Theorem~3.3]{AgVe24a}, decompose the noise coefficient as
\[
    G(t,\omega,x,u)
    =
    \widetilde G(t,\omega,x,u)+G_0(t,\omega,x),
\]
where
\[
    \widetilde G(t,\omega,x,u)
    \coloneq
    G(t,\omega,x,u)-G(t,\omega,x,0).
\]
By \eqref{eq:G0-L2loc}, the process \(G_0\) is an admissible
inhomogeneous additive noise term. Moreover, taking \(v=0\) in
\eqref{ineq:Local_Lip_G}, we obtain
\begin{align}
    \label{eq:growth-centered-G}
    \|\widetilde G(t,\omega,u)\|
        _{\mathscr L_2(\ell^2,\mathcal H)}
    &\leq
    C_L(1+\|u\|_{\mathcal V_{\beta_4}}^{\rho_4})
    \|u\|_{\mathcal V_{\beta_4}}
   \le 
    C_L(1+\|u\|_{\mathcal V_{\beta_4}}^{\rho_4+1}).
\end{align}
Consequently, \(\widetilde G\) satisfies the polynomial growth
condition required in \cite[Theorem~3.3]{AgVe24a}.

\subsection{Proof of well-posedness in the weak setting}
\label{subsec: proofs weak setting}

In this subsection, we prove Theorem \ref{theorem:GWP_d=2_L2_time} and Corollary \ref{Coroll:GWP_d=2_L2_time}. Recall $\mathcal V =  {}_\nu H^2 = \{ u \in H^2 \colon   \partial_n u = 0 \}$, $\mathcal H =L^2$ and $d\in \{1,2\}$. 

We begin by first establishing local well-posedness for equation \eqref{eq: LLB1}.
\begin{proposition}[Local well-posedness]\label{prop:LWP in weak setting}
     Let $d\in\{1,2\}$ and suppose that Assumption \ref{ass:noise in weak setting} holds. Then, for any $u_0 \in L_{\mathcal F_0}^0(\Omega; L^2)$, there exists a unique  maximal local solution $(u,\sigma)$ to equation \eqref{eq: LLB1} with $\sigma>0$ a.s., such that 
     $$u \in L^2_{\mathrm{loc}}([0,\sigma); {}_\nu H^2) \cap C([0,\sigma);L^2).$$
\end{proposition}
\begin{proof}
We verify the assumptions of Theorem \ref{Thm:LWP_variational} in our framework. We begin by checking that  the linear part of \eqref{eq: LLB1} is coercive, i.e., that \eqref{ineq:coerc_A} holds. By standard elliptic $L^2$-theory 
(see \cite[Proposition $7.4$]{Ta10}) it is enough to show that there are constants $\theta>0$ and $M\ge 0$ such that 
\begin{equation}\label{eq:proof_coerc_A}
\langle u, Au \rangle_{ {}_\nu H^2,( {}_\nu H^2)^*} \ge  \theta \|\Delta u\|_{L^2}^2-M\|u\|_{L^2}^2, \quad u \in  {}_\nu H^2.
\end{equation}
To prove \eqref{eq:proof_coerc_A}, using integration by parts, we calculate 
\begin{equation}\label{eq:formula for <u,Au>}
    \langle u, Au \rangle_{ {}_\nu H^2,( {}_\nu H^2)^*}= \alpha_1\|\Delta u\|_{L^2}^2-\alpha_2\|\nabla u\|^2_{L^2} +\alpha_3 \|u\|_{L^2}^2 .
\end{equation} 
We recall that for any $\epsilon>0$, Young's inequality and interpolation imply
\[\|\nabla u\|^2_{L^2} \le C_\epsilon\|u\|^2_{L^2}+\epsilon \|\Delta u\|^2_{L^2}.\]

 Hence, absorbing the term $\|\nabla u\|_{L^2}$ in \eqref{eq:formula for <u,Au>} gives \eqref{eq:proof_coerc_A}.

Next, we verify \eqref{ineq:Local_Lip_F} for the nonlinearities $F_1,F_3$ as defined in \eqref{eq:definition of F}. We estimate  
\begin{equation}\label{ineq:estimating F1 and F3}
	\begin{aligned}
		\| F_1(u) - F_1(v) \|_{( {}_\nu H^2)^*} + \| F_3(u) - F_3(v) \|_{( {}_\nu H^2)^*} & \lesssim  \| |u|^2 u - |v|^2 v \|_{L^2} 
		\\
		&\stackrel{\mathrm{(i)}}{\lesssim}  \| (u-v) (1+|u|^2+|v|^2) \|_{L^2} 
		\\ 
		&\stackrel{\mathrm{(ii)}}{\lesssim}   ( 1 + \|u\|_{L^6}^2 + \|v\|_{L^6}^2 ) \| u-v \|_{L^6}
		\\
		&\stackrel{\mathrm{(iii)}}{\lesssim} ( 1 + \|u\|_{H^{4\beta-2}}^2 + \|v\|_{H^{4\beta-2}}^2 ) \| u-v \|_{H^{4\beta-2}}
		\\
		& \stackrel{\mathrm{(iv)}}{\lesssim} ( 1 + \|u\|_{\mathcal V_\beta}^2 + \|v\|_{\mathcal V_\beta}^2 ) \| u-v \|_{\mathcal V_\beta}.
	\end{aligned}
\end{equation}
Indeed, (i) follows from 
$
\bigl||u|^2u-|v|^2v\bigr| \lesssim (1+|u|^2+|v|^2)|u-v|,
$ 
(ii) follows from H\"older's inequality, (iii) from the Sobolev embedding $H^{4\beta-2} \hookrightarrow L^6$ which holds in dimension $d$ when $\beta\ge (d+6)/12$, and (iv) follows from the embedding $\mathcal V_\beta \hookrightarrow  H^{4\beta-2}$. Therefore  \eqref{ineq:Local_Lip_F} holds with $\rho=2$, $\beta=2/3$ and in that case the criticality condition \eqref{ineq:subcriticality condition} is satisfied with equality.

We now verify \eqref{ineq:Local_Lip_F} for the nonlinearity $F_2$ given in \eqref{eq:definition of F}. We do so by a duality argument. Let $\phi \in  {}_\nu H^2$ be fixed but arbitrary. 
It suffices to prove that
\begin{equation}\label{ineq:duality estimate for F2}
	|(\phi,-\Delta u \times u + \Delta v \times v )_{L^2}| \lesssim  \|\phi\|_{H^2} (1 + \| u \|_{H^1} + \| v \|_{H^1} ) \| u - v \|_{H^1}.
\end{equation}
Integrating by parts and taking into account the boundary conditions of $u,v,\phi$ and the properties of the cross product, we get
\begin{align*}
\bigl(\phi,-\Delta u\times u+\Delta v\times v\bigr)_{L^2}
&=\int_{\mathscr D}
\bigl(\nabla u\times u-\nabla v\times v\bigr):\nabla\phi
\,\rd x
\\
&=
\int_{\mathscr D}
\Bigl(
\nabla(u-v)\times u+\nabla v\times(u-v)
\Bigr):\nabla\phi
\,\rd x.
\end{align*}
Consequently,
\begin{equation}
\begin{aligned}
\label{ineq:estimating F2}
\left|
\bigl(\phi,-\Delta u\times u+\Delta v\times v\bigr)_{L^2}
\right|
&\le
\int_{\mathscr D}
\Bigl(
|\nabla(u-v)|\,|u|
+
|\nabla v|\,|u-v|
\Bigr)|\nabla\phi|
\,\rd x
\\
&\stackrel{\mathrm{(i)}}{\le}
\|\nabla\phi\|_{L^4}
\left(
\|\nabla(u-v)\|_{L^2}\|u\|_{L^4}
+
\|\nabla v\|_{L^2}\|u-v\|_{L^4}
\right)
\\
&\stackrel{\mathrm{(ii)}}{\lesssim}
\|\phi\|_{H^2}
\left(
\|u\|_{H^1}+\|v\|_{H^1}
\right)
\|u-v\|_{H^1}
\\
&\, \lesssim \|\phi\|_{H^2}
(
\|u\|_{\mathcal V_{3/4}}
+
\|v\|_{\mathcal V_{3/4}}
)
\|u-v\|_{\mathcal V_{3/4}},
\end{aligned}
\end{equation}
where (i) follows from  H\"older's inequality with  $\tfrac 12+\tfrac 14+\tfrac 14=1$, (ii) from the Sobolev embeddings $H^{2} \hookrightarrow W^{1,4}$ and $H^1 \hookrightarrow L^4$ (both valid in $d \le 4$), and (iii) from the embedding $\mathcal V_{3/4} \hookrightarrow H^1 $. This shows that $F_2$ satisfies \eqref{ineq:Local_Lip_F} with $\rho=1$, $\beta=3/4$ and that \eqref{ineq:subcriticality condition} holds with equality.

Finally, we verify the assumptions on $G$. Note that the measurability of $G$ and the integrability assumption \eqref{eq:G0-L2loc} on $G_0$ follow by Assumption \ref{ass:noise in weak setting}. Recall that $\mathscr L_2(\ell^2,L^2)=L^2(\mathscr D; \ell^2)$ isometrically. Moreover, \eqref{ineq:g is Lip} implies that for all $u,v \in {}_\nu H^2$, $t \in \mathbb R_+$ and $\omega \in \Omega$ 
\begin{equation}\label{ineq:estimating G}
	\|G(t,\omega,u) - G(t,\omega,v)\|_{\mathscr{L}_2(\ell^2,L^2)} \lesssim \|u-v\|_{L^2} + \|\nabla (u-v)\|_{L^2} \lesssim \|u-v\|_{H^1} \lesssim \|u-v\|_{\mathcal V_{3/4}},
\end{equation}
where in the last step we used the embedding $\mathcal V_{3/4} \hookrightarrow H^1$. This shows that $G$ satisfies \eqref{ineq:Local_Lip_G} with $\rho=0$, $\beta=3/4$ and the sub-criticality condition \eqref{ineq:subcriticality condition} is satisfied with strict inequality.

\end{proof}

We now present the proofs of  Theorem \ref{theorem:GWP_d=2_L2_time} and Corollary \ref{Coroll:GWP_d=2_L2_time}.

\begin{proof}[Proof of Theorem \ref{theorem:GWP_d=2_L2_time}]
Let $(u,\sigma)$ be the local solution given by Proposition \ref{prop:LWP in weak setting}. 
We apply \cite[Theorem $3.5$]{AgVe24a}. We show that there exists $\phi \in L^2((0,T) \times \Omega)$ such that for any $v \in {}_\nu H^2$, $t \in [0,T]$ and $\omega\in \Omega$,
\begin{equation}\label{ineq: coercivity for GWP2}
\langle v, Av - F(v) \rangle_{{}_\nu H^2, ({}_\nu H^2)^*} - \tfrac12 \| G(t,\omega,v) \|_{\mathscr L_2(\ell^2,L^2)}^2 \geq \theta' \| v \|_{H^2}^2 - M' \| v \|_{L^2}^2 - \left|\phi(t,\omega)\right|^2.
\end{equation}
Recalling the definition of $F$ in \eqref{eq:definition of F} and
integrating by parts, we get
\begin{align*}
    -\langle v, F(v) \rangle_{{}_\nu H^2, ({}_\nu H^2)^*} 
    =  \gamma_1 \int_{\mathscr D} \Bigl(|v|^2 |\nabla v|^2+2 \sum_{j=1}^d (v\cdot \partial_j v)^2\Bigr) dx 
    + \gamma_3\| |v|^2 \|_{L^2}^2
    \ge 0.
\end{align*}
On the other hand, by the triangle inequality, Assumption \ref{ass:noise in weak setting} on the noise $G$, and Young's inequality there holds
\begin{equation}\label{eq:control_G_L2}
\begin{split}
\tfrac 12\|G(t,\omega,v)\|_{\mathscr{L}_2(\ell^2,L^2)}^2 & \leq      \|G(t,\omega,0)\|_{\mathscr{L}_2(\ell^2,L^2)}^2+\|G(t,\omega,v)-G(t,\omega,0)\|_{\mathscr{L}_2(\ell^2,L^2)}^2 
    \\
   & \le \|G(t,\omega,0)\|_{\mathscr{L}_2(\ell^2,L^2)}^2+ L\|v\|_{L^2}^2 + L\|\nabla v\|_{L^2}^2
   \\
   & \le \|G(t,\omega,0)\|_{\mathscr{L}_2(\ell^2,L^2)}^2 + LC_\epsilon \| v \|_{L^2}^2 + L \epsilon  \| \Delta v \|_{L^2}^2.
   \end{split}
\end{equation}
Let $\theta >0$ and $M\ge 0$ be as in \eqref{eq:proof_coerc_A}. Choosing $\varepsilon>0$ small enough, the above estimates and \eqref{eq:proof_coerc_A} give \eqref{ineq: coercivity for GWP2} 
with $\phi(t,\omega)\coloneq C\|G(t,\omega,0)\|_{\mathscr{L}_2(\ell^2,L^2)} $. Note  $\phi \in L^2((0,T)\times \Omega)$ by \eqref{eq:g0-weak-integrability}. Hence, the assumptions of \cite[Theorem $3.5$]{AgVe24a} are satisfied, which concludes the proof. 

\end{proof}

\begin{proof}[Proof of Corollary \ref{Coroll:GWP_d=2_L2_time}]
    This follows from \cite[Theorem 3.8]{AgVe24a}.
    
\end{proof}

\subsection{Proof of well-posedness in the intermediate setting}
\label{Sec:intermediate_setting}
In this section we prove Theorem \ref{theorem:GWP_intermediate} and Corollary \ref{Coroll:GWP_intermediate}. Recall that in this case
$\mathcal V =  {}_\nu H^3 = \{u\in H^3: \partial_n u=0 \text{ on }\partial\mathscr D\}$,
 $\mathcal H =H^1$ and $d\in \{1,2,3\}$.

\subsubsection{Local well-posedness in the intermediate setting} \label{sec: LWP in intermediate}
We begin by establishing local well-posedness of equation \eqref{eq: LLB1}:

\begin{proposition}[Local well-posedness] \label{thm:LWP in intermediate}
     Let $d\in\{1,2,3\}$ and suppose that Assumption \ref{ass:noise_intermediate} holds. Then for any $u_0 \in L_{\mathcal F_0}^{0}(\Omega; H^1)$ there exists a unique maximal solution $(u,\sigma)$ of equation \eqref{eq: LLB1} with $\sigma>0$ a.s. such that 
      \[
        u \in L_{\rm loc}^2([0,\sigma);   {}_\nu  H^3)\cap  C([0,\sigma); H^1) \, \text{ a.s.}
    \]
      Moreover, the following blow up criterion holds: 
     \begin{equation}\label{eq:blow up criterion in intermediate}
    \mathbb{P}\bigl(\sigma<\infty, \sup_{t\in[0,\sigma)}\|u(t)\|_{H^1}<\infty\bigr)=0.
\end{equation}
\end{proposition}
\begin{proof}
We verify the assumptions of Theorem \ref{Thm:LWP_variational} in our framework. To this end, we first check that the linear part of \eqref{eq: LLB1} is coercive in these
spaces, i.e., that \eqref{ineq:coerc_A} holds.
For sufficiently smooth $u$ satisfying $\partial_n u=\partial_n \Delta u =0$, we obtain
\begin{align*}
    \langle u,Au\rangle_{ {}_\nu H^3, ( {}_\nu H^3)^*}
    &=\int_{\mathscr D}(u-\Delta u)\cdot
    \bigl(\alpha_1\Delta^2u+\alpha_2\Delta u+\alpha_3u\bigr)\,{\rm d}x  \\
    &=\alpha_1\|\nabla\Delta u\|_{L^2}^2
      +(\alpha_1-\alpha_2)\|\Delta u\|_{L^2}^2
      +(\alpha_3-\alpha_2)\|\nabla u\|_{L^2}^2
      +\alpha_3\|u\|_{L^2}^2 ,
\end{align*}
where we used the identities
\[
\int_{\mathscr D}u\cdot\Delta^2u\,dx=\|\Delta u\|_{L^2}^2 
\quad \text{and} \quad 
-\int_{\mathscr D}\Delta u\cdot\Delta^2u\,{\rm d}x
=\|\nabla\Delta u\|_{L^2}^2.
\]

Note that by elliptic regularity,  $\|u\|_{H^3}^2
    \eqsim (\|\nabla\Delta u\|_{L^2}^2+\|u\|_{H^1}^2\bigr)$. Hence, by interpolation, for every \(\varepsilon>0\),
\[
    \|\Delta u\|_{L^2}^2 
    \lesssim   \varepsilon\|u\|_{H^3}^2
       +C_\varepsilon\|u\|_{H^1}^2 .
\]
Therefore, choosing \(\varepsilon>0\) small enough and absorbing the
\(\|\Delta u\|_{L^2}^2\)-term into the leading term $\|u\|^2_{H^3}$, we obtain
\begin{equation}\label{ineq: linear coercivity in weakened setting}
    \langle u,Au\rangle_{ {}_\nu H^3, ( {}_\nu H^3)^*} 
    \ge \theta\|u\|_{H^3}^2-M\|u\|_{H^1}^2 .
\end{equation}
Next, we verify that the local Lipschitz property \eqref{ineq:Local_Lip_F} holds for the nonlinearity $F$, which we again study separately in $F_1, F_2, F_3$ as defined in \eqref{eq:definition of F}. 
We start with the term $F_3$. Following the same arguments as in \eqref{ineq:estimating F1 and F3}, we get  
\[\begin{aligned}
\| F_3(u) - F_3(v) \|_{H^{-1}} &\lesssim ( 1 + \|u\|_{H^{4\beta-1}}^2 + \|v\|_{H^{4\beta-1}}^2 ) \| u-v \|_{H^{4\beta-1}},
	\end{aligned}
\]
where $\beta\ge (d+3)/12$. Therefore, \eqref{ineq:Local_Lip_F} holds with $\rho=2$, and all $\beta \in (1/2, 2/3)$ and in that case the criticality condition \eqref{ineq:subcriticality condition} is satisfied with a strict inequality. For simplicity, we take $\beta=7/12$.

To control the term $F_1$, we use a duality argument. Let $\phi \in H^1$ be fixed but arbitrary.
We prove that 
\begin{equation}\label{ineq: local Lip estimate for F1 in weakened}
   \left|
\left\langle
\Delta(|u|^2u)-\Delta(|v|^2v),\phi
\right\rangle_{H^{-1},H^1}
\right|  \lesssim  \|\phi\|_{H^1}(\|u\|^2_{\mathcal V_{7/12}}+\|v\|^2_{\mathcal V_{7/12}})\|u-v\|_{\mathcal V_{7/12}}
\end{equation}
which implies that $F_1$ satisfies \eqref{ineq:Local_Lip_F} with $\rho=2$, $\beta=7/12$ and that \eqref{ineq:subcriticality condition} holds with a strict inequality.
Integrating by parts implies
\[
\left|
\left\langle
\Delta(|u|^2u)-\Delta(|v|^2v),\phi
\right\rangle_{H^{-1},H^1}
\right|  \le
\int_{\mathscr D}
\left|\nabla(|u|^2u-|v|^2v)\right|
|\nabla\phi|\,dx .
\]
Using again the identity
$
|u|^2u-|v|^2v
=
|u|^2(u-v)+(|u|^2-|v|^2)v
$ 
and elementary calculations,
\begin{equation} \label{ineq: estimating gradient for F1 in weakened}
|\nabla(|u|^2u-|v|^2v)|
\lesssim
(|u|^2+|v|^2)|\nabla(u-v)|
+
(|u|+|v|)(|\nabla u|+|\nabla v|)|u-v|.
\end{equation}
Therefore, 
\begin{equation}\label{ineq:integration by parts I and II}
    \left|
\left\langle
\Delta(|u|^2u)-\Delta(|v|^2v),\phi
\right\rangle_{H^{-1},H^1}
\right|  \le
\int_{\mathscr D}
\left|\nabla(|u|^2u-|v|^2v)\right|
|\nabla\phi|\,{\rm d}x  \lesssim I+II,
\end{equation}
where we denoted
\[
\begin{split}
I& \coloneq  \int_{\mathscr D} \left( |u|^2 + |v|^2 \right)|\nabla(u - v)||\nabla \phi| \, {\rm d}x;\\
II& \coloneq \int_{\mathscr D}
(|u|+|v|)(|\nabla u|+|\nabla v|)
|u-v|\,|\nabla\phi|\,{\rm d}x .
\end{split}
\]
We estimate \(I\) by H\"older's inequality with exponents $
\frac{2}{18}+\frac{7}{18}+\frac12=1$
\[
\begin{aligned}
I
&\le
\left(\|u\|_{L^{18}}^2+\|v\|_{L^{18}}^2\right)
\|\nabla(u-v)\|_{L^{18/7}}
\|\nabla\phi\|_{L^2}.
\end{aligned}
\]
Since \(d\le3\), the Sobolev embedding
$
H^{4/3}\hookrightarrow L^{18}\cap W^{1,18/7}$ 
gives
\[
I
\lesssim
\left(\|u\|_{H^{4/3}}^2+\|v\|_{H^{4/3}}^2\right)
\|u-v\|_{H^{4/3}}
\|\phi\|_{H^1}.
\]
For \(II\), we use H\"older's inequality with exponents $
\frac1{18}+\frac7{18}+\frac1{18}+\frac12=1$
\[
\begin{aligned}
II
&\le
\left(\|u\|_{L^{18}}+\|v\|_{L^{18}}\right)
\left(\|\nabla u\|_{L^{18/7}}+\|\nabla v\|_{L^{18/7}}\right)
\|u-v\|_{L^{18}}
\|\nabla\phi\|_{L^2}
\\
& \lesssim \left(\|u\|_{H^{4/3}}+\|v\|_{H^{4/3}}\right)^2
\|u-v\|_{H^{4/3}}
\|\phi\|_{H^1}.
\end{aligned}
\]
By combining the estimates for \(I\) and \(II\) and noting $\mathcal V_{7/12}\hookrightarrow H^{4/3}$, we obtain  \eqref{ineq: local Lip estimate for F1 in weakened}.

To control the remaining term $F_2$, we again use a duality argument. 
Integration by parts gives
\[
\begin{aligned}
\left\langle
    \Delta u\times u-\Delta v\times v,\phi
\right\rangle_{H^{-1},H^1}
&=
-\int_{\mathscr D}
(
    \nabla u\times u-\nabla v\times v
):\nabla\phi\,\rd x .
\end{aligned}
\]
Moreover, by the identity $
    \nabla u\times u-\nabla v\times v
    =
    \nabla u\times(u-v)+\nabla(u-v)\times v 
$ we get
\[
\begin{aligned}
\left|
\left\langle
    \Delta u\times u-\Delta v\times v,\phi
\right\rangle_{H^{-1},H^1}
\right|
&\leq
C\int_{\mathscr D}
|\nabla\phi|
\left(
    |\nabla u||u-v|
    +
    |v||\nabla(u-v)|
\right)\,\rd x .
\end{aligned}
\]
Hence, by H\"older's inequality with exponents $\frac12+\frac13+\frac16=1$ and the Sobolev embedding $H^{3/2}\hookrightarrow W^{1,3} \cap  L^6$, we get 
\[                                                   
\begin{aligned}                                      
\left|                                              
\left\langle                                         
\Delta u\times u-\Delta v\times v,\phi               
\right\rangle_{H^{-1},H^1}                     
\right|
&\lesssim                                      
\|\nabla\phi\|_{L^2}                                 
\left(                                               
\|\nabla u\|_{L^3}\|u-v\|_{L^6}                      
+                                                    
\|v\|_{L^6}\|\nabla(u-v)\|_{L^3}                     
\right)
\\
&\lesssim  \|\phi\|_{H^1}\left(\|u\|_{H^{3/2}}+\|v\|_{H^{3/2}} \right)\|u-v\|_{H^{3/2}}. 
\end{aligned}                                        
\]                                                   

Hence, $F_2$ satisfies \eqref{ineq:Local_Lip_F} with $\rho=1$, $\beta=5/8$ and \eqref{ineq:subcriticality condition} holds with a strict inequality.

Now, we verify the assumptions on the noise term $G$. The measurability of $G$ and the integrability condition \eqref{eq:G0-L2loc} follow by Assumption \ref{ass:noise_intermediate}. We now prove the local Lipschitz property \eqref{ineq:Local_Lip_G} with $\beta=5/8$ and $\rho=1$. For notational brevity, we suppress the dependence on the variables $t,\omega,x$. By direct calculation, one obtains  that
\[\|G(u) - G(v)\|_{\mathscr{L}_2(\ell^2,H^1)} \lesssim \|G(u) - G(v)\|_{L^2(\ell^2)}+\|\nabla G(u) - \nabla G(v)\|_{L^2(\ell^2)}.\]
 The lower order term is immediate from \eqref{eq:g-Lip}:
\begin{equation}\label{eq:G-L2-Lip-detailed}
\|G(u)-G(v)\|_{L^2(\ell^2)}
\le L\|u-v\|_{L^2}
\le L\|u-v\|_{H^{3/2}}.
\end{equation}
It remains to estimate the term with the gradient. By the chain
rule 
$$
\nabla G(u)=\nabla_x g(\cdot,u)+\partial_y g(\cdot,u)\nabla u 
$$
and thus 
\begin{align}
\nabla G(u)-\nabla G(v)
&=\bigl[\nabla_x g(\cdot,u)-\nabla_x g(\cdot,v)\bigr]
+\partial_y g(\cdot,u)\nabla(u-v)
\notag\\
&\quad+\bigl[\partial_y g(\cdot,u)-\partial_y g(\cdot,v)\bigr]\nabla v.
\label{eq:gradient-difference-G}
\end{align}
We estimate the three terms on the right-hand side of
\eqref{eq:gradient-difference-G}. First, by \eqref{eq:gx-Lip},
\begin{equation}\label{eq:G-grad-term-1}
\|\nabla_x g(\cdot,u)-\nabla_x g(\cdot,v)\|_{L^2(\ell^2)}
\le L\|u-v\|_{L^2} \lesssim \|u-v\|_{H^{3/2}}.
\end{equation}
Second, by the uniform boundedness of $\partial_y g$ in \eqref{eq:gy-bdd},
\begin{equation}\label{eq:G-grad-term-2}
\|\partial_y g(\cdot,u)\nabla(u-v)\|_{L^2(\ell^2)}
\le L\|\nabla(u-v)\|_{L^2}\lesssim \|u-v\|_{H^{3/2}}.
\end{equation}
Finally, using the Lipschitz bound \eqref{eq:gy-Lip}, H\"older's inequality, and the
Sobolev embeddings $H^1\hookrightarrow L^6$ and $H^{3/2}\hookrightarrow W^{1,3}$ we obtain
\begin{align}
\|[\partial_y g(\cdot,u)-\partial_y g(\cdot,v)]\nabla v\|_{L^2(\ell^2)}
&\le L\||u-v|\,|\nabla v|\|_{L^2}
\le L\|u-v\|_{L^6}\|\nabla v\|_{L^3}
\notag\\
&\lesssim \|u-v\|_{H^1}\|v\|_{H^{3/2}}
\lesssim \|u-v\|_{H^{3/2}}\|v\|_{H^{3/2}}.
\label{eq:G-grad-term-3}
\end{align}
By combining \eqref{eq:G-L2-Lip-detailed}--\eqref{eq:G-grad-term-3}, and noting that $\mathcal V_{5/8} \hookrightarrow H^{3/2}$, we get 
\[
\|G(u)-G(v)\|_{H^1(\ell^2)}\lesssim  (1+\|u\|_{\mathcal V_{5/8}}+\|v\|_{\mathcal V_{5/8}})\|u-v\|_{\mathcal V_{5/8}},
\]
which proves \eqref{ineq:Local_Lip_G} with $\beta=5/8$ and $\rho=1$. Moreover, the sub-criticality condition \eqref{ineq:subcriticality condition} holds with strict inequality.

\end{proof}

\subsubsection{Global well-posedness in the intermediate setting}
In this section, we present the proof of Theorem \ref{theorem:GWP_intermediate}.
Note that we cannot simply apply \cite[Theorem $3.5$]{AgVe24a} to obtain global well-posedness as in the weak setting, since the full coercivity property \eqref{ineq: coercivity for GWP2} does not hold in this case. We instead rely on the blow-up criterion \eqref{eq:blow up criterion in intermediate}, for which we prove suitable energy estimates.  In fact, we establish energy estimates for general $L^p$-moments since these will also be used in Section \ref{subsesction: GWP in strong} to prove global well-posedness in the strong setting. 

\begin{proposition}[$L^p(\Omega)$-energy estimates in $H^1$]\label{prop:energy estimate in H1} Let $d \in \{1,2,3\}$ and suppose that Assumption~\ref{ass:noise_intermediate} holds. Let $p\in[1,\infty)$ and $u_0\in L_{\mathcal F_0}^{4p}(\Omega;H^1)$. Let $(u,\sigma)$ be the maximal solution of equation \eqref{eq: LLB1} given by Proposition \ref{thm:LWP in intermediate}. Then for any $T>0$,
\begin{equation}\label{ineq:H^1 estimate}
    \mathbb E
        \sup_{t\in[0,T\wedge \sigma)}\|u(t)\|_{H^1}^{2p}
     + \mathbb E \Big ( \int_0^{T\wedge\sigma}
    \|H_{\mathrm{eff}}(u(t))\|_{H^1}^2
    \,\rd t \Big)^p
    \leq C_{p,T} (1+\bbE \|u_0\|_{H^1}^{4p}),
\end{equation}
where $\Heff(u)=\alpha\Delta u+\kappa_1u-\kappa_2|u|^2u $.

\end{proposition}
\begin{remark}
    The assumption $u_0\in L_{\mathcal F_0}^{4p}(\Omega;H^1)$  can be replaced by $\bbE \mathscr E(u_0)^p<\infty$, where $\mathscr E(v)$ is given in \eqref{eq:corrected-GL-energy intro}. Moreover, using Lemma \ref{lem:d4-endpoint-elliptic} below, we are able to control the $H^3$-norm of $u$ by the $H^1$-norm of $\Heff(u)$. This will allow us to obtain the more classical energy estimate \eqref{eq:intermediate-main-energy-estimate 2}.     
\end{remark}

Before presenting the proof of Proposition \ref{prop:energy estimate in H1}, we need some preliminary results.

\begin{lemma}\label{lem:noise-trace}
Let $d \in \{1,2,3\}$ and suppose that Assumption~\ref{ass:noise_intermediate} holds. Then, for any $T>0$, there is a constant $C_T>0$ such that,
for a.e.\,$(t,\omega)\in[0,T]\times\Omega$ and all $v\in H^1$,
\begin{equation}\label{eq:noise-trace}
\|G(t,\omega,v)\|_{L^2(\ell^2)}^2
+\|\nabla G(t,\omega,v)\|_{L^2(\ell^2)}^2
+\||v|G(t,\omega,v)\|_{L^2(\ell^2)}^2
\le C_T(1+\|v\|_{H^1}^2+\|v\|_{L^4}^4).
\end{equation}
Moreover, if $v\in H^1$ is such that $\Heff(v)\in L^2$, then
\begin{equation}\label{eq:H-G-mart}
\sum_{k\ge1}|\langle \Heff(v),G_k(t,\omega,v)\rangle_{L^2}|^2
\le C_T(1+\|v\|_{L^2}^2)\|\Heff(v)\|_{L^2}^2.
\end{equation}
\end{lemma}

\begin{proof}
Let $\Phi=\|g(t,\omega,\cdot,0)\|_{\ell^2}$ and
$\Psi=\|\nabla_x g(t,\omega,\cdot,0)\|_{\ell^2}$. By
\eqref{eq:g0-assumption_intermediate},
$
\Phi\in H^1\hookrightarrow L^4$ and $
\Psi\in L^2$, 
uniformly on $[0,T]\times\Omega$. By \eqref{eq:g-Lip},
\begin{equation}\label{eq:G-L2-growth}
\|G(t,\omega,v)\|_{L^2(\ell^2)}^2
\le C_T+C\|v\|_{L^2}^2,
\end{equation}
while by the chain rule and \eqref{eq:gx-Lip}--\eqref{eq:gy-bdd},
\[
\|\nabla G(t,\omega,v)(x)\|_{\ell^2}
\le \Psi(x)+L|v(x)|+L|\nabla v(x)|,
\]
so that
\[
\|\nabla G(t,\omega,v)\|_{L^2(\ell^2)}^2
\le C_T+C\|v\|_{H^1}^2.
\]
Furthermore,
\[
\begin{aligned}
\||v|G(t,\omega,v)\|_{L^2(\ell^2)}^2
\lesssim
\||v|\Phi\|_{L^2}^2+\|v\|_{L^4}^4
\le
\|v\|_{L^4}^2\|\Phi\|_{L^4}^2+\|v\|_{L^4}^4
\le
C_T\|v\|_{H^1}^2+C\|v\|_{L^4}^4,
\end{aligned}
\]
where we used H\"older's inequality and the embedding
$H^1\hookrightarrow L^4$. This proves
\eqref{eq:noise-trace}.
Finally, by Cauchy--Schwarz's inequality and \eqref{eq:G-L2-growth},
\[
\sum_{k\ge1}|\langle \Heff(v),G_k(t,\omega,v)\rangle|^2
\le \|\Heff(v)\|_{L^2}^2\|G(t,\omega,v)\|_{L^2(\ell^2)}^2
\le C_T(1+\|v\|_{L^2}^2)\|\Heff(v)\|_{L^2}^2,
\]
and we have concluded.

\end{proof}

We now establish an $L^p(\Omega)$-estimate in $L^2$. The proof follows the same argument as in Proposition $3.3$ in \cite{GoSoTr25}, adapted for our more general noise.

\begin{lemma}[$L^p(\Omega)$-energy estimates in $L^2$]\label{lem:L2-estimate} 
Let $d\in \{1,2,3\}$ and suppose that Assumption~\ref{ass:noise_intermediate} holds. Let $p \in [1,\infty)$ and  $u_0 \in L^{2p}_{\mathcal F_0}(\Omega;H^1)$. Let  $(u,\sigma)$ be the maximal solution of equation \eqref{eq: LLB1} given by Proposition \ref{thm:LWP in intermediate}. Then, for every $p \in [1,\infty)$ and $T>0$, 
\begin{equation}\label{eq:L2-estimate}
\bbE \sup_{t\in[0,T\wedge\sigma)}\|u(t)\|_{L^2}^{2p} +
   \bbE  \Bigl(
        \int_0^{T\wedge\sigma}
        \|\Delta u(t)\|_{L^2}^2\,\rd t
    \Bigr)^p
\le C_{p,T}(1+\mathbb E\|u_0\|_{L^2}^{2p}).
\end{equation}
\end{lemma}

\begin{proof}
Let $(\sigma_n)_{n\geq1}$ be a localizing sequence for the maximal solution
$(u,\sigma)$. For $R>0$ and $n\geq1$, set
\begin{equation}\label{eq:stop time tRn}
\tau_{R,n}
 \coloneq 
T\wedge\sigma_n\wedge
\inf\left\{
0\leq t<\sigma_n:
\|u(t)\|_{H^1}\geq R
\right\},
\end{equation}
with the convention $\inf\emptyset=\infty$, and define
\[
    Y_{R,n}(t) \coloneq \|u(t\wedge\tau_{R,n})\|_{L^2}^2.
\]

We note that 
$
    u\in C([0,\tau_{R,n}];H^1)
    \cap L^2(0,\tau_{R,n}; {}_\nu H^3)
$. Applying It\^o formula in the Gelfand triple
$
     {}_\nu H^2\hookrightarrow L^2\hookrightarrow  ({}_\nu H^{2})^*
$
 gives

\[
\begin{aligned}
Y_{R,n}(t)
&=
\|u_0\|_{L^2}^2
+
2\int_0^{t\wedge\tau_{R,n}}
\left\langle
    \lambda_r\Heff(u)-\lambda_e\Delta\Heff(u)
    -\gamma u\times\Heff(u),
    u
\right\rangle_{ ({}_\nu H^{2})^*, {}_\nu H^2}\,\rd s
\\
&\quad
+
\int_0^{t\wedge\tau_{R,n}}
\|G(s,u)\|_{\mathcal L_2(\ell^2,L^2)}^2\,\rd s
+
2\sum_{k\geq1}
\int_0^{t\wedge\tau_{R,n}}
\langle u(s),G_k(s,u(s))\rangle_{L^2}\,\rd W_s^k .
\end{aligned}
\]
Since $
    \|G(s,u)\|_{\mathcal L_2(\ell^2,L^2)}^2
   =
    \|G(s,u)\|_{L^2(\ell^2)}^2
$ and 
$ u\cdot(u\times\Heff(u))=0$, 
we obtain
\[
\begin{aligned}
Y_{R,n}(t)&=\|u_0\|_{L^2}^2
+2\int_0^{t\wedge\tau_{R,n}}
\left\langle
    \lambda_r\Heff(u)-\lambda_e\Delta\Heff(u),
    u
\right\rangle_{ {}_\nu H^{-2}, {}_\nu H^2}\,\rd s
\\
&\hspace{5mm} +\int_0^{t\wedge\tau_{R,n}}\|G(s,u)\|_{L^2(\ell^2)}^2\,\rd s
+N_{R,n}(t),
\end{aligned}
\]
where
\[
    N_{R,n}(t)
     \coloneq 
    2\sum_{k\geq1}
    \int_0^{t\wedge\tau_{R,n}}
    \langle u(s),G_k(s,u(s))\rangle_{L^2}\,\rd W_s^k .
\]
Integration by parts gives
\[
\begin{aligned}
\langle u,\Heff(u)\rangle_{L^2}
&=
\alpha\langle u,\Delta u\rangle_{L^2}
+\kappa_1\|u\|_{L^2}^2
-\kappa_2\|u\|_{L^4}^4 \\
&=
-\alpha\|\nabla u\|_{L^2}^2
+\kappa_1\|u\|_{L^2}^2
-\kappa_2\|u\|_{L^4}^4
\leq C\|u\|_{L^2}^2 .
\end{aligned}
\]
Moreover,
\[
\begin{aligned}
-\langle \Delta\Heff(u),u\rangle_{( {}_\nu H^2)^*, {}_\nu H^2}
&=
-\langle \Heff(u),\Delta u\rangle_{L^2} \\
&=
-\alpha\|\Delta u\|_{L^2}^2
-\kappa_1\langle u,\Delta u\rangle_{L^2}
+\kappa_2\langle |u|^2u,\Delta u\rangle_{L^2} \\
&=
-\alpha\|\Delta u\|_{L^2}^2
+\kappa_1\|\nabla u\|_{L^2}^2
-\kappa_2\int_{\mathscr D}
\nabla(|u|^2u):\nabla u\,\rd x 
\\
& \le -\alpha\|\Delta u\|_{L^2}^2
+|\kappa_1|\|\nabla u\|_{L^2}^2,
\end{aligned}
\]
since 
\[
    \nabla(|u|^2u):\nabla u
    =
    |u|^2|\nabla u|^2
    +
    2\sum_{j=1}^d (u\cdot\partial_j u)^2
    \geq 0.
\]
Finally, by interpolation and Young's inequality,
\[
    \|\nabla u\|_{L^2}^2
    \leq
    \varepsilon\|\Delta u\|_{L^2}^2
    +
    C_\varepsilon\|u\|_{L^2}^2 .
\]
Choosing $\varepsilon>0$ sufficiently small, we get
\[
2\left\langle
    \lambda_r\Heff(u)-\lambda_e\Delta\Heff(u),
    u
\right\rangle_{( {}_\nu H^2)^*, {}_\nu H^2}
\leq
-c\|\Delta u\|_{L^2}^2
+
C(1+\|u\|_{L^2}^2).
\]
Together with \eqref{eq:G-L2-growth}, this gives
\begin{equation}\label{eq:L2-pathwise}
\begin{aligned}
&\sup_{r\leq t}Y_{R,n}(r)
+
c\int_0^{t\wedge\tau_{R,n}}\|\Delta u(s)\|_{L^2}^2\,\rd s
\\
&\qquad \leq
\|u_0\|_{L^2}^2
+
C\int_0^t
\bigl(
    1+\sup_{\rho\leq s}Y_{R,n}(\rho)
\bigr)\,\rd s
+
\sup_{r\leq t}|N_{R,n}(r)|.
\end{aligned}
\end{equation}
Raising \eqref{eq:L2-pathwise} to the power $p$, taking expectations, and using Jensen's inequality, we obtain
\begin{equation}
\begin{aligned}\label{eq:L2-power-before-BDG}
&\mathbb E\sup_{r\leq t}Y_{R,n}(r)^p
+
c\mathbb E
\Bigl(
    \int_0^{t\wedge\tau_{R,n}}
    \|\Delta u(s)\|_{L^2}^2\,\rd s
\Bigr)^p
\\
&\qquad
\leq
C_{p,T}
\Bigl(
    1+\mathbb E\|u_0\|_{L^2}^{2p}
\Bigr)
+
C_{p,T}
\int_0^t
\mathbb E\sup_{\rho\leq s}Y_{R,n}(\rho)^p\,\rd s
+
C_p\mathbb E\sup_{r\leq t}|N_{R,n}(r)|^p .
\end{aligned}
\end{equation}
By the Burkholder--Davis--Gundy inequality and Cauchy--Schwarz's inequality,
\[
\begin{aligned}
\mathbb E\sup_{r\leq t}|N_{R,n}(r)|^p
&\le C_p
\mathbb E
\Bigl(
    \int_0^{t\wedge\tau_{R,n}}
    \sum_{k\geq1}
    |\langle u(s),G_k(s,u(s))\rangle_{L^2}|^2
    \,\rd s
\Bigr)^{p/2}
\\
&\le C_p
\mathbb E
\Bigl(
    \int_0^{t\wedge\tau_{R,n}}
    \|u(s)\|_{L^2}^2
    \|G(s,u(s))\|_{L^2(\ell^2)}^2
    \,\rd s
\Bigr)^{p/2}.
\end{aligned}
\]
Using \eqref{eq:G-L2-growth}, we get
\[
\begin{aligned}
\mathbb E\sup_{r\leq t}|N_{R,n}(r)|^p
&\le C_{p,T}
\mathbb E
\Bigl[
\Bigl(1+\sup_{r\leq t}Y_{R,n}(r)
\Bigr)^{p/2}
\Bigl(
    \int_0^t
    (
        1+\sup_{\rho\leq s}Y_{R,n}(\rho)
    )\,\rd s
\Bigr)^{p/2}
\Bigr].
\end{aligned}
\]
Therefore, by Young's inequality and H\"older's inequality, for every
$\varepsilon>0$,
\[
\begin{aligned}
\mathbb E\sup_{r\leq t}|N_{R,n}(r)|^p
&\leq
\epsilon \,
\mathbb E\sup_{r\leq t}Y_{R,n}(r)^p
+
C_{p,T,\varepsilon}
\Bigl(
    1+
    \int_0^t
    \mathbb E\sup_{\rho\leq s}Y_{R,n}(\rho)^p\,\rd s
\Bigr).
\end{aligned}
\]
By choosing $\varepsilon>0$ sufficiently small and absorbing the first term into
the left-hand side of \eqref{eq:L2-power-before-BDG},
\[
\begin{aligned}
&\mathbb E\sup_{r\leq t}Y_{R,n}(r)^p
+
c\mathbb E
\Bigl(
    \int_0^{t\wedge\tau_{R,n}}
    \|\Delta u(s)\|_{L^2}^2\,\rd s
\Bigr)^p
\\
&\qquad
\leq
C_{p,T}
(
    1+\mathbb E\|u_0\|_{L^2}^{2p}
)
+
C_{p,T}
\int_0^t
\mathbb E\sup_{\rho\leq s}Y_{R,n}(\rho)^p\,\rd s .
\end{aligned}
\]
We can now use Gronwall's lemma to obtain
\begin{equation}\label{ineq:idk}
\bbE 
    \sup_{t\in[0,T]}\|u(t\wedge\tau_{R,n})\|_{L^2}^{2p}
    +
   \bbE  \Bigl(
        \int_0^{T\wedge\tau_{R,n}}
        \|\Delta u(s)\|_{L^2}^2\,\rd s
    \Bigr)^p
\leq
C_{p,T}
(
    1+\mathbb E\|u_0\|_{L^2}^{2p}
).    
\end{equation}
and letting $R\to \infty$ and then $n\to \infty$ in \eqref{ineq:idk}, together with Fatou's lemma yields \eqref{eq:L2-estimate}.

\end{proof}

Next, we establish a parabolic regularization result for $u$. This will allow us to obtain an It\^o formula for the term $\|u(t)\|_{L^4}^4$, which will then be used in the proof of Proposition \ref{prop:energy estimate in H1}. We note that in \cite[Proposition 4.9]{Cahn_Hill_crit_spaces} the authors obtained a similar It\^o formula via a regularization argument based on Yosida approximations. 

\begin{lemma}[Parabolic regularization]\label{lemma:parab reg variational}
Let $d \in \{1,2,3\}$ and suppose that Assumption~\ref{ass:noise_intermediate} holds. Let $(u,\sigma)$ be the maximal solution of equation \eqref{eq: LLB1} given by Proposition \ref{thm:LWP in intermediate}. Then,
    \begin{equation}\label{eq:parabolic reg of intermediate sol}
        u \in H^{\theta,r}_{\mathrm{loc}}(0,\sigma; {}_\nu H^{3-4\theta,q}) \, \text{ a.s.} \quad 0\le \theta<1/2, \quad r,q \in [2,\infty).
    \end{equation}
\end{lemma}
\begin{proof}
    We divide the proof into two steps.

\medskip
\noindent
\emph{Step 1: temporal regularization.}
We prove that
\begin{equation}
\label{eq:intermediate-Hilbert-regularization-subcritical}
u\in
H^{\theta,r}_{\mathrm{loc}}
(0,\sigma; {}_\nu H^{3-4\theta})
\, \text{ a.s.} 
\quad 0 \le \theta <1/2, \quad r\in[2,\infty).
\end{equation}
Evidently, it suffices to prove \eqref{eq:intermediate-Hilbert-regularization-subcritical} for $r \ge 4$. Set 
$ X_0:= {}_\nu H^{-1}$
 and $X_1:= {}_\nu H^{3}.$
Fix $\eta\in(0,1/2)$, put $\delta\coloneq \eta/4$, and consider the shifted
scale
$
Y_0:= {}_\nu H^{-1-\eta},
$ 
and
$Y_1:= {}_\nu H^{3-\eta}.
$
Note $Y_\delta=X_0$ and $
Y_1=X_{1-\delta}$. 

Set $\beta_1 \coloneq \frac{7}{12}+\delta$ and  $\beta_2 \coloneq \frac58+\delta.$
The Lipschitz estimates in the proof of
Proposition~\ref{thm:LWP in intermediate} and the embeddings  $ {}_\nu H^{-1}\hookrightarrow  {}_\nu H^{-1-\eta}$ and
$ {}_\nu H^1\hookrightarrow  {}_\nu H^{1-\eta}$, imply that 
\begin{align}
\|F_j(x)-F_j(y)\|_{Y_0}
&\lesssim
\bigl(1+\|x\|_{Y_{\beta_1}}^2
        +\|y\|_{Y_{\beta_1}}^2\bigr)
\|x-y\|_{Y_{\beta_1}},
&&j\in\{1,3\},
\label{eq:intermediate-shifted-cubic}\\
\|F_2(x)-F_2(y)\|_{Y_0}
&\lesssim
\bigl(1+\|x\|_{Y_{\beta_2}}
        +\|y\|_{Y_{\beta_2}}\bigr)
\|x-y\|_{Y_{\beta_2}},
\label{eq:intermediate-shifted-bilinear}\\
\|G(t,\omega,x)-G(t,\omega,y)\|_{\mathcal L_2(\ell^2,Y_{1/2})}
&\lesssim
\bigl(1+\|x\|_{Y_{\beta_2}}
        +\|y\|_{Y_{\beta_2}}\bigr)
\|x-y\|_{Y_{\beta_2}}.
\label{eq:intermediate-shifted-noise}
\end{align}
We apply  \cite[Proposition~6.8]{AgVe22b}. Take 
$r_0:=4$ and $
\alpha_0:=1-\eta\in(0,1).
$
Note that $A_{-1-\eta,2} \in \mathcal{SMR}^\bullet_{r_0,\alpha_0}$ in the $Y$-scale by Lemma \ref{lemma:SMR}. 
One readily checks that the assumptions of \cite[Proposition~6.8]{AgVe22b} hold in the $(Y_0,Y_1,\alpha_0,r_0)$-setting, and thus that proposition gives
\begin{equation}
\label{eq:intermediate-first-shifted-regularity}
u\in H^{\theta,4}_{\mathrm{loc}}
(0,\sigma;X_{1-\delta-\theta})
\, \text{ a.s.},
\quad 0\leq\theta<1/2.
\end{equation}
We now apply \cite[Corollary~6.5]{AgVe22b} to obtain further integrability in time. Following the notation therein, the initial setting is $(X_0,X_1,2,0)$ and the target setting is $(Y_0,Y_1,4,\alpha_0)$. Condition (1) of that corollary follows from \eqref{eq:intermediate-first-shifted-regularity} and the embedding  
\[
 Y^{\mathrm{Tr}}_4=(Y_0,Y_1)_{3/4,4}
 = {}_\nu B^{2-\eta}_{2,4}\hookrightarrow  {}_\nu H^1
 =X^{\mathrm{Tr}}_2.
\]
Given $ r \in [4,\infty)$ and $ \alpha \in [0,\frac{ r}{2}-1)$ satisfying $\frac{1+\alpha}{r} < \frac{1+\alpha_0}{4}$, set
$$ \widetilde \beta_1 \coloneq 1-\frac 34\frac{1+ \alpha}{ r}>\beta_1, \qquad \widetilde \beta_2 \coloneq 1-\frac 23 \frac{1+ \alpha}{ r}>\beta_2.$$
Then the Lipschitz estimates
\eqref{eq:intermediate-shifted-cubic}--\eqref{eq:intermediate-shifted-noise}  for $F$ and $G$  hold in the $(Y_0,Y_1, r,  \alpha)$-setting with $\beta_j$ replaced by $\widetilde \beta_j$. Hence, \cite[Corollary~6.5]{AgVe22b} gives that 
\begin{equation}
\label{eq:intermediate-shifted-all-r}
u\in H^{\theta, r}_{\mathrm{loc}}
(0,\sigma;X_{1-\delta-\theta})
\, \text{ a.s.},
\quad 0\leq\theta<1/2, \quad r \in [4,\infty).
\end{equation}
Finally, \cite[Theorem~6.3]{AgVe22b} with 
\[
 (Y_0,Y_1,r,\alpha)=(Y_0,Y_1,r,0),
 \qquad
 (\widehat Y_0,\widehat Y_1,\widehat r,\widehat\alpha)
 =(X_0,X_1,r,\widehat\kappa), \quad 
\widehat\kappa:=r\delta \in[0,\tfrac r2-1) 
\]
gives 
\begin{equation*}
u\in
H^{\theta,r}_{\mathrm{loc}}
(0,\sigma;{}X_{1-\theta})
\, \text{ a.s.},
\quad 0\leq\theta<1/2, \quad r \in [4,\infty),
\end{equation*}
which implies  \eqref{eq:intermediate-Hilbert-regularization-subcritical} since $X_{1-\theta}= {}_\nu  H^{3-4\theta}$.

\medskip
\noindent
\emph{Step 2: bootstrapping spatial integrability.}
Fix \(r\in[2,\infty)\), \(q\in(2,\infty)\). Choose an auxiliary exponent
\(R_0\geq r\), sufficiently large such that
\begin{equation*}
 \frac d4\left(\frac12-\frac1q\right)
 <\frac12-\frac1{R_0},
\end{equation*}
and then choose \(\kappa>0\) so that
\begin{equation}
\label{eq:intermediate-choice-alpha-q}
 \frac d4\left(\frac12-\frac1q\right)
 <\frac{\kappa}{R_0}
 <\frac12-\frac1{R_0}.
\end{equation}
 In particular, note that 
\(\kappa\in(0,R_0/2-1)\).  By \eqref{eq:intermediate-Hilbert-regularization-subcritical} and \cite[Proposition 2.1]{AgVe25},
\begin{equation}
\label{eq:intermediate-positive-Hilbert-trace}
 u\in L^{3R_0}_{\mathrm{loc}}(0,\sigma; {}_\nu H^3)
 \cap  C\bigl(
 (0,\sigma); {}_\nu B^{3-4/R_0}_{2,R_0}\bigr)\, \text{ a.s.}
\end{equation}
The 
inequality in \eqref{eq:intermediate-choice-alpha-q} yields the embedding
\begin{equation}
\label{eq:intermediate-Hilbert-to-weighted-q-trace}
  {}_\nu B^{3-4/R_0}_{2,R_0}
 \hookrightarrow
  {}_\nu B^{3-4(1+\kappa)/R_0}_{q,R_0}
 =(X_0^q,X_1^q)_{1-\frac{1+\kappa}{R_0},R_0},
\end{equation}
where $
 X_0^q:= {}_\nu H^{-1,q}$ and $X_1^q:= {}_\nu H^{3,q}.$
Indeed, 
\[
3-\frac4{R_0}-\frac d2
>
3-\frac{4(1+\kappa)}{R_0}-\frac dq
\quad\Longleftrightarrow\quad
\frac{4\kappa}{R_0}
>d\left(\frac12-\frac1q\right).
\]

We now use a stochastic maximal regularity technique to prove \eqref{eq:parabolic reg of intermediate sol}. By the embedding 
$ H^3\hookrightarrow W^{1,q}\cap L^\infty$ and the estimates (for $v \in {}_\nu H^3$)
\begin{align*}
\|F_1(v)\|_{H^{-1,q}} + \|F_3(v)\|_{H^{-1,q}}
&\lesssim \||v|^2v\|_{H^{1,q}}
\,\, \lesssim \|v\|_{L^\infty}^2\|v\|_{W^{1,q}}
\\
\|F_2(v)\|_{H^{-1,q}}
&\lesssim \|\nabla v\times v\|_{L^q}
\lesssim \|v\|_{L^\infty}\|\nabla v\|_{L^q}
\end{align*}
we obtain  
\begin{equation}
\label{eq:intermediate-F-H-1q-bootstrap}
 \|F(v)\|_{X_0^q}\lesssim 1+\|v\|_{H^3}^3.
\end{equation}
Furthermore, by \cite[Theorem 9.4.8]{AnalysisBspaces2}, the chain rule, and 
Assumption~\ref{ass:noise_intermediate},
\begin{align}
\|G(t,\omega,v)\|_{\gamma(\ell^2,X_{1/2}^q)}
&\eqsim
 \|g(t,\omega,\cdot,v)\|_{H^{1,q}(\ell^2)} \lesssim 
 \|g^{(0)}(t,\omega)\|_{H^{1,q}(\ell^2)}
 +\|v\|_{W^{1,q}}\notag\\
&\lesssim_{T} 1+\|v\|_{H^3}.
\label{eq:intermediate-G-H1q-bootstrap}
\end{align}
Here, $\gamma(\ell^2,X^q_{1/2})$ denotes the space of $\gamma$-radonifying operators (cf.\,\cite[Chapter 9]{AnalysisBspaces2}).

Let $0<\alpha<\beta <T<\infty$. Let
\((\sigma_n)_{n\geq1}\) be a localizing sequence for
\((u,\sigma \wedge T)\) such that $\sigma_n <\sigma \wedge T$ (cf.\,\cite[Proposition 5.4]{AgVe25}). Let   $V_n\coloneq \{\sigma_n>\alpha\}$, $u_{\alpha,n}:={\bf1}_{V_n}u(\alpha)$, and on $(\alpha,T)$ set
\begin{align*}
 f_n:={\bf1}_{V_n}{\bf1}_{[\alpha,\sigma_n)}F(u), \quad 
 g_{n}:={\bf1}_{V_n}{\bf1}_{[\alpha,\sigma_n)}G(\cdot,u).
\end{align*}
By \eqref{eq:intermediate-positive-Hilbert-trace} and \eqref{eq:intermediate-Hilbert-to-weighted-q-trace},
$
 u_{\alpha,n}\in
 L^0_{\mathcal F_\alpha}\!\bigl(
 \Omega;(X_0^q,X_1^q)_{1-\frac{1+\kappa}{R_0},R_0}\bigr).
$
Moreover, \eqref{eq:intermediate-positive-Hilbert-trace}, \eqref{eq:intermediate-F-H-1q-bootstrap} and \eqref{eq:intermediate-G-H1q-bootstrap} give that almost surely,
\begin{align*}
 f_n
 \in L^{R_0}\bigl(
 (\alpha,T),w_\alpha^\kappa;X_0^q\bigr), \quad 
 g_n
 \in L^{R_0}\bigl(
 (\alpha,T),w_\alpha^\kappa;
 \gamma(\ell^2,X_{1/2}^q)\bigr),
\end{align*}
where $w_\alpha^\kappa(t)=(t-\alpha)^\kappa$. 
Here we used that the weight $w_\alpha$ is bounded
on \([\alpha,\sigma_n]\). 

Let \(A_{-1,q}\) denote the realization of
\(A\) on $X_0^q$, and let  \(v_n\) solve
\begin{equation*}
\begin{cases}
 \rd v_n+A_{-1,q}v_n\,\rd t
 =f_n\,\rd t+g_n\,\rd W_{\ell^2}(t), \quad t \in (\alpha,T]\\
 v_n(\alpha)=u_{\alpha,n}.
\end{cases}
\end{equation*}
By Lemma \ref{lemma:SMR}, $A_{-1,q} \in \mathcal{SMR}^\bullet_{R_0,\kappa}$ and thus by \cite[Proposition 3.11]{AgVe25},
\begin{equation}
\label{eq:intermediate-weighted-q-SMR-output}
 v_n\in H^{\theta,R_0}\bigl(
 \alpha,T,w_\alpha^\kappa;X_{1-\theta}^q\bigr)
 \, \text{ a.s.} \quad 0 \le \theta<1/2.
\end{equation}
Note that $X_0^q \hookrightarrow Y_0$, $X_1^q \hookrightarrow Y_1$ and the family of operators $(A_{s,q})$ is consistent. Hence, uniqueness in the $(Y_0,Y_1)$-scale gives
\begin{equation}
\label{eq:intermediate-weighted-q-compatibility}
v_n={\bf1}_{V_n}u
 \quad\text{on }[\alpha,\sigma_n] \quad\text{a.s. on $V_n$.}
\end{equation}
Since $\beta>\alpha$, the weight $w_\alpha$ is bounded above and below
on $[\beta,T]$.  Thus
\eqref{eq:intermediate-weighted-q-SMR-output}, \eqref{eq:intermediate-weighted-q-compatibility} and the identity $X^q _{1-\theta}= {}_\nu H^{3-4\theta,q}$ yield
\[
 u\in
 H^{\theta,R_0}\bigl(
 \beta,\sigma_n; {}_\nu H^{3-4\theta,q}\bigr)
 \, \text{ on }V_n, \quad 0\le \theta<1/2.
\]
Letting \(n\to\infty\) and noting that $\beta>0$ was arbitrary gives \eqref{eq:parabolic reg of intermediate sol}.

\end{proof}

	In order to prove Proposition \ref{prop:energy estimate in H1}, we first obtain an It\^o formula for the (shifted) Ginzburg--Landau energy given by 
	\begin{equation}
		\label{eq:corrected-GL-energy intro}
		\mathscr E(v) \coloneq K+\frac\alpha2\|\nabla v\|_{L^2}^2
		+\frac{\kappa_2}{4}\|v\|_{L^4}^4
		-\frac{\kappa_1}{2}\|v\|_{L^2}^2,		\qquad v\in H^1,
	\end{equation}
	where \(K>0\) is chosen large enough that 
	\begin{equation}
		\label{eq:corrected-energy-coercive intro}
		\mathscr E(v)\ge1
		\quad \text{and} \quad
		\|v\|_{H^1}^2+\|v\|_{L^4}^4\le C\mathscr E(v).
	\end{equation}
	The lower bound above follows by noting that the polynomial $f(s) = \frac{\kappa_2}{4}s^4 - \frac{\kappa_1}{2}s^2$ has a minimum $f(s) \ge -\frac{\kappa_1^2}{4\kappa_2}$ and we are considering a bounded domain. 
	
	 While the classical It\^o formula applies for the terms $\|\nabla u(t)\|_{L^2}^2$ and $\|u(t)\|_{L^2}^2$, in order to obtain a formula for the term $\|u(t)\|_{L^4}^4$, we apply \cite[Lemma A.6]{AgVe25} (see also \cite[Lemma 4.3]{Cahn_Hill_crit_spaces}).
	
	Recall $\Heff(u) = \alpha\Delta u + \kappa_1u - \kappa_2|u|^2 u$.
	
	\begin{lemma}[It\^o formula]
		\label{lem:Ito-formula-GL-energy}
		Let $d \in \{1,2,3\}$ and suppose that Assumption~\ref{ass:noise_intermediate} holds. Let $(u,\sigma)$ be the maximal solution of equation \eqref{eq: LLB1} with initial data $u_0\in L_{\mathcal F_0}^{0}(\Omega;H^1)$ given by Proposition \ref{thm:LWP in intermediate}, and let $\tau_{R,n}$ be defined by
		\eqref{eq:stop time tRn}. Then, for all $T>0$, and $t \in [0,T]$,
		\begin{align}
			\label{eq:corrected-Ito-energy}
			\mathscr E(u(t\wedge\tau_{R,n}))
			&+\lambda_r\int_0^{t\wedge\tau_{R,n}}
			\|\Heff(u(s))\|_{L^2}^2\,\rd s
			+\lambda_e\int_0^{t\wedge\tau_{R,n}}
			\|\nabla\Heff(u(s))\|_{L^2}^2\,\rd s
			\notag\\
			&=
			\mathscr E(u_0)
			+\frac12\int_0^{t\wedge\tau_{R,n}}
			\sum_{k\ge1}
			D^2\mathscr E(u(s))
			[G_k(s,u(s)),G_k(s,u(s))]\,\rd s + M_{R,n}(t),
		\end{align}
		almost surely, where
		\begin{equation}
			\label{eq:corrected-martingale}
			M_{R,n}(t)
			\coloneq 
			-\sum_{k\ge1}\int_0^{t\wedge\tau_{R,n}}
			\langle\Heff(u(s)),G_k(s,u(s))\rangle_{L^2}\,
			\rd W_s^k.
		\end{equation}
	\end{lemma}
	
	\begin{proof}
    It is more convenient to work with equation \eqref{eq: LLB} instead of \eqref{eq: LLB1}. 
	We denote the drift term of \eqref{eq: LLB} by
		\[
		b(u)
		\coloneq
		\lambda_r\Heff(u)-\lambda_e\Delta\Heff(u)
		-\gamma u\times\Heff(u).
		\]
		We derive the It\^o formula separately for the three terms
		appearing in $\mathscr E$ and then combine all formulae. Throughout the proof, we apply 
		$$
		u \in L^2([0,\tau_{R,n});   {}_\nu  H^3)\cap  C([0,\tau_{R,n}); H^1) \, \text{ a.s.}
		$$
		without giving further reference.

		\smallskip
		\noindent
		\emph{Step 1: It\^o formula for the $L^2$-term.}
		We apply the classical variational It\^o formula
		in the Gelfand triple
		$
		{}_\nu H^2\hookrightarrow L^2
		\hookrightarrow{}_\nu H^{-2}
	$, which gives 
		\begin{align}
		\frac12\|u(t\wedge\tau_{R,n})\|_{L^2}^2
		&=
		\frac12\|u_0\|_{L^2}^2
		+\int_0^{t\wedge\tau_{R,n}}
		\langle b(u),u\rangle_{{}_\nu H^{-1},H^1}\,\rd s
		+\frac12\int_0^{t\wedge\tau_{R,n}}
		\sum_{k\ge1}\|G_k(s,u)\|_{L^2}^2\,\rd s
		\notag\\
		&\quad
		+\sum_{k\ge1}\int_0^{t\wedge\tau_{R,n}}
		\langle u,G_k(s,u)\rangle_{L^2}\,\rd W_s^k,
		\label{eq:Ito-L2-energy-component}
	\end{align}
    almost surely, for all $t \in [0,T]$.
	Indeed, by Assumption  \ref{ass:noise_intermediate}, 
	$
	G(\cdot,u)\in
	L^2\bigl(
	0,\tau_{R,n};
	\mathcal L_2(\ell^2,L^2)
	\bigr)
	$ a.s. 
	It remains to verify that  
	$b(u) \in L^2(0,T \wedge \tau_{R,n}; {}_\nu H^{-2}). $
		On $[0,\tau_{R,n}]$, the definition of $\Heff(u)$  and the
		embedding $H^1\hookrightarrow L^6$ give
		\[
		\|\Heff(u)\|_{L^2}
		\lesssim
		\|u\|_{H^2}+\|u\|_{L^6}^3+\|u\|_{L^2}
		\lesssim_R
		1+\|u\|_{H^3}
		\]
	    and thus $\Heff (u) \in L^2(0,T \wedge \tau_{R,n};L^2)$. 
		Moreover, for every $v \in\mathcal {}_\nu H^2$, the embedding ${}_\nu H^2 \hookrightarrow L^\infty$ gives 
		\begin{align*}
			\left|
			\langle u\times\Heff(u),v\rangle_{L^2}
			\right|
			&\le
			\|\Heff(u)\|_{L^2}\|u\,v\|_{L^2}
			\le
			\|\Heff(u)\|_{L^2}\|u\|_{L^2}\|v\|_{L^\infty}
			\lesssim_R
			\|\Heff(u)\|_{L^2}\|v\|_{H^2}
		\end{align*}
and thus $ u\times \Heff u \in L^2(0,T\wedge\tau_{R,n}; {}_\nu H^{-2}) $. These imply that $b(u) \in L^2(0,T \wedge\tau_{R,n}; {}_\nu H^{-2})$.

		\smallskip
		\noindent
		\emph{Step 2: It\^o formula for the $H^1$-term.}
		We next apply the variational It\^o formula in the Gelfand triple
		$
		{}_\nu H^3\hookrightarrow H^1
		\hookrightarrow {}_\nu H^{-1}$. 
We first verify its assumptions. Interpolation and ${}_\nu H^2 \hookrightarrow L^\infty$ give
\[
\|u\|_{L^\infty}^2
\lesssim
\|u\|_{H^2}^2
\lesssim
\|u\|_{H^1}\|u\|_{H^3}.
\]
It follows that, on $[0,\tau_{R,n}]$,
\begin{align*}
	\||u|^2u\|_{H^1}
	&\lesssim
	\|u\|_{L^\infty}^2\|u\|_{H^1}
	\lesssim_R
	\|u\|_{H^3}
\end{align*}
and thus 
$\Heff(u) \in L^2(0,T \wedge\tau_{R,n};H^1).
$
Moreover, using $H^1\hookrightarrow L^6 \hookrightarrow L^3$,
$$
\|u\times\Heff(u)\|_{L^2}
\lesssim
\|u\|_{L^6}\|\Heff(u)\|_{L^3}
\lesssim
\|u\|_{H^1}\|\Heff(u)\|_{H^1}.
$$
So,
$
b(u) \in L^2(0,T \wedge\tau_{R,n};{}_\nu H^{-1})
$. 
Assumption~\ref{ass:noise_intermediate} gives 
$
G(\cdot,u)\in
L^2\bigl(
0,T \wedge\tau_{R,n};
\mathcal L_2(\ell^2,H^1)
\bigr)$ a.s. 
Hence, the assumptions of It\^o's formula are satisfied. Since the duality in this Gelfand triple is induced by the $H^1$ inner product, we get, almost surely, for all $t \in [0,T]$,
		\begin{align*}
			\frac12\|u(t\wedge\tau_{R,n})\|_{H^1}^2
			&=
			\frac12\|u_0\|_{H^1}^2
			+\int_0^{t\wedge\tau_{R,n}}
			\langle b(u),u-\Delta u\rangle_{{}_\nu H^{-1}, H^1}\,\rd s
			+\frac12\int_0^{t\wedge\tau_{R,n}}
			\sum_{k\ge1}\|G_k(s,u)\|_{H^1}^2\,\rd s
			\\
			&\quad
			+\sum_{k\ge1}\int_0^{t\wedge\tau_{R,n}}
			(u,G_k(s,u))_{H^1}\,\rd W_s^k.
		\end{align*}
		Subtracting \eqref{eq:Ito-L2-energy-component} and using
		$
		\|v\|_{H^1}^2-\|v\|_{L^2}^2
		=\|\nabla v\|_{L^2}^2
		$
		gives, almost surely,
		\begin{align}
			\frac12\|\nabla u(t\wedge\tau_{R,n})\|_{L^2}^2
			&=
			\frac12\|\nabla u_0\|_{L^2}^2
			+\int_0^{t\wedge\tau_{R,n}}
		\langle b(u),-\Delta u\rangle_{{}_\nu H^{-1},H^1}\,\rd s
			\notag\\
			&\quad 	+\frac12\int_0^{t\wedge\tau_{R,n}}
			\sum_{k\ge1}\|\nabla G_k(s,u)\|_{L^2}^2\,\rd s
			+\sum_{k\ge1}\int_0^{t\wedge\tau_{R,n}}
			\langle\nabla u,\nabla G_k(s,u)\rangle_{L^2}\,
			\rd W_s^k.
			\label{eq:Ito-gradient-energy-component}
		\end{align}
		
		\smallskip
		\noindent
		\emph{Step 3: It\^o formula for the \(L^4\)-term.}
		We prove that, almost surely, for all $t \in [0,T]$,  \label{eq:Ito-L4-energy}
			\begin{align}
				\frac14\|u(t\wedge\tau_{R,n})\|_{L^4}^4 \nonumber
				&=
				\frac14\|u_0\|_{L^4}^4
				+\lambda_r\int_0^{t\wedge\tau_{R,n}}
				\int_{\mathscr D}|u|^2u\cdot \Heff(u)\,\rd x\,\rd s 
				\\
				&\quad +\lambda_e\sum_{j=1}^d
				\int_0^{t\wedge\tau_{R,n}}\int_{\mathscr D}
				\partial_j (|u|^2u) \cdot\partial_j \Heff(u)\,\rd x\,\rd s \nonumber
				\\ 
				&\quad
				+\frac12\int_0^{t\wedge\tau_{R,n}}
				\sum_{k\ge1}\int_{\mathscr D}
				\left[
				|u|^2|G_k(s,u)|^2
				+2(u\cdot G_k(s,u))^2
				\right]\,\rd x\,\rd s
				\\
				&\quad
				+\sum_{k\ge1}\int_0^{t\wedge\tau_{R,n}}
				\int_{\mathscr D}|u|^2u\cdot G_k(s,u)\,\rd x\, \nonumber
				\rd W_s^k.	
			\end{align}
		
		We apply the  $\mathbb R^3$-valued counterpart of
		\cite[Lemma~A.6]{AgVe25} (see also \cite[Lemma 4.3]{Cahn_Hill_crit_spaces}). We write \eqref{eq: LLB} as
		\[
		\rd u
		=
		\bigl(\varphi+\operatorname{div}\Phi\bigr)\,\rd t
		+\sum_{k\ge1}\psi_k\,\rd W^k_t,
		\]
		where
	\begin{equation}\label{eq:def of phi, Phi and psi}
		\varphi
		\coloneq
		\lambda_r\Heff(u)-\gamma u\times\Heff(u),
		\qquad
		\Phi_j\coloneq-\lambda_e\partial_j\Heff(u),
		\qquad
		\psi_k\coloneq G_k(\cdot,u).
	\end{equation}

		Let \(0<\delta<T\). By \eqref{eq:parabolic reg of intermediate sol} with
		\((r,q)=(6,4)\) and the trace embedding of \cite[Proposition 2.1]{AgVe25}, almost surely,
		$$
		u\in
		L^6(\delta,T \wedge\tau_{R,n};H^{3,4})
	 \cap C([\delta,T \wedge\tau_{R,n}]; {}_\nu H^{2,4}).$$
		The definition of $\Heff(u)$ and ${}_\nu H^{2,4}\hookrightarrow W^{1,\infty}$ give  
		$
		\Heff(u)\in L^6(\delta,T \wedge\tau_{R,n};H^{1,4}).
		$ 
		It follows that
		\[
		\varphi\in L^1(\delta,T \wedge\tau_{R,n};L^4),
		\qquad
		\Phi\in L^2(\delta,T \wedge\tau_{R,n};L^4),
		\]
		while Assumption~\ref{ass:noise_intermediate} gives
		$
		\psi\in
		L^2(0,T \wedge\tau_{R,n};L^4(\ell^2)).
		$
		Therefore, the integrability assumptions of the \(\mathbb R^3\)-valued counterpart of
		\cite[Lemma~A.6]{AgVe25} are satisfied, from which we get, with $q=4$,
		\begin{equation}			\label{eq:vector-valued-Lq-Ito}
		\begin{aligned}
			\frac 1q \|u(t\wedge \tau_{R,n})\|_{L^q}^q
			&=
		 \frac 1q	\|u(\delta)\|_{L^q}^q
			+\int_\delta ^{t\wedge\tau_{R,n}}\int_{\mathscr D}
			|u|^{q-2}u\cdot\varphi\,\rd x\,\rd s
			\\
			& \quad 
			-\sum_{j=1}^d \int_\delta^{t\wedge\tau_{R,n}} \int_{\mathscr D}
			\partial_j(|u|^{q-2} u) \cdot\Phi_j\,\rd x\,\rd s
\\
			&\quad
			+\sum_{k\ge1}\int_\delta ^{t\wedge\tau_{R,n}}\int_{\mathscr D}
			|u|^{q-2}u\cdot\psi_k\,\rd x\,\rd W_s^k
			\\
			& \quad  
			+\frac 12\int_\delta ^{t\wedge\tau_{R,n}}\sum_{k\ge1}\int_{\mathscr D}
			|u|^{q-2}|\psi_k|^2\,\rd x\,\rd s
\\
			&\quad
			+\frac{(q-2)}2
			\int_\delta ^{t\wedge\tau_{R,n}}\sum_{k\ge1}\int_{\mathscr D}
			|u|^{q-4}(u\cdot\psi_k)^2\,\rd x\,\rd s,
		\end{aligned}
	\end{equation}
	almost surely on $\{\tau_{R,n} > \delta \}$, for all $t \in [\delta,T]$.
		We note
		\begin{equation}
			\label{eq:L4-phi-term}
			\int_\delta ^{t\wedge\tau_{R,n}}\int_{\mathscr D}
		|u|^2u\cdot\varphi\,\rd x\,\rd s=	\lambda_r\int_\delta ^{t\wedge\tau_{R,n}}
			\int_{\mathscr D}|u|^2u\cdot \Heff(u)\,\rd x\,\rd s,
		\end{equation}
		since
		$u\cdot(u\times \Heff(u))=0$ pointwise, and thus 
		\[
		|u|^2u\cdot\varphi
		=
		|u|^2u\cdot(\lambda_r\Heff(u)-\gamma u\times \Heff(u))
		=
		\lambda_r|u|^2u\cdot \Heff(u).
		\]
Therefore, \eqref{eq:vector-valued-Lq-Ito}, \eqref{eq:L4-phi-term} and the definitions of  $\varphi, \Phi$ and $\psi$ in \eqref{eq:def of phi, Phi and psi} give 
		\begin{equation}		\label{eq:Ito-L4-energy-positive-time}
			\begin{aligned}
			\frac14\|u(t\wedge\tau_{R,n})\|_{L^4}^4
			&=
			\frac14\|u(\delta)\|_{L^4}^4
			+\lambda_r\int_\delta^{t\wedge\tau_{R,n}}
			\int_{\mathscr D}|u|^2u\cdot \Heff(u)\,\rd x\,\rd s
			\\
			&\quad +\lambda_e\sum_{j=1}^d
						\int_\delta^{t\wedge\tau_{R,n}}\int_{\mathscr D}
						\partial_j (|u|^2u) \cdot\partial_j \Heff(u)\,\rd x\,\rd s
			\\
			&\quad
			+\frac12\int_\delta^{t\wedge\tau_{R,n}}
			\sum_{k\ge1}\int_{\mathscr D}
			\left[
			|u|^2|G_k(s,u)|^2
			+2(u\cdot G_k(s,u))^2
			\right]\,\rd x\,\rd s
			\notag\\
			&\quad
			+\sum_{k\ge1}\int_\delta^{t\wedge\tau_{R,n}}
			\int_{\mathscr D}|u|^2u\cdot G_k(s,u)\,\rd x\,
			\rd W_s^k.	
		\end{aligned}
	\end{equation}
		It remains to let \(\delta\downarrow0\) in \eqref{eq:Ito-L4-energy-positive-time} and obtain \eqref{eq:Ito-L4-energy}. Since
		$
		u\in C([0,\sigma);H^1)$ and $H^1 \hookrightarrow L^4,$
		we have
		$
		\|u(\delta)\|_{L^4}^4\longrightarrow\|u_0\|_{L^4}^4.
		$
		Moreover, by the calculations of Step~2, both 
		$	
		\Heff(u)$ and $|u|^2u$ belong to $L^2(0,T \wedge\tau_{R,n};H^1)$, and thus we can apply the dominated convergence theorem for the two terms containing $\Heff(u)$. 
		Similarly, Hölder's inequality gives
		\[
			\sum_{k\ge1}\int_{\mathscr D}
			\left[
			|u|^2|G_k|^2+2(u\cdot G_k)^2
			\right]\,\rd x
			\lesssim
			\|u\|_{L^6}^2
			\|G(\cdot,u)\|_{L^3(\ell^2)}^2
			\lesssim_R
			\|G(\cdot,u)\|_{H^1(\ell^2)}^2,
		\]
        which is integrable.
		Moreover, 
		\[
			\sum_{k\ge1}
			\Bigl|
			\int_{\mathscr D}|u|^2u\cdot G_k(s,u)\,\rd x
			\Bigr|^2
			\le
			\||u|^2u\|_{L^2}^2
			\|G(s,u)\|_{L^2(\ell^2)}^2
			\lesssim_R
			\|G(s,u)\|_{L^2(\ell^2)}^2,
		\]
		and thus the stochastic term is a continuous local martingale.

		\medskip
		\noindent
		\emph{Step 4: combination of all the formulae.}
		We now combine \eqref{eq:Ito-L2-energy-component}--\eqref{eq:Ito-L4-energy}. Recalling the definition of 
		$
		\mathscr E(u)
		$ in \eqref{eq:corrected-GL-energy intro},
		the terms on the left-hand side combine into
		\[
		\mathscr E(u(t\wedge\tau_{R,n})),
		\]
		and the initial-value terms combine into \(\mathscr E(u_0)\).
		
		We next combine the drift terms. Recalling
		$
		b(u)=\lambda_r \Heff(u)-\lambda_e\Delta \Heff(u)-\gamma u\times \Heff(u),
		$
		and noting that $(u\times \Heff(u))\cdot u=0$, we get 
			\begin{align*}
			\big\langle b(u),|u|^2u\big\rangle_{{}_\nu H^{-1},H^1}
			&=
			\lambda_r\int_{\mathscr D}|u|^2u\cdot \Heff(u)\,\rd x
			+\lambda_e\sum_{j=1}^d
			\int_{\mathscr D}
			\partial_j(|u|^2u)\cdot\partial_j \Heff(u)\,\rd x,
		\end{align*}
		which is precisely the drift contribution from the \(L^4\)-identity in \eqref{eq:Ito-L4-energy}. 
		It follows that the sum of the three drift contributions is
		\begin{align*}
			&-\kappa_1\langle b(u),u\rangle_{{}_\nu H^{-1},H^1}
			+\alpha\langle b(u),-\Delta u\rangle_{{}_\nu H^{-1},H^1}
			+\kappa_2\langle b(u),|u|^2u\rangle_{{}_\nu H^{-1},H^1}\\
			&\qquad
			=
			\left\langle
			b(u),-\Heff(u)
			\right\rangle_{{}_\nu H^{-1},H^1}.
		\end{align*}
		Using the definition of \(b(u)\), we calculate
		\begin{align*}
			\langle b(u),-\Heff(u)\rangle_{{}_\nu H^{-1},H^1}
			&=
			-\lambda_r\|\Heff(u)\|_{L^2}^2
			+\lambda_e\langle\Delta \Heff(u),\Heff(u) \rangle_{{}_\nu H^{-1},H^1}
			\\
			&=
			-\lambda_r\|\Heff(u)\|_{L^2}^2
			-\lambda_e\|\nabla \Heff(u)\|_{L^2}^2.
		\end{align*}

		We next combine the It\^o-correction terms. By \eqref{eq:corrected-second-derivative}, 
        they give
		\begin{align*}
			&-\frac{\kappa_1}{2}
			\sum_{k\ge1}\|G_k(s,u)\|_{L^2}^2
			+\frac{\alpha}{2}
			\sum_{k\ge1}\|\nabla G_k(s,u)\|_{L^2}^2
			+\frac{\kappa_2}{2}
			\sum_{k\ge1}\int_{\mathscr D}
			\left[
			|u|^2|G_k(s,u)|^2
			+2(u\cdot G_k(s,u))^2
			\right]\,\rd x\\
			&\qquad
			=
			\frac12\sum_{k\ge1}
			D^2\mathscr E(u)
			[G_k(s,u),G_k(s,u)].
		\end{align*}
		
		Finally, the three stochastic-integral terms combine into
		\begin{align*}
		M_{R,n}(t) &=\sum_{k\ge1}\int_0^{t\wedge\tau_{R,n}}
			\Bigg[
			-\kappa_1\langle u,G_k(s,u)\rangle_{L^2}
			+\alpha\langle\nabla u,\nabla G_k(s,u)\rangle_{L^2}
			+\kappa_2\int_{\mathscr D}
			|u|^2u\cdot G_k(s,u)\,\rd x
			\Bigg]\,\rd W_s^k
			\\
			&=	
			-\sum_{k\ge1}\int_0^{t\wedge\tau_{R,n}}
			\langle\Heff(u(s)),G_k(s,u(s))\rangle_{L^2}\,
			\rd W_s^k.
		\end{align*}
	\end{proof}

We now present the proof of Proposition \ref{prop:energy estimate in H1}. 

\begin{proof}[Proof of Proposition \ref{prop:energy estimate in H1}]
Let \((\sigma_n)_{n\ge1}\) be a localizing sequence for \((u,\sigma)\). For \(R>0\) and \(n\ge1\), let $\tau_{R,n}$ be as defined in \eqref{eq:stop time tRn}. Note that 
$\tau_{R,n}\uparrow T\wedge\sigma_n $ as  $R\to\infty .$
Let $\mathscr E(v)$ be as in \eqref{eq:corrected-GL-energy intro}.
Note that the assumption \(u_0\in L^{4p}_{\mathcal F_0}(\Omega;H^1)\) implies \(\bbE\mathscr E(u_0)^p<\infty\). 
By Lemma \ref{lem:Ito-formula-GL-energy} we get
\begin{align}
\label{eq:corrected-Ito-energy after}
\mathscr E(u(t\wedge\tau_{R,n}))
&+\lambda_r\int_0^{t\wedge\tau_{R,n}}\|\Heff(u(s))\|_{L^2}^2\,\rd s
+\lambda_e\int_0^{t\wedge\tau_{R,n}}\|\nabla\Heff(u(s))\|_{L^2}^2\,\rd s\notag\\
&=\mathscr E(u_0)
+\frac12\int_0^{t\wedge\tau_{R,n}}
\sum_{k\ge1}D^2\mathscr E(u(s))[G_k(s,u(s)),G_k(s,u(s))]  \,\rd s
+M_{R,n}(t),
\end{align}
where
\begin{equation}
\label{eq:corrected-second-derivative}
D^2\mathscr E(v)[\phi,\phi]
=\alpha\|\nabla\phi\|_{L^2}^2
+\kappa_2\int_{\mathscr D}\bigl(|v|^2|\phi|^2+2(v\cdot\phi)^2\bigr)\,\rd x
-\kappa_1\|\phi\|_{L^2}^2.
\end{equation}
and
\begin{equation*}
M_{R,n}(t)=-\sum_{k\ge1}\int_0^{t\wedge\tau_{R,n}}
\langle \Heff(u(s)),G_k(s,u(s))\rangle_{L^2}\,\rd W_s^k .
\end{equation*}
By \eqref{eq:corrected-second-derivative}, \eqref{eq:noise-trace} and \eqref{eq:corrected-energy-coercive intro},
\begin{equation}
\label{eq:corrected-Ito-correction-bound}
\frac12\sum_{k\ge1}D^2\mathscr E(u)[G_k(u),G_k(u)]
\le C_T\bigl(1+\|u\|_{H^1}^2+\|u\|_{L^4}^4\bigr)
\le C_T\mathscr E(u).
\end{equation}
Setting
\begin{equation}
\label{eq:corrected-D-R-n}
D_{R,n}(t)=\int_0^{t\wedge\tau_{R,n}}
\left(\lambda_r\|\Heff(u(s))\|_{L^2}^2
+\lambda_e\|\nabla\Heff(u(s))\|_{L^2}^2\right)\,\rd s,
\end{equation}
taking the supremum in \eqref{eq:corrected-Ito-energy after}, using \eqref{eq:corrected-Ito-correction-bound}, and the fact that \(D_{R,n}\) is increasing gives, for \(t\in[0,T]\),
\begin{equation}
\label{eq:corrected-pathwise-energy}
\sup_{r\le t}\mathscr E(u(r\wedge\tau_{R,n}))+D_{R,n}(t)
\le C\Big(\mathscr E(u_0)+\int_0^t\sup_{\rho\le s}\mathscr E(u(\rho\wedge\tau_{R,n}))\,\rd s
+\sup_{r\le t}|M_{R,n}(r)|\Big).
\end{equation}
Raising \eqref{eq:corrected-pathwise-energy} to the power \(p\), taking expectations, and using Jensen's inequality yields
\begin{align}
\label{eq:corrected-energy-before-BDG}
&\bbE\sup_{r\le t}\mathscr E(u(r\wedge\tau_{R,n}))^p+\bbE D_{R,n}(t)^p\notag\\
&\quad\le C_p\bbE\mathscr E(u_0)^p
+C_{p,T}\int_0^t\bbE\sup_{\rho\le s}\mathscr E(u(\rho\wedge\tau_{R,n}))^p\,\rd s
+C_p\bbE\sup_{r\le t}|M_{R,n}(r)|^p .
\end{align}
We estimate the martingale term by the Burkholder--Davis--Gundy inequality, \eqref{eq:H-G-mart}, and Young's inequality as follows:
\begin{align}
\bbE\sup_{r\le t}|M_{R,n}(r)|^p
&\le C_p  \bbE\Bigl(\int_0^{t\wedge\tau_{R,n}}
\sum_{k\ge1}|\langle\Heff(u),G_k(s,u)\rangle_{L^2}|^2\,\rd s\Bigr)^{p/2}\notag\\
&\le C_{p,T}\bbE\Bigr[(1+\sup_{s\le t\wedge\tau_{R,n}}\|u(s)\|_{L^2}^2)^{p/2}D_{R,n}(t)^{p/2}\Bigr]\notag\\
&\le \epsilon\,\bbE D_{R,n}(t)^p
+C_{p,T,\varepsilon}\bbE\Bigl[1+
\sup_{s\le t\wedge\tau_{R,n}}\|u(s)\|_{L^2}^{2p}\Bigr].
\label{eq:corrected-H1-BDG}
\end{align}
Choosing \(\varepsilon>0\) small enough we can absorb the first term on the right-hand side of \eqref{eq:corrected-H1-BDG} into the left-hand side of \eqref{eq:corrected-energy-before-BDG}. By the $L^p(\Omega)$-energy estimate in $L^2$ (Lemma \ref{lem:L2-estimate}),
\begin{multline*}
\bbE\sup_{r\le t}
\mathscr E(u(r\wedge\tau_{R,n}))^p
+\bbE D_{R,n}(t)^p\\
\le
C_{p,T}\left(
\bbE\mathscr E(u_0)^p
+1+\bbE\|u_0\|_{L^2}^{2p}
\right)
+C_{p,T}\int_0^t
\bbE\sup_{\rho\le s}
\mathscr E(u(\rho\wedge\tau_{R,n}))^p\,\rd s,
\end{multline*}
and since $\mathscr E(u_0)\ge1$ and
$\|u_0\|_{L^2}^2\le C\mathscr E(u_0)$, this reduces to
\begin{equation}
\bbE\sup_{r\le t}
\mathscr E(u(r\wedge\tau_{R,n}))^p
+\bbE D_{R,n}(t)^p
\le
C_{p,T}\bbE\mathscr E(u_0)^p
+C_{p,T}\int_0^t
\bbE\sup_{\rho\le s}
\mathscr E(u(\rho\wedge\tau_{R,n}))^p\,\rd s.
\label{eq:corrected-energy-gronwall-ready}
\end{equation}
Gronwall's lemma now gives
\begin{equation}
\label{eq:corrected-stopped-energy-estimate}
\begin{aligned}
\bbE\sup_{t\in[0,T]}\mathscr E(u(t\wedge\tau_{R,n}))^p
+\bbE D_{R,n}(T)^p
& \le C_{p,T} \bbE\mathscr E(u_0)^p \le C_{p,T} (1+\bbE \|u_0\|_{H^1}^{4p}),
\end{aligned}
\end{equation}
so that after letting \(R\to\infty\) and then $n \to \infty$ in \eqref{eq:corrected-stopped-energy-estimate}, together with Fatou's lemma and \eqref{eq:corrected-energy-coercive intro}, we obtain the desired energy estimate \eqref{ineq:H^1 estimate}.

\end{proof}

\begin{remark}[Weaker integrability of $G$ at the origin]
\label{rem:intermediate-weaker-origin}
For the proof of Proposition \ref{prop:energy estimate in H1}, the integrability assumption of  $g^{(0)}$ in
Assumption~\ref{ass:noise_intermediate} (see  \eqref{eq:g0-assumption_intermediate}) can be weakened to
\begin{equation}\label{eq:intermediate-weaker-origin}
    g^{(0)}\in
L^\infty((0,T)\times\Omega;W^{1,4}(\ell^2)),
\qquad T>0.
\end{equation}
Indeed, the proof of Lemma~\ref{lemma:parab reg variational}
still yields
\[
u\in H^{\theta,r}_{\mathrm{loc}}
((0,\sigma);{}_\nu H^{3-4\theta,q})
\quad\text{a.s.},
\qquad
0\le\theta<\tfrac12,\quad
r\in[2,\infty),\quad q\in[2,4].
\]
In particular, taking $(r,q)=(6,4)$ provides the regularity needed to
justify the It\^o formula in Lemma \ref{lem:Ito-formula-GL-energy}.
\end{remark}

We now need the following Lemma to control the $H^3$-norm of $u$ by the $H^1$-norm of $\Heff(u)$.

\begin{lemma}[Elliptic estimate]
\label{lem:d4-endpoint-elliptic}
Let $d\in \{1,2,3\}$. Then there exists \(C>0\) such that, for every
\(v\in   {}_\nu  H^3\),
\begin{equation}
\label{eq:d4-endpoint-elliptic-explicit1}
    \|v\|_{H^3}^2
    \leq
    C(1+\|v\|_{H^1}^4)
    +
    C(1+\|v\|_{H^1}^2)
    \|\Heff(v)\|_{H^1}^2.
\end{equation}
\end{lemma}

\begin{proof}
By elliptic regularity,
and the identity
$
\alpha\Delta v
=
\Heff(v)-\kappa_1v+\kappa_2|v|^2v
$
we infer 
\[
\|v\|_{H^3}
\leq
C\left(
\|v\|_{L^2}
+\|\Heff(v)\|_{H^1}
+\|v\|_{H^1}
+\||v|^2v\|_{H^1}
\right),
\]
and by interpolation, 
\[
\|v\|_{H^1}
\leq
\varepsilon\|v\|_{H^3}
+C_\varepsilon\|v\|_{L^2},
\]
so that by absorbing the
$H^3$-term into the left-hand side, we conclude that
\begin{equation}\label{ineq:ell reg for Heff}
\|v\|_{H^3}
\leq
C\left(
\|v\|_{L^2}
+\|\Heff(v)\|_{H^1}
+\||v|^2v\|_{H^1}
\right).
\end{equation}
Hence, it suffices to show that 
\begin{equation}
\label{eq:d4-cubic-H1-final}
    \||v|^2v\|_{H^1}^2
    \le C
    \|v\|_{H^1}^4
+C \|v\|_{H^1}^2\|\Heff(v)\|_{H^1}^2,
\end{equation}
since by \eqref{eq:d4-cubic-H1-final},  
\eqref{ineq:ell reg for Heff}, and the embedding $H^1\hookrightarrow L^2$, we obtain
\begin{align*}
    \|v\|_{H^3}^2
    &\le C(
    \|v\|_{H^1}^2
    +\|\Heff(v)\|_{H^1}^2
    +\|v\|_{H^1}^4 
   +\|v\|_{H^1}^2\|\Heff(v)\|_{H^1}^2)
   \\
   &\le C (1+\|v\|_{H^1}^4)
    +
    C(1+\|v\|_{H^1}^2)
    \|\Heff(v)\|_{H^1}^2,
\end{align*}
where the last step follows from $
    \|v\|_{H^1}^2\leq 1+\|v\|_{H^1}^4
$.

Let us then prove \eqref{eq:d4-cubic-H1-final}. 
Set
\begin{equation}\label{eq:def of f with Heff}
      f \coloneq \kappa_1v-\Heff(v)
      =-\alpha\Delta v+\kappa_2|v|^2v.
\end{equation}
Testing this equation with \(|v|^4v\), and using
\(\partial_n v=0\), gives
\begin{align*}
    \alpha\int_{\mathscr D}
    \Big(
        |v|^4|\nabla v|^2
        +4|v|^2\sum_{j=1}^d
          (v\cdot\partial_jv)^2
    \Big)\,{\rm d}x
    +\kappa_2\|v\|_{L^8}^8  
    =\int_{\mathscr D} f\cdot |v|^4v\, {\rm d}x,
\end{align*}
while interpolation between \(L^4\) and \(L^8\),
\[
    \|v\|_{L^{20/3}}
    \leq
    \|v\|_{L^4}^{1/5}\|v\|_{L^8}^{4/5}.
\]
Consequently, H\"older's and Young's inequalities imply
\begin{align*}
    \Big|
    \int_{\mathscr D}f\cdot |v|^4v\,{\rm d}x
    \Big|
    \leq
    \|f\|_{L^4}\|v\|_{L^{20/3}}^5 
    \leq
    \|f\|_{L^4}\|v\|_{L^4}\|v\|_{L^8}^4 
    \leq
    \frac{\kappa_2}{2}\|v\|_{L^8}^8
    +C\|f\|_{L^4}^2\|v\|_{L^4}^2.
\end{align*}
After absorbing the $\|v\|_{L^8}^8$-term, we obtain
\begin{equation}
\label{eq:d4-quintic-test1}
    \int_{\mathscr D}|v|^4|\nabla v|^2\,{\rm d}x
    +\|v\|_{L^8}^8
    \leq
    C\|f\|_{L^4}^2\|v\|_{L^4}^2.
\end{equation}
Since
$
    |\nabla(|v|^2v)|
    \leq 3|v|^2|\nabla v|,
$
\begin{equation}\label{ineq:gradient term for Heff lemma}
     \|\nabla(|v|^2v)\|_{L^2}^2
    \le C \int_{\mathscr D}|v|^4|\nabla v|^2\,{\rm d}x.
\end{equation}
Moreover, by interpolation $\|v\|_{L^6}\le C \|v\|_{L^4}^{1/3} \|v\|_{L^8}^{2/3}$ and therefore 
\begin{equation}\label{ineq:estimating L2 term Heff lemma}
      \||v|^2v\|_{L^2}^2
    =\|v\|_{L^6}^6
    \le C 
    \|v\|_{L^4}^2\|v\|_{L^8}^4.
\end{equation}
Consequently, \eqref{ineq:estimating L2 term Heff lemma}, \eqref{ineq:gradient term for Heff lemma} and \eqref{eq:d4-quintic-test1} give
\begin{equation}\label{ineq:idkk1}
    \||v|^2v\|_{H^1}^2\le C \|f\|_{L^4}^2\|v\|_{L^4}^2  +C\|v\|_{L^4}^3 \|f\|_{L^4}.
\end{equation}
By Young's inequality and the Sobolev embedding $H^1\hookrightarrow L^4$, 
$$\|v\|_{L^4}^3 \|f\|_{L^4} \le C \|v\|_{L^4}^4+C\|v\|_{L^4}^2\|f\|_{L^4}^2 \le C \|v\|_{H^1}^4+C\|v\|_{H^1}^2\|f\|_{H^1}^2  .$$
Also, by \eqref{eq:def of f with Heff} there holds 
$$ \|f\|_{H^1}\le C (\|v\|_{H^1}+ \|\Heff(v)\|_{H^1}) $$
and thus 
\begin{equation}\label{ineq:idkk2}
    \|v\|_{L^4}^3 \|f\|_{L^4} \le C \|v\|_{H^1}^4 + C \|v\|_{H^1}^2\|\Heff (v)\|_{H^1}^2.
\end{equation}
Finally, \eqref{ineq:idkk1} and \eqref{ineq:idkk2} imply \eqref{eq:d4-cubic-H1-final} and we are done. 

\end{proof}

With the energy estimates at hand, we are now ready to prove the global well-posedness of \eqref{eq: LLB1} in the intermediate setting for $d \leq 3$, as well as the properties of the solution stated in Theorem \ref{theorem:GWP_intermediate}.

\begin{proof}[Proof of Theorem \ref{theorem:GWP_intermediate}]
    By a localization argument (cf.\,\cite[Remark 5.3]{AgVe25}), in order to prove global existence, it suffices to consider $u_0 \in L^\infty_{\mathcal F_0}(\Omega;H^1)$. Hence, we can apply Proposition \ref{prop:energy estimate in H1} with $p=1$, which together with the blow up criterion \eqref{eq:blow up criterion in intermediate} of Proposition \ref{thm:LWP in intermediate}, implies  
    the global existence of equation \eqref{eq: LLB1}. The instantaneous regularization \eqref{eq:parabolic reg of intermediate sol0} follows from \eqref{eq:parabolic reg of intermediate sol}.

    We now establish the energy estimate \eqref{eq:intermediate-main-energy-estimate 2}, for which it suffices to show that for every $T>0$,
\begin{equation}\label{ineq:integral H3 estimate in intermediate}
    \mathbb E\int_0^{T}
    \|u(t)\|_{H^3}^2
\,\rd t \le C_{T} (1+\bbE \|u_0\|_{H^1}^{8}).
\end{equation} 
By Lemma \ref{lem:d4-endpoint-elliptic}, for almost every $t\in (0,T]$,
\begin{equation}
\label{eq:d4-endpoint-elliptic-explicit2}
    \|u(t)\|_{H^3}^2
    \leq
    C(1+\|u(t)\|_{H^1}^4)
    +
    C(1+\|u(t)\|_{H^1}^2)
    \|\Heff(u(t))\|_{H^1}^2.
\end{equation}
Setting
\[
X_T \coloneq \sup_{t\in[0,T)}\|u(t)\|_{H^1},
\qquad
Y_T \coloneq \int_0^{T}\|\Heff(u(t))\|_{H^1}^2\,\rd t,
\]
we have that \eqref{eq:d4-endpoint-elliptic-explicit2} implies 
\[
\int_0^{T}\|u(t)\|_{H^3}^2\,\rd t
\le C_T (1+X_T^4)+C(1+X_T^2)Y_T.
\]
Therefore, by Cauchy--Schwarz's inequality,
\[
\begin{aligned}
\mathbb E\int_0^{T}\|u(t)\|_{H^3}^2\,\rd t
&\le C_T (1+\mathbb EX_T^4)
+C\mathbb EY_T+C \mathbb E(X_T^2Y_T) \\
&\le C_T (1+\mathbb EX_T^4)
+C (\mathbb EY_T^2)^{1/2} +C
(\mathbb EX_T^4)^{1/2}
(\mathbb EY_T^2)^{1/2}.
\end{aligned}
\]
Thus, taking $p=2$ in Proposition~\ref{prop:energy estimate in H1} gives 
\[
\mathbb EX_T^4+\mathbb EY_T^2
\leq
C_T\left(1+\mathbb E\|u_0\|_{H^1}^8\right).
\]
Consequently, by Young's inequality, we obtain
\[
\mathbb E\int_0^{T}\|u(t)\|_{H^3}^2\,\rd t
\leq
C_T\left(1+\mathbb E\|u_0\|_{H^1}^8\right),
\]
which is precisely \eqref{ineq:integral H3 estimate in intermediate}.

\end{proof}

\begin{proof}[Proof of Corollary \ref{Coroll:GWP_intermediate}]
     Following the proof of \cite[Proposition 4.5]{AgVe24a} it suffices to consider $u_0 \in L^\infty_{\mathcal F_0}(\Omega ;H^1)$ and to have the linear coercivity \eqref{ineq:coerc_A}, the Lipschitz estimates \eqref{ineq:Local_Lip_F}--\eqref{ineq:Local_Lip_G}, and an energy estimate of the form \eqref{eq:intermediate-main-energy-estimate 2}. Then, the arguments in the proof of \cite[Theorem 3.8]{AgVe24a} can be used verbatim.
    
\end{proof}

\subsection{Proof of well-posedness in the strong setting}\label{Sec:strong_setting}

The aim of this subsection is to prove Theorem \ref{theorem:GWP_d=3_L2_time}.
Recall that in the strong setting we have $\mathcal V=  {}_\nu H^4=\{ u \in H^4 \colon \partial_n u = \partial_n \Delta u=0\} $, $\mathcal H =  {}_\nu H^2 = \{ u \in H^2 \colon \partial_n u = 0 \}$ and consider $d\in \{1,2,3\}$.

\subsubsection{Local well-posedness in the strong setting} \label{sec: LWP variational strong}
We begin by obtaining local well-posedness for \eqref{eq: LLB1} in the strong setting:

\begin{proposition}[Local well-posedness] \label{prop:LWP in strong}
     Let $d\in\{1,2,3\}$ and suppose that Assumption \ref{assumptions on g strong setting} holds. Then, for any $u_0 \in L_{\mathcal F_0}^{0}(\Omega;  {}_\nu H^2)$, there exists a unique maximal solution $(u,\sigma)$ of equation \eqref{eq: LLB1} with $\sigma>0$ a.s., such that 
     \[
        u \in L_{\rm loc}^2([0,\sigma);   {}_\nu  H^4)\cap  C([0,\sigma);  {}_\nu H^2)\, \text{ a.s.}
    \]
     Moreover, the following blow up criterion holds: 
     \begin{equation}\label{eq:blow up criterion in strong}
    \mathbb{P}\bigl(\sigma<\infty, \sup_{t\in[0,\sigma)}\|u(t)\|_{H^2}<\infty\bigr)=0.
\end{equation}
\end{proposition}
\begin{proof}

We verify the assumptions of Theorem \ref{Thm:LWP_variational}.  We begin by proving the coercivity of the linear part of \eqref{eq: LLB1}, i.e.,
\begin{equation}\label{eq:linear coercivity strong setting}
    \langle Au,u \rangle_{L^2,  {}_\nu H^4} \ge \theta \|u\|_{H^4}^2 - M \|u\|_{H^2}^2,
\end{equation}
for some $\theta>0$ and $ M \geq 0$.

We calculate
\begin{align*}
     \langle u,Au \rangle_{ {}_\nu H^4, L^2} & = \int_{\mathscr D} ( \alpha_1\Delta^2 u +\alpha_2\Delta u +\alpha_3 u) (u+\Delta^2u) \, {\rm d}x
     \\
     &= \alpha_1 \|\Delta^2u\|^2_{L^2} - \alpha_2 \|\nabla \Delta u\|^2_{L^2} + (\alpha_1+\alpha_3) \|\Delta u\|^2_{L^2} - \alpha_2  \|\nabla u\|^2_{L^2} +\alpha_3 \|u\|^2_{L^2}. 
\end{align*}
By elliptic regularity (cf.\,\cite[Proposition 7.4]{Ta10}), 
$$\alpha_1 \|\Delta^2 u\|_{L^2}^2 + (\alpha_1+\alpha_3) \|\Delta u\|^2_{L^2} - \alpha_2  \|\nabla u\|^2_{L^2} +\alpha_3 \|u\|^2_{L^2} \ge c \|u\|_{H^4}^2- C\|u\|_{H^2}^2 .$$
On the other hand, by interpolation and Young's inequality we get
$$ \|\nabla \Delta u\|_{L^2}^2 \le \|u\|_{H^3}^2 \le C \|u\|_{H^4} \|u\|_{H^2}\le  \epsilon \|u\|_{H^4}^2 + C_\epsilon \|u\|_{H^2}^2 .$$
Hence, by picking $\epsilon>0$ small enough, the inequalities above, and a standard absorption argument give \eqref{eq:linear coercivity strong setting}.

We now verify the growth condition \eqref{ineq:Local_Lip_F} for the nonlinearities $F_1,F_2,F_3$ as defined in \eqref{eq:definition of F}. For $u,v \in  {}_\nu H^4$ we estimate 
\begin{equation}\label{ineq:estimating F1 and F3 strong setting}
\begin{aligned}
    \| F_1(u) - F_1(v) \|_{L^2} + \| F_3(u) - F_3(v) \|_{L^2} & \lesssim \| |u|^2 u - |v|^2 v \|_{H^2} 
    \\[4pt]
    &=\| (u-v) |u|^2 + v(|u|^2-|v|^2)\|_{H^2} 
    \\ 
    &\stackrel{\mathrm{(i)}}{\lesssim}   \|u-v\|_{H^2} \|u\|^2_{H^2} + \|v\|_{H^2}  \||u|^2-|v|^2\|_{H^2} 
    \\
    &\stackrel{\mathrm{(ii)}}{\lesssim} \|u-v\|_{H^2} \|u\|^2_{H^2} + \|v\|_{H^2}  \|u-v\|_{H^2}  \|u+v\|_{H^2} 
    \\
    & \stackrel{\mathrm{(iii)}}{\lesssim}  (\|u\|_{\mathcal V_\beta}^2 + \|v\|_{\mathcal V_\beta}^2 ) \| u-v \|_{\mathcal V_\beta},
\end{aligned}
\end{equation}
where (i) and (ii) follow from the Banach algebra property of $H^2$ since $H^2 \hookrightarrow L^\infty$ in dimension $d\le3$,
and (iii) follows from the embedding $\mathcal V_{\beta}\hookrightarrow H^{2}$ which holds for any $\beta \ge 1/2$. Therefore  \eqref{ineq:Local_Lip_F} holds with $\rho=2$ and any $\beta \in (\tfrac 12, \tfrac 23)$, in which case the sub-criticality condition \eqref{ineq:subcriticality condition} is satisfied with strict inequality. For the nonlinearity $F_2$ we estimate 
\begin{equation}
    \begin{aligned}
        \| F_2(u)-F_2(v)\|_{L^2} &\eqsim  \| \Delta(u-v) \times u + \Delta v \times (u-v) \|_{L^2}
        \\[4pt]
        & \le  \| \Delta(u-v) \times u\|_{L^2}+ \| \Delta v \times (u-v) \|_{L^2}
        \\
        & \stackrel{\mathrm{(i)}}{\le}  \| \Delta(u-v)\|_{L^2} \|u\|_{L^\infty} + \|\Delta v\|_{L^2} \|u-v\|_{L^\infty}
        \\
        & \stackrel{\mathrm{(ii)}}{\lesssim}  \|u-v\|_{H^2} \|u\|_{H^2} + \|v\|_{H^2} \|u-v\|_{H^2}
        \\[4pt]
        & \stackrel{\mathrm{(iii)}}{\lesssim} \|u-v\|_{\mathcal V_\beta} (\|u\|_{\mathcal V_\beta} + \|v\|_{\mathcal V_\beta}),
    \end{aligned}
\end{equation}
where (i) follows from the elementary inequality $|z\times w|\le |z| |w|$, (ii) follows from the Sobolev embedding $H^2 \hookrightarrow L^\infty$, and (iii) follows from the embedding $\mathcal V_\beta \hookrightarrow H^{2}$ which holds for any $\beta \ge 1/2$. Therefore \eqref{ineq:Local_Lip_F} is fulfilled with $\rho=1$ and any $\beta\in (\tfrac 12, \tfrac 34)$, in which case the sub-criticality condition \eqref{ineq:subcriticality condition} is satisfied with strict inequality. Note that, for example, $\beta=7/12$ works for all cases. 

We now verify the assumptions on the noise term $G$. The measurability of $G$ and the integrability condition \eqref{eq:G0-L2loc} follow by Assumption \ref{assumptions on g strong setting}. We now check the local Lipschitz
condition \eqref{ineq:Local_Lip_G}. 
Let $u,v\in \mathcal V_\beta$. For notational brevity, we write $g(u)$ and suppress the dependence on the remaining variables. By the Neumann compatibility condition \eqref{eq:g-Neumann-compatibility}, the chain
rule, and noting that $\partial_n v=\partial_n u=0$,
\[
\begin{aligned}
\partial_n\bigl(g(\cdot,u)-g(\cdot,v)\bigr)
&=
\partial_n^xg(\cdot,u)-\partial_n^xg(\cdot,v)
+\partial_yg(\cdot,u)\partial_nu
-\partial_yg(\cdot,v)\partial_nv
=0.
\end{aligned}
\]
Therefore,  $G(u)-G(v)\in\mathcal L_2(\ell^2,  {}_\nu  H^2)$. 
Using the elliptic norm equivalence on $  {}_\nu  H^2$, we obtain
\begin{align}
\|G(u)-G(v)\|_{\mathcal L_2(\ell^2,  {}_\nu  H^2)}
&\lesssim
\|g(u)-g(v)\|_{L^2(\ell^2)}
+
\|\Delta(g(u)-g(v))\|_{L^2(\ell^2)}.
\label{eq:G-strong-H2-equivalence}
\end{align}
Further, by the Lipschitz continuity of $g$ in the $y$-variable, we have
\begin{equation}
\label{eq:G-strong-L2}
\|g(u)-g(v)\|_{L^2(\ell^2)}
\le
L\|u-v\|_{L^2}.
\end{equation}
It remains to estimate the term with the Laplacian. The chain rule gives
\begin{align}
\Delta g(u)
&=
\Delta_xg(u)
+\partial_yg(u)\Delta u
+2\sum_{m=1}^d
\partial_{x_my}^2g(u)\cdot \partial_mu
+\sum_{m=1}^d
\partial_{yy}^2g(u)
[\partial_mu,\partial_mu].
\label{eq:G-strong-chain-rule}
\end{align}
Here, for each component \(g_k\), we use the notation
\begin{align*}
     \partial_{yy}^2g_k(u)[\partial_m u,\partial_m u]
      &\coloneq                                 
     \sum_{j=1}^3\sum_{\ell=1}^3
     \frac{\partial^2 g_k(u)}{\partial y_j\partial y_\ell}                                              (\partial_m u_j)(\partial_m u_\ell)
     \intertext{and}
     \partial_{x_my}^2g_k(u)\cdot\partial_m u
     & \coloneq                                               
     \sum_{j=1}^3        
     \frac{\partial^2 g_k(u)}{\partial x_m\partial y_j}
     \partial_m u_j.  
\end{align*}     
Hence
\begin{align*}
\Delta g(u)-\Delta g(v)
&=
\Delta_xg(u)-\Delta_xg(v)
\\
&\quad
+\partial_yg(u)\Delta(u-v)
+\bigl[\partial_yg(u)-\partial_yg(v)\bigr]\Delta v
\\
&\quad
+2\sum_{m=1}^d
\partial_{x_my}^2g(u)\partial_m(u-v)
\\
&\quad
+2\sum_{m=1}^d
\bigl[
\partial_{x_my}^2g(u)-\partial_{x_my}^2g(v)
\bigr]\partial_mv
\\
&\quad
+\sum_{m=1}^d
\partial_{yy}^2g(u)
[\partial_m(u-v),\partial_mu]
\\
&\quad
+\sum_{m=1}^d
\partial_{yy}^2g(u)
[\partial_mv,\partial_m(u-v)]
\\
&\quad
+\sum_{m=1}^d
\bigl[
\partial_{yy}^2g(u)-\partial_{yy}^2g(v)
\bigr]
[\partial_mv,\partial_mv].
\end{align*}
By the Lipschitz continuity of $\Delta_xg$,
\begin{equation}
\label{eq:G-strong-pure-spatial}
\|\Delta_xg(u)-\Delta_xg(v)\|_{L^2(\ell^2)}
\lesssim \|u-v\|_{L^2}.
\end{equation}
Using the boundedness and Lipschitz continuity of $\partial_yg$, we have
\begin{equation}
\begin{aligned}
&\bigl\|
\partial_yg(u)\Delta(u-v)
+
[\partial_yg(u)-\partial_yg(v)]\Delta v
\bigr\|_{L^2(\ell^2)}
\\
&\qquad
\lesssim
\|\Delta(u-v)\|_{L^2}
+
\||u-v|\,|\Delta v|\|_{L^2}
\\
&\qquad
\lesssim \|u-v\|_{H^2}
+
\|u-v\|_{L^\infty}\|v\|_{H^2}.
\label{eq:G-strong-linear}
\end{aligned}
\end{equation}
Similarly, the assumptions on $\partial_{x_my}^2g$ yield
\begin{align}
&
2\sum_{m=1}^d \|
\partial_{x_my}^2g(u)\partial_m(u-v)\|_{L^2(\ell^2)}
+
2\sum_{m=1}^d \|
[\partial_{x_my}^2g(u)-\partial_{x_my}^2g(v)]
\partial_mv
\|_{L^2(\ell^2)}
\notag\\
&\qquad
\lesssim
\||\nabla(u-v)|\|_{L^2}
+
\||u-v|\,|\nabla v|\|_{L^2}
\notag\\
&\qquad
\lesssim
\|u-v\|_{H^2}
+
\|u-v\|_{L^\infty}\|v\|_{H^2}.
\label{eq:G-strong-mixed}
\end{align}
Finally, using the boundedness and Lipschitz continuity of
$\partial_{yy}^2g$, we obtain
\begin{equation}
\begin{aligned}
&
\sum_{m=1}^d \| 
\partial_{yy}^2g(u)
[\partial_m(u-v),\partial_mu]
\|_{L^2(\ell^2)}
\\
&\qquad
+
\sum_{m=1}^d \|
\partial_{yy}^2g(u)
[\partial_mv,\partial_m(u-v)]
\|_{L^2(\ell^2)}
\\
&\qquad
+
\sum_{m=1}^d \|
[\partial_{yy}^2g(u)-\partial_{yy}^2g(v)]
[\partial_mv,\partial_mv]
\|_{L^2(\ell^2)}
\\
& \qquad \lesssim \||\nabla (u-v)| |\nabla u|\|_{L^2} + \||\nabla v ||\nabla (u-v)|\|_{L^2} + \|(u-v) |\nabla v|^2\|_{L^2}
\\
&\qquad
\lesssim
\||\nabla(u-v)|\|_{L^4}
\bigl(
\||\nabla u|\|_{L^4}
+
\||\nabla v|\|_{L^4}
\bigr)
+
\|u-v\|_{L^\infty}
\||\nabla v|\|_{L^4}^2.
\label{eq:G-strong-quadratic}
\end{aligned}
\end{equation}
Hence, 
\eqref{eq:G-strong-pure-spatial}--\eqref{eq:G-strong-quadratic} and the Sobolev embeddings $
    H^2\hookrightarrow L^\infty$, $    H^2\hookrightarrow W^{1,4}
$ give 
\begin{align}
\|\Delta(g(u)-g(v))\|_{L^2(\ell^2)}
&\lesssim
(
1+\|u\|_{H^2}^2+\|v\|_{H^2}^2
)
\|u-v\|_{H^2}.
\label{eq:G-strong-Laplacian-Lipschitz}
\end{align}
Together with \eqref{eq:G-strong-H2-equivalence} and
\eqref{eq:G-strong-L2}, this means
\begin{equation}
\label{eq:G-strong-H2-Lipschitz}
\|G(u)-G(v)\|_{\mathcal L_2(\ell^2,  {}_\nu  H^2)}
\lesssim
(
1+\|u\|_{H^2}^2+\|v\|_{H^2}^2
)
\|u-v\|_{H^2}.
\end{equation}
Noting that $ \mathcal V_{\beta} \hookrightarrow H^2$ for any $\beta \ge 1/2$, we obtain that \eqref{ineq:Local_Lip_G} holds with $\rho=2$ and any $\beta \in (\tfrac  12,\tfrac 23)$. Moreover, the subcriticality condition \eqref{ineq:subcriticality condition} is satisfied with a strict inequality.
In particular, we may again take $\beta=7/12$ to conclude the proof.

\end{proof}

\subsubsection{Global well-posedness in the strong setting}\label{subsesction: GWP in strong}
In this section, we prove the global well-posedness result of  Theorem \ref{theorem:GWP_d=3_L2_time}. Similarly to the intermediate setting, we rely on the blow-up criterion \eqref{eq:blow up criterion in strong} of Proposition \ref{prop:LWP in strong} for which we prove suitable energy estimates. As before, in what follows, $d\in \{1,2,3\}$.

\begin{proposition}[$L^2(\Omega)$-energy estimates in $H^2$]\label{prop:energy estimate H2}
 Suppose that Assumption~\ref{assumptions on g strong setting} holds. Let $u_0\in L_{\mathcal F_0}^{28}(\Omega;H^1) \cap L_{\mathcal F_0}^{2}(\Omega; {}_\nu H^2)$ and let $(u,\sigma)$ be the maximal solution of equation \eqref{eq: LLB1} given by Proposition \ref{prop:LWP in strong}. Then for any $T>0$,
\begin{equation}\label{eq:strong-H2-estimate}
    \mathbb{E} \sup_{t \in [0, T \wedge \sigma)} \|u(t)\|_{H^2}^2+
\mathbb E\int_0^{T\wedge\sigma}
\|u(s)\|_{H^4}^2\,\rd s
 \le C_T \left(
1+\mathbb E\|u_0\|_{H^2}^{2}
+\bbE \|u_0\|_{H^1}^{28}
\right).
\end{equation}
\end{proposition}
\begin{remark}
    Since only the second-moment estimate is needed for the blow-up criterion \eqref{eq:blow up criterion in strong}, we state and prove
Proposition~\ref{prop:energy estimate H2} only for this case. The same
argument extends to arbitrary moments. More precisely, one can show that if 
$p\in[1,\infty)$ 
and 
$ 
    u_0\in L_{\mathcal F_0}^{28p}(\Omega;H^1)\cap L_{\mathcal F_0}^{2p}(\Omega;  {}_\nu  H^2),
$
then for every $T>0$,
\[
\mathbb E
    \sup_{t\in[0,T\wedge\sigma)}
    \|u(t)\|_{H^2}^{2p}
 +  
\mathbb E \left( \int_0^{T\wedge\sigma}
\|u(s)\|_{H^4}^2\,\rd s \right)^p
\le
C_{T,p}(
1+\mathbb E\|u_0\|_{H^2}^{2p}
+\bbE \|u_0\|_{H^1}^{28p}
).
\]
More generally, the assumption
$
   u_0\in L_{\mathcal F_0}^{28p}(\Omega;H^1)\cap L_{\mathcal F_0}^{2p}(\Omega;  {}_\nu  H^2)
$
may be replaced by the weaker condition
$
    \mathbb E\|u_0\|_{H^2}^{2p}
    +
    \mathbb E\mathscr E(u_0)^{7p}
    <\infty,
$
with $\mathscr E$ being the energy functional defined in
\eqref{eq:corrected-GL-energy intro}.
\end{remark}

Before presenting the proof of Proposition \ref{prop:energy estimate H2}, we collect some useful lemmata. 
\begin{lemma}
\label{lem:strong-drift-estimate}
For every $\varepsilon>0$, there exists $C_\varepsilon>0$ such that
\begin{equation}
\label{eq:strong-drift-estimate}
    2\langle v, F(v) \rangle_{ {}_\nu H^4, L^2}
    \le
    \varepsilon\|v\|_{H^4}^2
    +
    C_\varepsilon\bigl(1+\|v\|_{H^1}^{14}\bigr),
    \qquad v\in  {}_\nu H^4.
\end{equation}
\end{lemma}
\begin{proof}

We have
\begin{equation}
\begin{aligned}
\langle v, F(v) \rangle_{ {}_\nu H^4, L^2}
&=
\gamma_1\int_{\mathscr D}
\Delta(|v|^2v)\cdot(v+\Delta^2v)\,\rd x
+\gamma_2\int_{\mathscr D}
(\Delta v\times v)\cdot(v+\Delta^2v)\,\rd x
\\
&\quad
-\gamma_3\int_{\mathscr D}
|v|^2v\cdot(v+\Delta^2v)\,\rd x,
\label{eq:strong-F-pairing-expanded}
\end{aligned}
\end{equation}
where $\gamma_1,\gamma_2,\gamma_3>0$. 
We first consider the lower-order terms. Since $\partial_nv=0$, we have
$
    \partial_n(|v|^2v)
    =
    2(v\cdot\partial_nv)v+|v|^2\partial_nv
    =0$
    on $\partial\mathscr D.$ 
Therefore, integration by parts gives
\begin{align}
\int_{\mathscr D}\Delta(|v|^2v)\cdot v\,\rd x
&=
-\int_{\mathscr D}\nabla(|v|^2v):\nabla v\,\rd x
\notag\\
&=
-\int_{\mathscr D}
\left(
    |v|^2|\nabla v|^2
    +
    2\sum_{m=1}^d(v\cdot\partial_mv)^2
\right)\,\rd x
\le0.
\label{eq:strong-F1-lower}
\end{align}
Moreover,
\begin{equation}
\label{eq:strong-F2-lower}
    \int_{\mathscr D}(\Delta v\times v)\cdot v\,\rd x=0,
\end{equation}
and
\begin{equation}
\label{eq:strong-F3-lower}
    -\int_{\mathscr D}|v|^2v\cdot v\,\rd x
    =
    -\|v\|_{L^4}^4
    \le0.
\end{equation}
Hence, the lower order terms are non-positive.

It remains to estimate the terms containing $\Delta^2v$.
Interpolation and the sharp embedding 
$
( {}_\nu H^1, {}_\nu H^4)_{1/6,1}
=
 {}_\nu B^{3/2}_{2,1}
\hookrightarrow L^\infty
$ imply
\begin{align}
\label{eq:strong-interpolation-Linfty}
\|v\|_{L^\infty}
&\lesssim
\|v\|_{H^1}^{5/6}
\|v\|_{H^4}^{1/6}.
\intertext{Moreover, the embedding $  {}_\nu H^{7/4} \hookrightarrow {}_\nu  H^{1,4}$ and interpolation between $ {}_\nu H^1$ and $ {}_\nu H^4$ yield
\label{eq:strong-interpolation-W14}}
\|\nabla v\|_{L^4}
&\lesssim
\|v\|_{H^{7/4}}
\lesssim
\|v\|_{H^1}^{3/4}
\|v\|_{H^4}^{1/4},
\\
\label{eq:strong-interpolation-H2}
\|v\|_{H^2}
&\lesssim
\|v\|_{H^1}^{2/3}
\|v\|_{H^4}^{1/3}.
\end{align}

For the first nonlinearity, a direct calculation gives
\[
\Delta(|v|^2v)
=
|v|^2\Delta v
+
2(v\cdot\Delta v)v
+
4\sum_{m=1}^d(v\cdot\partial_mv)\partial_mv
+
2|\nabla v|^2v
\]
and thus
$$
|\Delta(|v|^2v)|
\lesssim
|v|^2|\Delta v|
+
|v||\nabla v|^2.
$$
Hence, together with H\"older's inequality, this implies
\begin{align}
\Bigl|
\int_{\mathscr D}
\Delta(|v|^2v)\cdot\Delta^2v\,\rd x
\Bigr|
\lesssim
\|v\|_{L^\infty}^2
\|\Delta v\|_{L^2}
\|\Delta^2v\|_{L^2}
+
\|v\|_{L^\infty}
\|\nabla v\|_{L^4}^2
\|\Delta^2v\|_{L^2}.
\label{eq:strong-F1-high-preliminary}
\end{align}
By
\eqref{eq:strong-interpolation-Linfty} and
\eqref{eq:strong-interpolation-H2}, we have
\begin{align*}
\|v\|_{L^\infty}^2
\|\Delta v\|_{L^2}
\|\Delta^2v\|_{L^2}
&\lesssim
\|v\|_{H^1}^{5/3}
\|v\|_{H^4}^{1/3}
\,
\|v\|_{H^1}^{2/3}
\|v\|_{H^4}^{1/3}
\,
\|v\|_{H^4}
\\
&=
\|v\|_{H^1}^{7/3}
\|v\|_{H^4}^{5/3},
\end{align*}
while by
\eqref{eq:strong-interpolation-Linfty} and
\eqref{eq:strong-interpolation-W14},
\begin{align*}
\|v\|_{L^\infty}
\|\nabla v\|_{L^4}^2
\|\Delta^2v\|_{L^2}
&\lesssim
\|v\|_{H^1}^{5/6}
\|v\|_{H^4}^{1/6}
\,
\|v\|_{H^1}^{3/2}
\|v\|_{H^4}^{1/2}
\,
\|v\|_{H^4}
\\
&=
\|v\|_{H^1}^{7/3}
\|v\|_{H^4}^{5/3}.
\end{align*}
Young's inequality with conjugate exponents $6/5$ and $6$ gives
$$
\|v\|_{H^1}^{7/3}
\|v\|_{H^4}^{5/3}
\le
\varepsilon\|v\|_{H^4}^2
+
C_\varepsilon\|v\|_{H^1}^{14}.
$$
Hence,
\begin{equation}
\label{eq:strong-F1-high}
\Bigl|
\int_{\mathscr D}
\Delta(|v|^2v)\cdot\Delta^2v\,\rd x
\Bigr|
\le
\varepsilon\|v\|_{H^4}^2
+
C_\varepsilon\|v\|_{H^1}^{14}.
\end{equation}

For the cross-product term, H\"older's inequality, \eqref{eq:strong-interpolation-Linfty}, and \eqref{eq:strong-interpolation-H2} give
\begin{equation}
\Bigl|
\int_{\mathscr D}
(\Delta v\times v)\cdot\Delta^2v\,\rd x
\Bigr|
\le
\|v\|_{L^\infty}
\|\Delta v\|_{L^2}
\|\Delta^2v\|_{L^2}
\lesssim
\|v\|_{H^1}^{3/2}
\|v\|_{H^4}^{3/2},
\end{equation}
and after using Young's inequality with conjugate exponents $4/3$ and $4$ there follows
\begin{equation}
\label{eq:strong-F2-high}
\Bigl|
\int_{\mathscr D}
(\Delta v\times v)\cdot\Delta^2v\,\rd x
\Bigr|
\le
\varepsilon\|v\|_{H^4}^2
+
C_\varepsilon\|v\|_{H^1}^{6}.
\end{equation}

Finally, by the embedding $H^1\hookrightarrow L^6$ and Young's inequality,
\begin{align}
\Bigl|
\int_{\mathscr D}
|v|^2v\cdot\Delta^2v\,\rd x
\Bigr|
&\le
\|v\|_{L^6}^3
\|\Delta^2v\|_{L^2}
\lesssim
\|v\|_{H^1}^3
\|v\|_{H^4}
\le
\varepsilon\|v\|_{H^4}^2
+
C_\varepsilon\|v\|_{H^1}^{6}.
\label{eq:strong-F3-high}
\end{align}

Combining the preceding estimates and adjusting $\varepsilon>0$ to
absorb the coefficients $\gamma_1,\gamma_2,\gamma_3$ results in
$$
2\langle v,F(v)\rangle_{ {}_\nu H^4,L^2}
\le
\varepsilon\|v\|_{H^4}^2
+
C_\varepsilon
(1+\|v\|_{H^1}^{14}).
$$
\end{proof}

\begin{lemma}
\label{lem:strong-noise-growth}
Suppose that Assumption~\ref{assumptions on g strong setting} holds. Then, for every
$T>0$ and $\varepsilon>0$, there exists
$C_{T,\varepsilon}>0$ such that, for a.e.\ $(t,\omega)\in[0,T]\times\Omega$,
\begin{equation}
\label{eq:strong-noise-growth}
\|G(t,\omega,v)\|_{\mathcal L_2(\ell^2, {}_\nu H^2)}^2
\le
\varepsilon\|v\|_{H^4}^2
+
C_{T,\varepsilon}
(
1+\|v\|_{H^2}^2+\|v\|_{H^1}^6
),
\qquad v\in  {}_\nu H^4.
\end{equation}
\end{lemma}

\begin{proof}
For notational brevity, we suppress the dependence on $(t,\omega)$ and
write $g(v)=g(t,\omega,\cdot,v(\cdot))$. All constants below are uniform in $(t,\omega)$.
By the Neumann compatibility condition on $G$, we have $
    G(t,\omega,v)\in\mathcal L_2(\ell^2,  {}_\nu  H^2).
$
Using the elliptic norm equivalence on $  {}_\nu  H^2$, we obtain
\begin{equation}
\label{eq:strong-noise-growth-elliptic-short}
\|G(t,\omega,v)\|_{\mathcal L_2(\ell^2, {}_\nu H^2)}^2
\lesssim
\|g(v)\|_{L^2(\ell^2)}^2
+
\|\Delta g(v)\|_{L^2(\ell^2)}^2.
\end{equation}
By the Lipschitz continuity of
$g$ in the $y$-variable and \eqref{eq:g0-H2-strong}, 
\begin{equation}
\label{eq:strong-noise-growth-L2-short}
    \|g(v)\|_{L^2(\ell^2)}^2
    \le C_T(1+\|v\|_{L^2}^2).
\end{equation}
We now use the chain-rule identity \eqref{eq:G-strong-chain-rule}, which we recall:
\begin{align*}
\Delta g(v)
&=
\Delta_xg(v)
+\partial_yg(v)\Delta v
+2\sum_{m=1}^d
\partial_{x_my}^2g(v)\cdot\partial_mv
+\sum_{m=1}^d
\partial_{yy}^2g(v)
[\partial_mv,\partial_mv].
\end{align*}
The bound \eqref{eq:g0-H2-strong} in Assumption~\ref{assumptions on g strong setting},
together with
\[
    \|\Delta_xg(v)\|_{L^2(\ell^2)}
    \le
    \|\Delta_xg(0)\|_{L^2(\ell^2)}
    +
    C\|v\|_{L^2} \lesssim 1 + \|v\|_{L^2}
\]
therefore give
\[
\|\Delta g(v)\|_{L^2(\ell^2)}
\lesssim 
1+\|v\|_{L^2}+\|\Delta v\|_{L^2}+ \||\nabla v|\|_{L^2} +\||\nabla v|\|_{L^4}^2 \lesssim 1+\|v\|_{H^2}+\||\nabla v|\|_{L^4}^2 .
\]
Consequently,
\begin{equation}
\label{eq:strong-noise-growth-pre-interpolation-short}
\|G(t,\omega,v)\|_{\mathcal L_2(\ell^2, {}_\nu H^2)}^2
\le
C_T(
1+\|v\|_{ {}_\nu H^2}^2+\|\nabla v\|_{L^4}^4
).
\end{equation}
By the Sobolev embedding  $
    H^{7/4}\hookrightarrow W^{1,4}
$ and interpolating between $H^1$ and $H^4$, we have 
$
    \|v\|_{H^{7/4}}
    \lesssim
    \|v\|_{H^1}^{3/4}
    \|v\|_{H^4}^{1/4}.
$
Hence,
\[
\|\nabla v\|_{L^4}^4
\lesssim
\|v\|_{H^{7/4}}^4
\lesssim
\|v\|_{H^1}^3\|v\|_{H^4}.
\]
By Young's inequality, for every $\varepsilon>0$,
\[
    \|v\|_{H^1}^3\|v\|_{H^4}
    \le
    \varepsilon\|v\|_{H^4}^2
    +
    C_{\varepsilon}\|v\|_{H^1}^6.
\]
Substituting this into
\eqref{eq:strong-noise-growth-pre-interpolation-short} yields \eqref{eq:strong-noise-growth}.
\end{proof}

We are now ready to present the proof of Proposition~\ref{prop:energy estimate H2}.
\begin{proof}[Proof of Proposition~\ref{prop:energy estimate H2}]
Let $(\sigma_n)_{n\ge1}$ be a localizing sequence for $(u,\sigma)$. For
$R>0$ and $n\ge1$, set
\begin{equation*}
\tau_{R,n}
 \coloneq 
T\wedge\sigma_n\wedge
\inf\left\{
0\leq t<\sigma_n:
\|u(t)\|_{H^2}\geq R
\right\}.
\end{equation*}
Applying the variational It\^o formula in the Gelfand triple
$
{}_\nu H^4 \hookrightarrow {}_\nu H^2 \hookrightarrow L^2
$ 
gives
\begin{equation} 
\begin{aligned}
\|u(t\wedge\tau_{R,n})\|_{H^2}^2
&=
\|u_0\|_{H^2}^2
-2\int_0^{t\wedge\tau_{R,n}}
\langle u,Au\rangle_{ {}_\nu H^4,L^2}\,\rd s
\\
&\quad
+2\int_0^{t\wedge\tau_{R,n}}
\langle u, F(u) \rangle_{ {}_\nu H^4,L^2}\,\rd s
\\
&\quad
+\int_0^{t\wedge\tau_{R,n}}
\|G(s,u)\|_{\mathcal L_2(\ell^2, {}_\nu H^2)}^2\,\rd s
+M_{R,n}(t),
\label{eq:strong-Ito-H2}
\end{aligned}
\end{equation}
where
\[
M_{R,n}(t)
=
2\sum_{k\ge1}
\int_0^{t\wedge\tau_{R,n}}
(u(s),G_k(s,u(s)))_{ {}_\nu H^2}\,\rd W_s^k.
\]
By the linear coercivity estimate \eqref{eq:linear coercivity strong setting},
\[
-2\langle u,Au\rangle_{{}_\nu H^4,L^2}
\le
-2\theta\|u\|_{H^4}^2
+
2M\|u\|_{H^2}^2,
\]
and using Lemma~\ref{lem:strong-drift-estimate} and
Lemma~\ref{lem:strong-noise-growth}, with $\varepsilon>0$ chosen
sufficiently small, we obtain
\begin{equation}
\begin{aligned}
&\|u(t\wedge\tau_{R,n})\|_{H^2}^2
+
c\int_0^{t\wedge\tau_{R,n}}
\|u(s)\|_{H^4}^2\,\rd s
\\
&\qquad\le
\|u_0\|_{H^2}^2
+
C_T\int_0^{t\wedge\tau_{R,n}}
\left(
1+\|u(s)\|_{H^2}^2+\|u(s)\|_{H^1}^{14}
\right)\,\rd s
+
M_{R,n}(t).
\label{eq:strong-H2-pathwise}
\end{aligned}
\end{equation}
Set now
\[
X_{R,n}(t)
 \coloneq 
\sup_{r\in[0,t]}
\|u(r\wedge\tau_{R,n})\|_{H^2}^2.
\]
Taking the supremum in \eqref{eq:strong-H2-pathwise} and then expectations
gives
\begin{align}
&\mathbb EX_{R,n}(t)
+
c\mathbb E\int_0^{t\wedge\tau_{R,n}}
\|u(s)\|_{H^4}^2\,\rd s
\notag\\
&\qquad\le
\mathbb E\|u_0\|_{H^2}^2
+
C_T\int_0^t
\big(
1+\mathbb EX_{R,n}(s)
+\mathbb E\sup_{\rho< s\wedge\sigma}
\|u(\rho)\|_{H^1}^{14}
\big)\,\rd s
+\mathbb E\sup_{r\le t}|M_{R,n}(r)|.
\label{eq:strong-H2-before-BDG}
\end{align}
By the Burkholder--Davis--Gundy inequality,
\begin{align*}
\mathbb E\sup_{r\le t}|M_{R,n}(r)|
&\le
C\mathbb E
\Big(
\int_0^{t\wedge\tau_{R,n}}
\sum_{k\ge1}
|(u,G_k(s,u))_{ {}_\nu H^2}|^2\,\rd s
\Big)^{1/2}
\\
&\le
C\mathbb E
\Big[
X_{R,n}(t)^{1/2}
\Big(
\int_0^{t\wedge\tau_{R,n}}
\|G(s,u)\|_{\mathcal L_2(\ell^2, {}_\nu H^2)}^2\,\rd s
\Big)^{1/2}
\Big]
\\
&\le
\frac14\mathbb EX_{R,n}(t)
+
C\mathbb E\int_0^{t\wedge\tau_{R,n}}
\|G(s,u)\|_{\mathcal L_2(\ell^2, {}_\nu H^2)}^2\,\rd s,
\end{align*}
and applying Lemma~\ref{lem:strong-noise-growth} once more, and choosing its
$\varepsilon$ sufficiently small, yields
\begin{align}
\mathbb E\sup_{r\le t}|M_{R,n}(r)|
&\le
\frac14\mathbb EX_{R,n}(t)
+
\frac c2
\mathbb E\int_0^{t\wedge\tau_{R,n}}
\|u(s)\|_{H^4}^2\,\rd s
\notag\\
&\quad
+
C_T\int_0^t
\big(
1+\mathbb EX_{R,n}(s)
+\mathbb E\sup_{\rho < s\wedge\sigma}
\|u(\rho)\|_{H^1}^{14}
\big)\,\rd s.
\label{eq:strong-H2-BDG}
\end{align}
Then substituting \eqref{eq:strong-H2-BDG} into
\eqref{eq:strong-H2-before-BDG} and absorbing the first two terms gives
\begin{equation}
\begin{aligned}
&\mathbb EX_{R,n}(t)
+
c\mathbb E\int_0^{t\wedge\tau_{R,n}}
\|u(s)\|_{H^4}^2\,\rd s
\\
& \quad \le
C\mathbb E\|u_0\|_{H^2}^2
+
C_T\int_0^t\mathbb EX_{R,n}(s)\,\rd s
+
C_T(
1+
\mathbb E\sup_{s\in[0,T\wedge\sigma)}
\|u(s)\|_{H^1}^{14}
).
\label{eq:strong-H2-Gronwall}
\end{aligned}
\end{equation}
We now apply the $H^1$-energy estimate of Proposition~\ref{prop:energy estimate in H1}. This is justified by Remark~\ref{rem:intermediate-weaker-origin},
since \eqref{eq:g0-H2-strong} and the embedding
$H^2(\ell^2)\hookrightarrow W^{1,4}(\ell^2)$ imply ~\eqref{eq:intermediate-weaker-origin}. Therefore, Proposition~\ref{prop:energy estimate in H1} with $p=7$ yields
\[
\mathbb E\sup_{s\in[0,T\wedge\sigma)}
\|u(s)\|_{H^1}^{14}
\le
C_{T} (1+\bbE \|u_0\|_{H^1}^{28}),
\]
and consequently,
\begin{align*}
&\mathbb EX_{R,n}(t)
+
c\mathbb E\int_0^{t\wedge\tau_{R,n}}
\|u(s)\|_{ {}_\nu H^4}^2\,\rd s
\\
&\qquad\le
C_T\Big(
1+\mathbb E\|u_0\|_{ {}_\nu H^2}^2
+\bbE \|u_0\|_{H^1}^{28}
\Big)
+
C_T\int_0^t\mathbb EX_{R,n}(s)\,\rd s.
\end{align*}
Gronwall's lemma now gives
\begin{align*}
\mathbb E\sup_{t\in[0,T]}
\|u(t\wedge\tau_{R,n})\|_{H^2}^2
+
\mathbb E\int_0^{T\wedge\tau_{R,n}}
\|u(s)\|_{H^4}^2\,\rd s
\le
C_T(
1+\mathbb E\|u_0\|_{H^2}^2
+\bbE \|u_0\|_{H^1}^{28}
),
\label{eq:strong-H2-stopped}
\end{align*}
and letting first $R\to\infty$ and then $n\to\infty$, and using Fatou's
lemma, proves \eqref{eq:strong-H2-estimate}.

\end{proof}

\begin{proof}[Proof of Theorem \ref{theorem:GWP_d=3_L2_time}]
 By a localization argument (cf.\,\cite[Remark 5.3]{AgVe25}), we may assume that $u_0 \in L^\infty_{\mathcal F_0}(\Omega; {}_\nu H^2)$. Hence, global well-posedness follows from Proposition \ref{prop:energy estimate H2} and 
 the blow-up criterion \eqref{eq: blow-up-criterion-subcritical}. Moreover, the energy estimate \eqref{eq:strong-main-energy-estimate} follows from \eqref{eq:strong-H2-estimate}. 
    
\end{proof}

\begin{proof}[Proof of Corollary \ref{Coroll:GWP_strong}]
     The proof is analogous to the proof of Corollary \ref{Coroll:GWP_intermediate}, this time with the energy estimate \eqref{eq:strong-main-energy-estimate}.
    
\end{proof}

\subsection{Proof of well-posedness with rough initial data} \label{subsec:critical-LpLq-full-range}
In this section, we prove the global well-posedness result stated in Theorem \ref{thm:rough-final-global}. Throughout the section, 
$$X_0^{s,q} =  {}_\nu H^{-s,q},\qquad X_1^{s,q}= {}_\nu H^{4-s,q}.$$

\subsubsection{Local well-posedness}
We begin by proving local well posedness for initial data in the critical Besov space $B^{d/q-1}_{q,p}$. 
Applying the local theory of \cite{AgVe25} with our nonlinear
estimates restricts the admissible range to $q<2d$, owing to the
subcriticality condition for the nonlinearity $F_2$. To treat the
full range $q\in [2,\infty)$, we instead use the \enquote{shifted} local well-posedness theory of
\cite[Section~4.1]{AGV26shift} (see also \cite{BGV26shift}), which exploits estimates of the
nonlinearities from $X_\beta^{s,q}$ to $X_a^{s,q}$ with $a>0$.
For a nonlinearity with
growth exponent $\rho>0$, this replaces the sub-criticality condition
\[
\frac{1+\kappa}{p}
\le \frac{(1+\rho)(1-\beta)}{\rho}
\quad\text{by}\quad
\frac{1+\kappa}{p}
\le \frac{(1+\rho)(1-\beta)+a}{\rho}.
\]
This additional spatial
regularity ensures that the nonlinearity $F_2$ remains strictly subcritical. See Remark \ref{remark:role of shift a} for more details.

We begin by establishing general local Lipschitz estimates for $F$ and $G$.

\begin{proposition}[Lipschitz estimates for $F$]
\label{prop:rough-final-Lipschitz}
Let $d\in\{1,2,3\}$, $q\in[2,\infty)$, and
$
 3-\frac dq\le s<3.
$
Choose $a\ge0$ such that
\begin{equation}\label{eq:final-cubic-range}
 0\le 2-s+4a<\min\left\{1,\frac dq\right\}.
\end{equation}
Set
\begin{equation}\label{eq:final-spatial-exponents}
 t_1=\frac{2+2d/q-s+4a}{3},\qquad
 t_2=\begin{cases}
       (1+d/q)/2,&s<2,\\[1mm]
       1/2+d/(5q),&s\ge2,
     \end{cases}
\end{equation}
Then,
\begin{align}
 \|F_j(u)-F_j(v)\|_{H^{4a-s,q}}
 &\lesssim(\|u\|_{H^{t_1,q}}^2+\|v\|_{H^{t_1,q}}^2)
                 \|u-v\|_{H^{t_1,q}},\quad j\in\{1,3\},
                 \label{eq:final-cubics}\\
 \|F_2(u)-F_2(v)\|_{H^{-(s\wedge 2),q}}
 &\lesssim(\|u\|_{H^{t_2,q}}+\|v\|_{H^{t_2,q}})
                 \|u-v\|_{H^{t_2,q}}.\label{eq:final-F2}
\end{align}
\end{proposition}

\begin{proof}
We first prove \eqref{eq:final-cubics}. Recall $F_1(u)=\Delta (|u|^2u)$ and $F_3(u)=-|u|^2u$. Let 
$T(z_1,z_2,z_3)\coloneq (z_1\cdot z_2)z_3$. We claim that
\begin{equation}\label{eq:final-trilinear}
 \|T(z_1,z_2,z_3)\|_{H^{2+4a-s,q}}
       \lesssim\prod_{j=1}^3\|z_j\|_{H^{t_1,q}}.
\end{equation}
Consequently,  trilinearity and
\eqref{eq:final-trilinear} imply
\begin{align*}\label{eq:qrange-s-cubic-difference}
   \|F_j(u)-F_j(v)\|_{H^{4a-s,q}} &\lesssim  \||u|^2u-|v|^2v\|_{H^{2+4a-s,q}}
    \\ 
    &\lesssim
    (\|u\|_{H^{t_1,q}}^2+\|v\|_{H^{t_1,q}}^2)
    \|u-v\|_{H^{t_1,q}}, \quad j \in \{1,3\}.
\end{align*}
In order to prove \eqref{eq:final-trilinear}, we let $h\coloneq 2+4a-s$.

If $h=0$, the claim follows from H\"older's inequality and   $H^{\frac{2d}{3q},q}\hookrightarrow L^{3q}$.
Suppose that $h>0$. 

Set 
\[
 r\coloneq\frac{3d}{d/q+2h},\qquad m\coloneq\frac{3d}{d/q-h}
\]
and note $1/r+2/m=1/q$. Note that by \eqref{eq:final-cubic-range}, $h<1+1/q$ and thus $ {}_\nu H^{h,q}$ coincides with the usual Bessel potential space $H^{h,q}$. By the standard paraproduct estimate,  after extension from $\mathscr D$ to the full space $\bbR^d$ (cf.\,\cite[Chapter 2, Proposition 1.1]{Taylor_book_paraproduct}), we get 
\[
 \|T(z_1,z_2,z_3)\|_{H^{h,q}}
 \lesssim\sum_{i=1}^3\|z_i\|_{H^{h,r}}
                   \prod_{j\ne i}\|z_j\|_{L^m}.\]
Finally, the Sobolev embedding
$ {}_\nu H^{t_1,q}\hookrightarrow H^{h,r}\cap L^m$ gives
\[
 \|T(z_1,z_2,z_3)\|_{H^{h,q}}
 \lesssim\prod_{j=1}^3\|z_j\|_{H^{t_1,q}}.
\]
This proves \eqref{eq:final-trilinear}.

We next prove \eqref{eq:final-F2}. Recall $F_2(u)=\operatorname{div}(\nabla u \times u)$. We distinguish the following two cases:
\begin{itemize}
    \item \emph{Case 1: $s<2$.} Since $s \ge 3-d/q$, it must hold that $d/q>1$ and $s>1$. This forces $d=3$.  Set
\[
 r\coloneq\frac{6}{3/q-1},\qquad m\coloneq\frac{6}{3/q+1}.
\]
We have $1/r+1/m=1/q$ and
$ {}_\nu  H^{t_2,q}\hookrightarrow L^r\cap W^{1,m}$. Therefore, 
\begin{align*}
 \|F_2(u)-F_2(v)\|_{H^{-1,q}}
 &\lesssim\|\nabla(u-v)\|_{L^m}\|u\|_{L^r}
          +\|\nabla v\|_{L^m}\|u-v\|_{L^r}\\
 &\lesssim(\|u\|_{H^{t_2,q}}+\|v\|_{H^{t_2,q}})
                \|u-v\|_{H^{t_2,q}}.
\end{align*}
The embedding $ {}_\nu H^{-1,q}\hookrightarrow {}_\nu H^{-s,q}$ proves
\eqref{eq:final-F2} in this case.

\item \emph{Case 2: $s\ge2$.} For
smooth functions satisfying the Neumann boundary condition, define the symmetric bilinear form 
\[
 Q(u,v)\coloneq\tfrac12(\Delta u\times v+\Delta v\times u).
\]
Since $F_2(u)=Q(u,u)$, it suffices to prove
\begin{equation}
\label{eq:qrange-s-Q-estimate final}
 \|Q(u,v)\|_{H^{-2,q}}
 \lesssim\|u\|_{H^{t_2,q}}\|v\|_{H^{t_2,q}}.
\end{equation}
Put 
$b\coloneq 2t_2$. Bilinear complex interpolation reduces this to
\begin{align}
\label{eq:qrange-s-Q-endpoint final}
 \|Q(u,v)\|_{H^{-2,q}}
 &\lesssim\|u\|_{L^q}\|v\|_{H^{b,q}},
 \\
 \|Q(u,v)\|_{H^{-2,q}}
 &\lesssim\|u\|_{H^{b,q}}\|v\|_{L^q},
\end{align}
see \cite[Theorem~4.4.1]{BerghLofstrom1976}.
Indeed, $[L^q, {}_\nu H^{b,q}]_{1/2}= {}_\nu H^{t_2,q}$. By symmetry, it suffices to establish \eqref{eq:qrange-s-Q-endpoint final}. 
To this end, let $q'=q/(q-1)$ and $\phi\in {}_\nu H^{2,q'}$.
Two integrations by parts give
\begin{equation}\label{eq:final-Q-duality}
 \begin{aligned}
\langle Q(u,v),\phi\rangle
&=\sum_{j=1}^d\int_{\mathscr D}(u\times\partial_jv)\cdot\partial_j\phi\,\mathrm dx +\frac12\int_{\mathscr D}(u\times v)\cdot\Delta\phi\,\mathrm dx
   \\
   &\eqcolon T_1+T_2.
   \end{aligned}
\end{equation}
Note $d/q\le3/2<5/3$ and $b=1+2d/(5q)$. Therefore, $b>\max\{1,d/q\}$. 
In particular, there is $x>0$ such that 
\[
 \frac1q-\frac{b-1}d <x<\min\Bigl\{\frac1q,\frac1d\Bigr\}.
\]
Let $m,n\in(1,\infty)$ such that 
 $\frac 1m=x$ and $\frac1n=\frac1{q'}-x.$
Then, the embeddings  
$$ {}_\nu  H^{b,q}\hookrightarrow W^{1,m}, \quad 
 {}_\nu H^{2,q'}\hookrightarrow W^{1,n}$$
hold. Therefore, H\"older's inequality with $\frac1q+\frac1m+\frac1n=1$ yields
\begin{equation*} 
    |T_1| \lesssim \|u\|_{L^q} \|\nabla v\|_{L^m} \|\nabla \phi\|_{L^n} \lesssim \|u\|_{L^q}\|v\|_{H^{b,q}}
                    \|\phi\|_{H^{2,q'}}.
\end{equation*}
Finally, the embedding 
$H^{b,q}\hookrightarrow L^\infty$ and H\"older's inequality give
\begin{equation*}
    |T_2| \lesssim \|u\|_{L^q} \|v\|_{L^\infty} \|\Delta \phi\|_{L^{q'}} 
 \lesssim\|u\|_{L^q}\|v\|_{H^{b,q}}
                    \|\phi\|_{H^{2,q'}}.
\end{equation*}
The estimates for $T_1$ and $T_2$ imply \eqref{eq:qrange-s-Q-endpoint final} and we conclude the proof.
\end{itemize}

\end{proof}

\begin{lemma}[Lipschitz estimates for $G$]
\label{lemma:rough-final-noise-lipschitz}
Suppose
that Assumption~\ref{ass:noise_intermediate} holds. Let $d\in\{1,2,3\}$, $q\in[2,\infty)$ and
$$
 3-\frac dq\le s<3.
$$ 
If $s\ge2$, then, 
\begin{equation}\label{eq:final-noise-nonpositive}
 \|G(t,\omega,u)-G(t,\omega,v)\|_{
       \gamma(\ell^2,X_{1/2}^{s,q})}
 \lesssim\|u-v\|_{L^q}.
\end{equation}
If $s<2$, set $t_2=(1+d/q)/2$. Then,
\begin{equation}\label{eq:final-noise-positive}
 \|G(t,\omega,u)-G(t,\omega,v)\|_{
       \gamma(\ell^2,X_{1/2}^{s,q})}
 \lesssim
 (1+\|u\|_{H^{t_2,q}}+\|v\|_{H^{t_2,q}})
       \|u-v\|_{H^{t_2,q}}.
\end{equation}
Moreover, for every $T>0$,
\begin{equation}\label{eq:final-noise-origin}
 \|G(t,\omega,0)\|_{\gamma(\ell^2,X_{1/2}^{s,q})}\le C_T.
\end{equation}
\end{lemma}

\begin{proof}
We suppress the dependence on $(t,\omega)$. Recall that
$X_{1/2}^{s,q}= {}_\nu H^{2-s,q}$. We use the $\gamma$-Fubini
isomorphism and the ideal property of $\gamma$-radonifying operators;
see \cite[Theorems~9.1.10 and~9.4.8]{AnalysisBspaces2}. 
If $s\ge2$, the embedding $L^q\hookrightarrow {}_\nu H^{2-s,q}$
and \eqref{eq:g-Lip} imply
\[
 \|G(u)-G(v)\|_{\gamma(\ell^2,X_{1/2}^{s,q})}
 \lesssim\|g(\cdot,u)-g(\cdot,v)\|_{L^q(\ell^2)}
 \lesssim\|u-v\|_{L^q},
\]
which proves \eqref{eq:final-noise-nonpositive}.

Suppose now that $s<2$. Then $d/q>1$, $t_2>1$, and
$0<2-s\le1/2$. This forces $d=3$. Set
\[
 r\coloneq\frac{6}{3/q-1},\qquad m\coloneq\frac{6}{3/q+1}.
\]
Then $r,m \in (1,\infty)$, $
 \frac1r+\frac1m=\frac1q$
 and the Sobolev embedding $
 H^{t_2,q}\hookrightarrow L^r\cap W^{1,m}$ holds.
 
The chain rule gives
\begin{align*}
 \nabla[g(\cdot,u)-g(\cdot,v)]
 ={}&g_x(\cdot,u)-g_x(\cdot,v)
       +g_y(\cdot,u)\nabla(u-v)+[g_y(\cdot,u)-g_y(\cdot,v)]\nabla v.
\end{align*}
Hence, by Assumption~\ref{ass:noise_intermediate},
\[
 |g(x,u)-g(x,v)|_{\ell^2}
 +|\nabla[g(x,u)-g(x,v)]|_{\ell^2}
 \lesssim |u-v|+|\nabla(u-v)|+|u-v||\nabla v|.
\]
By H\"older's inequality with $1/r+1/m=1/q$, and the above embeddings,
\begin{align*}
 \|g(\cdot,u)-g(\cdot,v)\|_{H^{1,q}(\ell^2)}
 &\lesssim\|u-v\|_{H^{1,q}}
          +\|u-v\|_{L^r}\|\nabla v\|_{L^m}\\
 &\lesssim(1+\|u\|_{H^{t_2,q}}+\|v\|_{H^{t_2,q}})
                 \|u-v\|_{H^{t_2,q}}.
\end{align*}
Since $ {}_\nu H^{1,q}\hookrightarrow {}_\nu H^{2-s,q}=X_{1/2}^{s,q}$,
this proves \eqref{eq:final-noise-positive}. 

Finally,  \eqref{eq:final-noise-origin} follows from \eqref{eq:g0-assumption_intermediate}.

\end{proof}

We now choose the parameters for the local theory. Let $p\in(2,\infty)$ and
$q\in[2,\infty)$. Choose $s>0$ such that
\begin{equation}\label{eq:final-local-parameters}
\quad
 3-\frac dq<s<\min\left\{3,5-\frac dq-\frac4p\right\}.
\end{equation}
Such $s$ exists since $p>2$. Let
\begin{equation}\label{eq:qrange-critical-weight-s final}
\kappa = \frac{p(5 - s - d/q) - 4}{4}
\end{equation}
and note that 
$\kappa \in (0,p/2-1)$. Moreover, 
\begin{equation}\label{eq:qrange-s-admissibility final}
 X_{1-\frac{1+\kappa}{p},p}^{s,q}
 =B^{d/q-1}_{q,p}.
\end{equation}
For the endpoint $p=q=2$, we instead set
\begin{equation}
\label{eq:qrange-Hilbert-parameters final}
 s=3-\frac d2,\qquad 
 \kappa=0.
\end{equation}
Then $X_{1/2}^{s,2}=B_{2,2}^{d/2-1}=H^{d/2-1}$.

\begin{proposition}[Local well-posedness]
\label{thm:rough-final-local}
Suppose that Assumption~\ref{ass:noise_intermediate} holds. Let $d\in\{1,2,3\}$ and assume either 
\begin{enumerate}[(I)]
    \item   $q\in[2,\infty)$, $p \in (2,\infty)$ and $s,\kappa$ as in
\eqref{eq:final-local-parameters}--\eqref{eq:qrange-critical-weight-s final}; \label{item: I}

\vspace{2mm}

\item  $p=q=2$ and $s,\kappa$ as in 
\eqref{eq:qrange-Hilbert-parameters final}. \label{item: II}
\end{enumerate}
Then, for every 
$$u_0\in L^0_{\mathcal F_0}(\Omega;B^{d/q-1}_{q,p}),$$
there exists a unique maximal local solution $(u,\sigma)$ of equation \eqref{eq: LLB1} with $\sigma>0$ a.s., such that
\begin{equation}\label{eq:final-local-regularity}
 u\in L^p_{\mathrm{loc}}([0,\sigma),t^\kappa; {}_\nu H^{4-s,q})
       \cap C([0,\sigma);B^{d/q-1}_{q,p})\,\text{ a.s.}
\end{equation}
Moreover, if $p>2$, then almost surely,
\begin{equation}\label{eq:final-local-interior}
 u\in H^{\theta,r}_{\mathrm{loc}}((0,\sigma); {}_\nu H^{4-s-4\theta,q})
       \cap C((0,\sigma); {}_\nu B^{4-s-4/r}_{q,r}), \quad 0\le \theta<1/2, \quad  r \in [2,\infty).
\end{equation}
\end{proposition}

\begin{proof}
Suppose first that \ref{item: I} holds.
We apply \cite[Theorem 4.5]{AGV26shift}. 
Put 
$$\mu\coloneq \frac{1+\kappa}p=\frac{5-s-d/q}4$$ and note $1/p<\mu<1/2$. We apply  Proposition~\ref{prop:rough-final-Lipschitz} with 
$$a=\frac{(s-2)_+}{4} \in [0,1/4)$$
and $t_1,t_2$ as in \eqref{eq:final-spatial-exponents}.
Note
$$X_a^{s,q}= {}_\nu  H^{-(s\wedge 2),q}.$$
Moreover, 
$2-s+4a=(2-s)_+$ and thus 
\eqref{eq:final-cubic-range} holds. 
Set 
$$\beta_1\coloneq \frac{s+t_1}4=\frac{s+d/q+1+2a}{6}.$$
Note $t_1<4-s$ and thus $\beta_1 <1$. Indeed, 
when $s\ge2$ this follows from $t_1= 2d/(3q)\le1$ and $s<3$. When $s<2$, it follows from 
$s+t_1=(2s+2+2d/q)/3<3$.
By \eqref{eq:final-cubics}, $F_1$ and $F_3$ are locally Lipschitz from $X_{\beta_1}^{s,q}$ to $ X_{a}^{s,q}$ with $\rho_1=\rho_3=2$. Moreover,  
\begin{equation}\label{eq:final-cubic-criticality}
 \frac{3(1-\beta_1)+a}{2}=\frac{5-s-d/q}{4}=\mu
\end{equation}
and thus $F_1$ and $F_3$ are critical. Note $a<1-\mu<\beta_1<1$. The lower bound for $\beta_1$
is equivalent to $t_1>d/q-1$, which follows from
$3-d/q+(2-s)_+>0$. The bound for $a$ follows from
$4(1-\mu-a)=s+d/q-1-(s-2)_+>0$.

Similarly, with $\beta_2=(s+t_2)/4$, \eqref{eq:final-F2} implies that $F_2 \colon X_{\beta_2}^{s,q} \to X_{a}^{s,q}$ is locally Lipschitz with $\rho_2=1$. Moreover, 
\begin{equation}\label{eq:final-F2-subcriticality}
2(1-\beta_2)+a-\mu
=
\begin{cases}
\tfrac{2-s}{4}, & s<2,\\[2mm]
\tfrac{3d}{20q}, & s\ge 2.
\end{cases}\end{equation}
Since the right hand side is always positive, $F_2$ is strictly subcritical. Note that $\beta_2\in(1-\mu,1)$. Indeed,
$
\beta_2-(1-\mu)=\frac{t_2+1-d/q}{4}>0,
$
since $t_2>d/q-1$ in both cases.

For the noise $G$, set 
$$\beta_4 \coloneq 
\begin{cases}
    \dfrac{s+d/(2q)}{4}, & s \ge 2,
    \\[2mm]
    \dfrac{s+t_2}{4}, & s<2,
\end{cases}$$
where $t_2=(1+d/q)/2$. If $s\ge2$ then $X_{\beta_4}^{s,q}= {}_\nu H^{\frac{d}{2q},q} \hookrightarrow L^q$ and so \eqref{eq:final-noise-nonpositive} implies that $G \colon X_{\beta_4}^{s,q} \to \gamma(\ell^2,X_{1/2}^{s,q})$ is globally Lipschitz, and in particular strictly subcritical. If $s<2$ then $X_{\beta_4}^{s,q}= {}_\nu H^{t_2,q}$ and so \eqref{eq:final-noise-positive} implies that $G$ is locally Lipschitz with $\rho_4=1$, which is again strictly subcritical by \eqref{eq:final-F2-subcriticality}. The required $L^0(\Omega;L^p(0,T;w_\kappa;\gamma(\ell^2,X_{1/2}^{s,q})))$ bound for $G(t,\omega,0)$ follows from
\eqref{eq:final-noise-origin}.

The realization $A_{-s,q}$ of $A$ on $X_{0}^{s,q}$ 
has stochastic maximal regularity by Lemma \ref{lemma:SMR}. 
All hypotheses of 
\cite[Theorem~4.5]{AGV26shift} are satisfied.
That theorem implies the existence of a unique maximal solution $(u,\sigma)$ with $\sigma>0$ a.s. and the regularity
\eqref{eq:final-local-regularity}. Since $p>2$, \cite[Theorem~4.8]{AGV26shift}, the trace embedding of \cite[Proposition 2.1]{AgVe25} and the identity $X^{s,q}_{1-1/r,r}= {}_\nu B_{q,r}^{4-s-4/r}$ imply \eqref{eq:final-local-interior}. 

Suppose now that \ref{item: II} holds. Then 
$$s=3-d/2 
=
\begin{cases}
    5/2, &d=1
    \\
    2,&d=2
    \\
    3/2, &d=3.
\end{cases}
$$
We apply \cite[Theorem 5.6]{BGV26shift} and note that the linear coercivity assumption therein can be replaced by the stochastic maximal regularity result of Lemma \ref{lemma:SMR}.
By Proposition \ref{prop:rough-final-Lipschitz} with $a=(s-2)_+/4$, which is allowed since 
$2-s+4a \in\{0,1/2\}$, and by Lemma \ref{lemma:rough-final-noise-lipschitz}, the nonlinearities $F_j \colon X_{\beta_j}^{s,2} \to X_a^{s,2}$ and the noise $G\colon X_{\beta_4}^{s,2} \to \gamma(\ell^2,X_{1/2}^{s,2}) $ are locally Lipschitz, with 
\[
\beta_1=\beta_3
\coloneq 
\frac{2}{3}+\frac{(2-d)_+}{24}, \qquad 
\beta_2=\beta_4 \coloneq 
\begin{cases}
\dfrac{7}{8}-\dfrac{d}{10}, & d\in\{1,2\},\\[2mm]
\dfrac{11}{16}, & d=3.
\end{cases}
\]
Moreover, the sub-criticality condition
$$  (1+\rho_j)(2\beta_j-1)\le 1 +2a, \quad j \in \{1,2,3\}$$
is satisfied, where equality holds only for $j\in\{1,3\}$. Similarly, the noise $G$ remains strictly subcritical.  Therefore, all the assumptions of  \cite[Theorem 5.5]{BGV26shift} are satisfied and we are done. 

\end{proof}

\begin{remark}
Proposition \ref{thm:rough-final-local} remains valid
if \eqref{eq:g0-assumption_intermediate} is replaced by
the weaker assumption
\eqref{eq:intermediate-weaker-origin intro}.
Indeed, the only thing left to check is the required $L^0(\Omega;L^p(0,T;w_\kappa;\gamma(\ell^2,X_{1/2}^{s,q})))$ bound for $G(t,\omega,0)$, which follows from the identity $X_{1/2}^{s,q} =  H^{2-s,q}$ and the embedding $W^{1,4} \hookrightarrow H^{2-s,q}$. 
\end{remark}

\subsubsection{Global well-posedness}
We are now ready to present the proof of Theorem \ref{thm:rough-final-global}. We recall that $d\in \{1,2,3\}$ and  either one of the following is satisfied:
\begin{enumerate}[(I)]
    \item \label{item: I1} $q\in[2,\infty)$, $p \in (2,\infty)$ and $s,\kappa$ satisfy 
    $$
 3-\frac dq<s<\min\left\{3,5-\frac dq-\frac4p\right\}, \qquad 
\kappa = \frac{p(5 - s - d/q) - 4}{4};
$$

\item \label{item: I2}  $p=q=2$ and 
$$s=3-\frac d2 , \qquad \kappa=0.$$
\end{enumerate}
For the reader's convenience, we split the proof into two parts: first we consider the case \ref{item: I1}  and then \ref{item: I2}. Recall 
$$X_0^{s,q} =  {}_\nu H^{-s,q},\qquad X_1^{s,q}= {}_\nu H^{4-s,q}.$$

\begin{proof}[Proof of Theorem \ref{thm:rough-final-global} with assumption \ref{item: I1}] Let $(u,\sigma)$ be the unique maximal solution of
Proposition~\ref{thm:rough-final-local}. We divide the proof into three steps.

\medskip\noindent
\emph{Step 1:  parabolic regularization.} We show that
\begin{equation}\label{eq:rough-global-H2q}
 u\in H^{\theta,r}_{\mathrm{loc}}((0,\sigma); {}_\nu H^{2-4\theta,q})
\, \text{ a.s.}, \quad 0\le \theta <1/2, \quad r \in [2,\infty).
\end{equation}
Evidently, it suffices to prove \eqref{eq:rough-global-H2q} for sufficiently large $r$. Let
$r_0>\max\{6,4/(3-s)\}$ be fixed but arbitrary. Then the embedding $ {}_\nu B_{q,r_0}^{4-s-4/r_0} \hookrightarrow B^1_{q,r_0}$ and \eqref{eq:final-local-interior} yield 
\begin{equation}\label{eq:rough-global-first-regularity}
 u\in H^{\theta,r}_{\mathrm{loc}}((0,\sigma); {}_\nu H^{4-s-4\theta,q}) \cap  C((0,\sigma); B^1_{q,r_0}) 
 \quad\text{a.s.,} \quad 0\le \theta<1/2, \quad  r\in[r_0,\infty).
\end{equation}
If $s\le2$,  then \eqref{eq:rough-global-H2q} follows immediately from
\eqref{eq:rough-global-first-regularity}. 

Suppose that $s>2$. We prove \eqref{eq:rough-global-H2q} using a stochastic maximal regularity argument, similar to Lemma \ref{lemma:parab reg variational}. 
Proposition~\ref{prop:rough-final-Lipschitz}, with
$a=(s-2)/4$, and Lemma~\ref{lemma:rough-final-noise-lipschitz}
give, 
\begin{equation}\label{eq:rough-global-Hminus2-forcing}
 \begin{aligned}
 \|F(z)\|_{ {}_\nu H^{-2,q}}&\lesssim 1+\|z\|_{H^{4-s,q}}^3,\\
 \|G(t,\omega,z)\|_{\gamma(\ell^2,L^q)}
     &\lesssim_T 1+\|z\|_{H^{4-s,q}}.
 \end{aligned}
\end{equation}
Here, we also used that $t_1,t_2 < 4-s$. 
We consider the scale 
$Z_0= {}_\nu H^{-2,q}$, $Z_1= {}_\nu  H^{2,q}$. Then 
\begin{equation}\label{eq:trace space for Z in rough data}
    Z_{1-\frac{1+\kappa_0}{r_0},r_0}=Z_{3/4,r_0}=B^1_{q,r_0} ,\qquad \kappa_0\coloneq \frac{r_0}{4}-1.
\end{equation}
Fix $0<\alpha<\beta<T<\infty$ 
and let $(\sigma_n)$ be a localizing sequence for $(u,\sigma \wedge T)$ with $\sigma_n<\sigma \wedge T$ (cf.\,\cite[Proposition~5.4]{AgVe25}). 
Let $V_n\coloneq \{\sigma_n>\alpha\}$, $\xi_n \coloneq \mathbf{1}_{V_n} u(\alpha)$, and on $(\alpha,T)$ set
\begin{equation}\label{eq:fn and gn for rough data}
f_n\coloneq\mathbf1_{V_n}\mathbf1_{(\alpha,\sigma_n]}F(u),\quad  g_n\coloneq\mathbf1_{V_n}\mathbf1_{(\alpha,\sigma_n ]}G(\cdot,u).
\end{equation}
By \eqref{eq:rough-global-first-regularity}, with $r=3r_0$ and $\theta=0$, and
\eqref{eq:rough-global-Hminus2-forcing}, almost surely,
\[
 f_n\in L^{r_0}(\alpha,T , w_\alpha^{\kappa_0}; Z_0),\qquad
 g_n\in L^{r_0}(\alpha,T ,w_\alpha^{\kappa_0}; \gamma(\ell^2,Z_{1/2})),
\]
where $w_\alpha^{\kappa_0}(t)=(t-\alpha)^{\kappa_0}$. Here we also used that the weight $w_\alpha$ is bounded. 

Note that \eqref{eq:rough-global-first-regularity} and \eqref{eq:trace space for Z in rough data} give $\xi_n \in L^0_{\mathcal F_\alpha}(\Omega; Z_{1-(1+\kappa_0)/r_0,r_0} )$. 
By Lemma~\ref{lemma:SMR}, the realization $A_{-2,q}$ of $A$ on $Z_0$ has $\mathcal{SMR}^\bullet_{r_0,\kappa_0}$. By \cite[Proposition~3.11]{AgVe25} there exists a unique solution to 
$$
\begin{cases}
  \mathrm{d}z_n + A_{-2,q}z_n\,\mathrm{d}t = f_n\,\mathrm{d}t + g_n\,\mathrm{d}W_{\ell^2}, \quad t \in (\alpha, T] \\
  z_n(\alpha) = \xi_n.
\end{cases}
$$
with 
\begin{equation}\label{eq:regularity for zn in rough data}
    z_n\in H^{\theta,r_0}(\alpha,T, w_\alpha^{\kappa_0}; {}_\nu H^{2-4\theta,q})\, \text{ a.s.,} \quad 0\le \theta<1/2.
\end{equation}
Note that the family of operators $(A_{s,q})$ is consistent and that 
$Z_0 \hookrightarrow X_0^{s,q}$,
$Z_1\hookrightarrow X_1^{s,q}$. Hence uniqueness for the linear
equation in $X_0^{s,q}$ gives
\begin{equation}\label{eq:zn=u in rough data}
     z_n=\mathbf1_{V_n} u\quad\text{on }[\alpha,\sigma_n ] \, \text{ a.s. on $V_n$.}
\end{equation}
Since $\beta>\alpha$, the weight $w_\alpha$ is bounded below on $[\beta,T]$. 
Hence, \eqref{eq:regularity for zn in rough data} and \eqref{eq:zn=u in rough data} give 
$$ u \in H^{\theta,r_0}(\beta,\sigma_n, w_\alpha^{\kappa_0}; {}_\nu H^{2-4\theta,q})\, \text{ a.s. on $\{\sigma_n>\beta\}$,} \quad 0\le \theta<1/2.$$
Letting $n\to \infty$, and recalling that $\beta>0$ was arbitrary, 
proves \eqref{eq:rough-global-H2q}. 

\medskip\noindent
\emph{Step 2: identification with the variational solution.}
Consider the scale $ Y_0= {}_\nu H^{-1}$, $Y_1= {}_\nu H^3$ from the variational intermediate setting. 
The Lipschitz estimates from that setting (see
Section~\ref{sec: LWP in intermediate}) and the embedding $ {}_\nu H^{2,q}\hookrightarrow H^{3/2}$ give
\begin{equation}\label{eq:rough-global-Y-forcing}
 \begin{aligned}
  \|F(z)\|_{Y_0}&\lesssim 1+\|z\|_{H^{3/2}}^3 \lesssim 1+\|z\|_{H^{2,q}}^3 ,
\\
 \|G(t,\omega,z)\|_{\gamma(\ell^2,Y_{1/2})}
 &\lesssim_T1+\|z\|_{H^{3/2}} \lesssim 1+\|z\|_{H^{2,q}}.
 \end{aligned}
\end{equation}
Let $\alpha>0$. 
By \eqref{eq:rough-global-Y-forcing} and \eqref{eq:rough-global-H2q}, the nonlinearities $f_n$ and $g_n$ as defined in \eqref{eq:fn and gn for rough data}, satisfy
\[
 f_n \in L^2(\alpha,T;Y_0),\qquad
 g_n \in L^2
       (\alpha,T;\gamma(\ell^2,Y_{1/2})), \, \text{ a.s.}
\]
Moreover, $\xi_n \coloneq \mathbf{1}_{V_n} u(\alpha) \in L^0_{\mathcal F_\alpha}(\Omega;H^1)$. Indeed, this follows from \eqref{eq:final-local-interior} and the embedding $B_{q,r}^{4-s-4/r} \hookrightarrow H^1$, which holds for any $r>\max \{2,4/(3-s)\}$. 
Repeating the stochastic maximal regularity argument of
Step 1 in the couple $(Y_0,Y_1)$ yields 
\begin{equation}\label{eq:rough-global-local-variational}
 u\in L^2_{\mathrm{loc}}((0,\sigma); {}_\nu H^3)
       \cap C((0,\sigma);H^1)\quad\text{a.s.}
\end{equation}
Here we also used the embeddings
$  Y_0\hookrightarrow X_0^{s,q} $ and 
 $Y_1\hookrightarrow X_1^{s,q}$, which hold since $s>3-\frac dq>1+\frac d2-\frac dq
$. 

On the other hand, by
Theorem \ref{theorem:GWP_intermediate},  there is a unique global variational solution
$v_\alpha$ with initial data $\mathbf{1}_{\{\sigma>\alpha\}} u(\alpha)$ at time $\alpha$.
By \eqref{eq:rough-global-local-variational}, the restriction of $u$ to $[\alpha,\sigma)\times \{\sigma>\alpha\}$ is a local
variational solution with the same initial value. 
Uniqueness in the intermediate setting implies
\begin{equation}\label{eq:rough-global-identification}
 u=v_\alpha \quad\text{on }[\alpha,\sigma)\, \text{ a.s. on } \{\sigma>\alpha\}.
\end{equation}

\medskip\noindent
\emph{Step 3: global well-posedness and regularity.} 
Fix $0<\alpha<\beta<T<\infty$. 
By the blow-up criterion of 
\cite[Theorem~4.7]{AGV26shift} with 
$
\gamma=4-s-\frac{4\kappa}{p}
$, 
\begin{equation}\label{eq:rough-global-critical-blowup}
 \mathbb P\Bigl(
 \beta<\sigma<T,\ 
 \sup_{t \in {[\beta,\sigma)}}\|u(t)\|_{B^{d/q-1}_{q,p}}
 +\|u\|_{L^p(\beta,\sigma; {}_\nu H^{\gamma,q})}<\infty
 \Bigr)=0.
\end{equation}
Indeed, this follows from  $X_{1-(1+\kappa)/p,p}=B^{d/q-1}_{q,p}$ and $X_{1-\kappa/p}= {}_\nu H^{\gamma,q}$. 
Note $\gamma<4-s<3$, since $s>1$.
Therefore, $H^{3,q} \hookrightarrow H^{\gamma,q}$ and thus 
\eqref{eq:parabolic reg of intermediate sol} gives
\[
 v_\alpha\in L^p(\beta,T; {}_\nu H^{\gamma,q})\quad\text{a.s.}
\]
Sobolev's embedding $
 H^1\hookrightarrow B^{d/q-1}_{q,p},
 $
\eqref{eq:rough-global-identification}  and the global $H^1$ continuity of $v_\alpha$ imply that, on $\{\beta<\sigma<T\}$,
\[
 \begin{aligned}
 &\sup_{t\in[\beta,\sigma)}\|u(t)\|_{B^{d/q-1}_{q,p}}
       +\|u\|_{L^p(\beta,\sigma;H^{\gamma,q})}
\lesssim
 \sup_{t\in [\beta, T]}\|v_\alpha(t)\|_{H^1}
       +\|v_\alpha\|_{L^p(\beta,T;H^{\gamma,q})}<\infty
 \quad\text{a.s.}
 \end{aligned}
\]
 This, together with \eqref{eq:rough-global-critical-blowup}, yields
$\mathbb P(\beta<\sigma<T)=0$. Since $\beta$ and $T$ were arbitrary, 
$\sigma=\infty$ almost surely.

Uniqueness and the regularity \eqref{eq:rough-global-critical-regularity} follow directly from 
Proposition \ref{thm:rough-final-local}.
The parabolic regularization \eqref{eq:rough-global-smoothing} follows from \eqref{eq:rough-global-identification}  and \eqref{eq:parabolic reg of intermediate sol}. 

\end{proof}

We now present the proof for the endpoint case $p=q=2$. 
\begin{proof}[Proof of Theorem \ref{thm:rough-final-global} with assumption \ref{item: I2}] Let $(u,\sigma)$ be the unique maximal solution of
Proposition~\ref{thm:rough-final-local}. 
In this case,
\[
 X_0^{s,2}= {}_\nu H^{d/2-3},\qquad X_1^{s,2}= {}_\nu H^{1+d/2},\qquad
 X_{1/2}^{s,2}=H^{d/2-1}.
\]
Since $u\in L^2_{\mathrm{loc}}([0,\sigma);X_1^{s,2})$ and
$X_1^{s,2}\hookrightarrow H^1$, Fubini's theorem gives
\[
 \xi_\alpha\coloneq \mathbf1_{\{\sigma>\alpha\}}u(\alpha)
          \in L^0_{\mathcal F_{\alpha}}(\Omega;H^1), \quad \text{for a.a. $\alpha>0$}
\]
Let $Y_0= {}_\nu H^{-1}$ and $Y_1= {}_\nu H^{3}$. 
Let $v_\alpha$  be the global variational solution with initial data $\mathbf{1}_{\{\sigma>\alpha\}} u(\alpha)$ at time $\alpha$, given by Theorem \ref{theorem:GWP_intermediate}. For every $T>\alpha$, almost surely,
\[
 v_\alpha\in C([\alpha,T];H^1)\cap L^2(\alpha,T; {}_\nu H^3).
\]
Note 
$
 Y_0\hookrightarrow X_0^{s,2},$ 
 $Y_{1/2} \hookrightarrow X_{1/2}^{s,2}$ and 
 $Y_1 \hookrightarrow X_1^{s,2}
 $. Moreover, $
 \|z\|_{H^{3/2}}
       \lesssim\|z\|_{H^1}^{3/4}\|z\|_{H^3}^{1/4}
$       
and thus  $v_\alpha \in L^8(\alpha,T;H^{3/2})$. The above, together with \eqref{eq:rough-global-Y-forcing}, 
imply   that $F(v_\alpha)\in L^2(\alpha,T;X_0^{s,2})$ and
$G(\cdot,v_\alpha)\in L^2(\alpha,T;\gamma(\ell^2,X_{1/2}^{s,2}))$ almost surely.
Hence $v_\alpha$ is also a local solution in the $(X_0^{s,2},X_1^{s,2})$-scale, with initial value $\mathbf{1}_{\{\sigma>\alpha\}} u(\alpha)$ at time $\alpha$. 
Uniqueness from Proposition~\ref{thm:rough-final-local} implies 
\begin{equation}\label{va:u=va rough p=q=2}
    u=v_\alpha \quad\text{on }[\alpha,\sigma)
       \quad\text{a.s. on }\{\sigma>\alpha\}.
\end{equation}
Hence, by the blow-up criterion of \cite[Proposition 6.2]{BGV26shift} we conclude that $\sigma=\infty$ a.s.

Uniqueness and the regularity \eqref{eq:rough-global-critical-regularity} follow directly from 
Proposition \ref{thm:rough-final-local}.
The parabolic regularization \eqref{eq:rough-global-smoothing} follows from \eqref{va:u=va rough p=q=2} and \eqref{eq:parabolic reg of intermediate sol}. 

\end{proof}

\begin{remark}[The role of the positive spatial shift $a>0$] \label{remark:role of shift a}
Suppose that $p>2$ and that we apply the unshifted local theory of \cite{AgVe25},
estimating the nonlinearities in $X_0^{s,q}$. To obtain the trace space
$B^{d/q-1}_{q,p}$, we must choose
\[
\mu\coloneq\frac{1+\kappa}{p}
=\frac{5-s-d/q}{4}.
\]
The restriction $\mu<1/2$ then requires
\[
s>3-\frac dq.
\]
Let $t_2$ be as in \eqref{eq:final-spatial-exponents} and let $\beta_2\coloneq (s+t_2)/4$. The Lipschitz estimate \eqref{eq:final-F2} for $F_2 \colon X_{\beta_2}^{s,q} \to X_0^{s,q}$ holds with  growth exponent $\rho_2=1$. The unshifted subcriticality condition
becomes
\[
\mu\le2(1-\beta_2)
\quad\Longleftrightarrow\quad
s+2t_2\le3+\frac dq.
\]
Since $t_2\ge1/2$, this implies $s\le2+d/q$. Consequently, the
parameters can be admissible only if
\[
3-\frac dq<2+\frac dq
\quad\Longleftrightarrow\quad q<2d.
\]
In particular, this also forces $d>1$. 
 Thus, for $q\ge2d$, the
precession term is supercritical for every admissible $s$. The positive shift $a>0$ removes this obstruction:
$F_2$ is strictly subcritical in the shifted theory; see
\eqref{eq:final-F2-subcriticality}.
\end{remark}

\begingroup
\renewcommand{\addcontentsline}[3]{}
\section*{Acknowledgements}
\endgroup
The authors wish to thank Antonio Agresti and Mark Veraar for organizing the Oberwolfach Seminar “Stochastic Partial Differential Equations in Critical Spaces” from which this project
originated based on their proposal.

\newpage
\bibliographystyle{alpha}
\bibliography{references}

\end{document}